\documentclass[11pt]{amsart}
\usepackage{amsmath,amsfonts,amssymb,mathrsfs}

\usepackage[margin=1in]{geometry}
\usepackage{url}
\usepackage[shortlabels]{enumitem}

\usepackage[unicode,bookmarks]{hyperref}
\usepackage[dvipsnames]{xcolor}
\hypersetup{colorlinks=true,citecolor=NavyBlue,linkcolor=Brown,urlcolor=Orange}

\usepackage{tikz-cd,adjustbox}
\usepackage[arrow,curve,matrix]{xy}

\newcommand{\tensor}{\otimes}
\newcommand{\isom}{\simeq}

\newcommand{\C}{\mathbb{C}}

\newcommand{\Q}{\mathbb{Q}}
\newcommand{\Z}{\mathbb{Z}}
\newcommand{\A}{\mathbb{A}}
\renewcommand{\P}{\mathbb{P}}

\newcommand{\M}{\mathcal{M}}

\newcommand{\D}{\mathcal{D}}

\newcommand{\G}{\mathcal{G}}
\newcommand{\HH}{\mathbb{H}}
\newcommand{\I}{\mathcal{I}}
\renewcommand{\O}{\mathcal{O}}
\newcommand{\DR}{\mathrm{DR}}

\renewcommand{\L}{\mathcal{L}}
\newcommand{\codim}{\mathrm{codim}}
\newcommand{\w}{\omega}

\newcommand{\K}{\mathcal{K}}
\newcommand{\gr}{\mathrm{Gr}}
\newcommand{\Hom}{\mathcal{H}om}
\newcommand{\h}{\underline h}

\renewcommand{\sp}{\mathrm{sp}}

\newcommand{\DB}{\underline{\Omega}} 
\newcommand{\Sing}{\mathrm{Sing}}
\newcommand{\dual}{\mathbf{D}}
\newcommand{\Supp}{\mathrm{Supp}}
\newcommand{\Koz}{\textrm{Koz}}

\newcommand{\sm}{\smallsetminus}

\newcommand{\Pic}{\mathrm{Pic}}
\newcommand{\IC}{\mathrm{IC}}
\newcommand{\wt}{\widetilde}

\newcommand{\IH}{\mathit{IH}}
\newcommand{\X}{\mathcal{X}}

\numberwithin{equation}{section}

\newtheorem{thm}[equation]{Theorem}
\newtheorem{cor}[equation]{Corollary}
\newtheorem{lem}[equation]{Lemma}
\newtheorem{prop}[equation]{Proposition}

\theoremstyle{definition}
\newtheorem{defn}[equation]{Definition}

\theoremstyle{remark}
\newtheorem{rmk}[equation]{Remark}

\theoremstyle{plain}

\begin{document}

\title[Indeterminacy of period map and Griffiths heights for hypersurfaces]{Indeterminacy of period map and Griffiths heights for hypersurfaces}

\author{Sung Gi Park}
\address{Department of Mathematics, Princeton University, and School of Mathematics, Institute for Advanced Study, Princeton, NJ, USA}
\email{sp6631@princeton.edu \,\,\,\,\,\,sgpark@ias.edu}



\thanks{}


\date{\today}

\begin{abstract}
We prove that the period map $\overline {\mathcal P}$, from the GIT moduli space of degree $d$ hypersurface $X$ in $\mathbb P^n$ to the Baily--Borel compactification of the period image associated to the middle cohomology $H^{n-1}_{\rm prim}(X,\mathbb Z)$, extends regularly to $X$ if and only if $X$ has at worst nodal singularities (resp. suspensions of slc surface singularities) when $n$ is even (resp. $n$ is odd). This verifies and generalizes Laza--O'Grady's prediction on the indeterminacy locus of the period map for quartic K3s to arbitrary hypersurfaces. For Calabi--Yau hypersurfaces, we prove that the birational map from the GIT moduli space to the BBH compactification extends regularly to $X$ if and only if $X$ has semi-log canonical singularities. The results apply to distinguish the Baily--Borel and BBH compactifications. In the process, we answer two questions of Mordant on the Griffiths height for families of hypersurfaces over curves, which are crucial to the proofs of the main results.
\end{abstract}

\maketitle

\makeatletter
\newcommand\@dotsep{4.5}
\def\@tocline#1#2#3#4#5#6#7{\relax
  \ifnum #1>\c@tocdepth 
  \else
    \par \addpenalty\@secpenalty\addvspace{#2}%
    \begingroup \hyphenpenalty\@M
    \@ifempty{#4}{%
      \@tempdima\csname r@tocindent\number#1\endcsname\relax
    }{%
      \@tempdima#4\relax
    }%
    \parindent\z@ \leftskip#3\relax
    \advance\leftskip\@tempdima\relax
    \rightskip\@pnumwidth plus1em \parfillskip-\@pnumwidth
    #5\leavevmode\hskip-\@tempdima #6\relax
    \leaders\hbox{$\m@th
      \mkern \@dotsep mu\hbox{.}\mkern \@dotsep mu$}\hfill
    \hbox to\@pnumwidth{\@tocpagenum{#7}}\par
    \nobreak
    \endgroup
  \fi}
\def\l@section{\@tocline{1}{0pt}{1pc}{}{\bfseries}}
\def\l@subsection{\@tocline{2}{0pt}{25pt}{5pc}{}}
\makeatother


\setcounter{tocdepth}{1}
\tableofcontents

\section{Introduction}\label{scn:intro}

The study of Hodge structures and period maps of hypersurfaces has played a central role in understanding their geometry and moduli. Denote by $\Gamma\backslash \D_{(n,d)}$ the period space associated to the primitive $\Z$-Hodge structure of the middle cohomology for smooth degree $d$ hypersurfaces in $\P^n$. Then the period map
$$
\Phi:\P|\O_{\P^n}(d)|\dashrightarrow\Gamma\backslash \D_{(n,d)},
$$
from the projective parameter space of hypersurfaces, descends to a map from the GIT moduli space
$$
\mathcal P: \overline\M^{\rm GIT}_{(n,d)} \dashrightarrow \Gamma\backslash \D_{(n,d)}.
$$
A priori, both period maps are defined over the locus of smooth hypersurfaces. The geometry of these maps has been studied extensively in the literature from many different perspectives. Griffiths proved the local Torelli theorem in \cite[Theorem 9.8]{Griffiths69}: the period map on the deformation space for a smooth hypersurface is a local immersion for $n\ge 2$, $d\ge 3$ and $(n,d)\neq(3,3)$. Donagi proved the generic Torelli theorem in \cite{Donagi83}, which was later improved by a number of papers, including \cite{CoxGreen, Voisin22}: the map $\mathcal P$ is birational onto its image for all but finitely many pairs $(n,d)$.

In several classical cases, the period space (associated to some cyclic cover if necessary) admits a natural arithmetic compactification, and the birational geometry of the compactified period map has been studied for decades; see, for instance, \cite{Shah80, Shah81, ACT02, LS07, Artebani09, Laza09, Looijenga09, Laza10, ACT11, WILDBAB, BL24}. The common theme in the literature is to describe the resolution of the period map and find a moduli-theoretic interpretation. As a first step, it is necessary to understand the indeterminacy locus of the period map.

For example, for quartic K3 surfaces, Ascher--DeVleming--Liu \cite{ADL23} gave an explicit interpolation between the GIT moduli space and the Baily--Borel compactification of the period space through wall-crossings in K-moduli, thereby verifying a prediction of Laza--O'Grady \cite{LOG18, LOG19}. More precisely, they resolved the indeterminacy of the period map by taking a blow-up and then a sequence of flips arising from wall-crossings for K-moduli of pairs. The locus of indeterminacy had been predicted by Laza and O'Grady: it is precisely the locus of quartic surfaces with singularities worse than semi-log canonical, and this is exactly the locus replaced in the wall-crossing construction of \cite{ADL23}.

In this paper, we determine the indeterminacy locus of the period map from the GIT moduli space of hypersurfaces to the Baily--Borel compactification of the period image constructed by Bakker, Filipazzi, Mauri, and Tsimerman \cite{BFMT25}. Our result holds in every dimension and degree, and generalizes the Laza--O'Grady prediction for quartic K3 surfaces.

\begin{thm}\label{thm:indeterminacy}
Let $(n,d)\neq(3,3)$, where $n\geq 2$ and $d\geq 3$. Let
$$
\overline {\mathcal P}:\overline\M^{\rm GIT}_{(n,d)} \dashrightarrow \overline \M^{\rm BB}_{(n,d)}
$$
be the period map from the GIT moduli space of degree $d$ hypersurfaces in $\mathbb P^n$ to the Baily--Borel compactification of the period image. Then the domain of definition of $\overline{\mathcal P}$ is exactly the following locus:
\begin{enumerate}
\item if $n$ is even, the locus of hypersurfaces with nodal singularities;
\item if $n$ is odd, the locus of hypersurfaces whose singularities are suspensions of semi-log canonical surface singularities.
\end{enumerate}
\end{thm}

By the generic Torelli theorem for hypersurfaces, which is expected to hold for all pairs $(n,d)$ in our range and is known in all but finitely many cases, the period image is birational to the GIT moduli space. Assuming generic Torelli, the Baily--Borel compactification $\overline{\M}^{\rm BB}_{(n,d)}$ constructed in \cite[Theorem 1.1]{BFMT25} provides an alternative compactification of an open moduli space of hypersurfaces.

Here, a nodal singularity means an ordinary double point. A suspension of a semi-log canonical surface singularity is a singularity locally analytically equivalent to a hypersurface
$$
x_1^2+\dots+x_{n-3}^2+f(x_{n-2},x_{n-1},x_n)=0
$$
in a germ $(\C^n,0)$, where the surface germ $\{f=0\}\subset (\C^3,0)$ is semi-log canonical. The domain of definition can be alternatively phrased in terms of higher Du Bois singularities. We prove in Lemma \ref{lem:classification of n/2 Du Bois} that a hypersurface $X\subset \P^n$ has nodal singularities (resp. suspensions of semi-log canonical surface singularities) if and only if its minimal exponent satisfies $\wt\alpha(X)\ge\frac{n}{2}$ when $n$ is even (resp. $\wt\alpha(X)\ge\frac{n-1}{2}$ when $n$ is odd). In particular, the domain of definition is the locus
$$
\left\{X\subset \mathbb P^n \mid X \text{ has $\left(\lfloor\frac{n}{2}\rfloor-1\right)$-Du Bois singularities}\right\}.
$$

In the Calabi--Yau case, that is $d=n+1$, the moduli space $\M^{\rm can}_{(n,n+1)}$ of Calabi--Yau hypersurfaces with canonical singularities admits a Hodge-theoretic compactification $\overline \M^{\rm BBH}_{(n,n+1)}$, obtained as the ample model for the Hodge line bundle in \cite[Corollary 1.8]{BFMT25}. On the other hand, Calabi--Yau hypersurfaces with canonical singularities are GIT stable \cite{Tian94}, and thus, $\M^{\rm can}_{(n,n+1)}$ is an open subset of the GIT moduli space $\overline\M^{\rm GIT}_{(n,n+1)}$. There exists a birational map
$$
\overline {\mathcal P}^{\rm BBH}:\overline\M^{\rm GIT}_{(n,n+1)} \dashrightarrow \overline \M^{\rm BBH}_{(n,n+1)},
$$
which factors through the Stein factorization $\overline\M^{\rm GIT}_{(n,n+1)}\dashrightarrow \left(\overline \M^{\rm BB}_{(n,n+1)}\right)^\nu$ of the period map $\overline {\mathcal P}$; see \cite[Remark 5.4]{BFMT25}. Note that
$$
\pi_{CY}:\left(\overline \M^{\rm BB}_{(n,n+1)}\right)^\nu\to \overline \M^{\rm BBH}_{(n,n+1)}
$$
is a morphism that can contract curves in the interior of the period domain; see \emph{loc. cit.} for details. We determine the domain of definition for $\overline {\mathcal P}^{\rm BBH}$ and prove that $\pi_{CY}$ is a divisorial contraction over $\M^{\rm can}_{(n,n+1)}$, which always contracts some divisor in $\left(\overline \M^{\rm BB}_{(n,n+1)}\right)^\nu$ that maps to the interior of the period domain for all $n\ge 4$.

\begin{thm}\label{thm:BBH indeterminacy}
Let $d=n+1$, where $n\geq 2$. Then the domain of definition of $\overline{\mathcal P}^{\rm BBH}$ is exactly the locus of hypersurfaces with semi-log canonical singularities.

Furthermore, for every $n\ge 4$, the morphism $\pi_{CY}$ is a divisorial contraction over the moduli space $\M^{\rm can}_{(n,n+1)}$ of Calabi--Yau hypersurfaces with canonical singularities, which contracts some exceptional divisor that generically maps to the interior of the period space.
\end{thm}

In particular, the last statement extends the discussion in \cite[7.3.2]{BFMT25} to arbitrary Calabi--Yau hypersurfaces.

For a hypersurface singularity, semi-log canonicity is equivalent to the condition that the log canonical threshold of its defining equation is $1$. As a consequence, the above theorem solves \cite[Question 8.4]{Park25} when $d=n+1$.

By Theorem \ref{thm:indeterminacy}, we obtain the domain of definition of the period map $\mathcal P$ by restricting the target to the period space $\Gamma\backslash \D_{(n,d)}$, from which we deduce a Torelli-type theorem for mildly singular Calabi--Yau hypersurfaces:

\begin{cor}\label{cor:indeterminacy}
Let $(n,d)\neq(3,3)$, where $n\geq 2$ and $d\geq 3$. Then the domain of definition of the period map ${\mathcal P}: \overline\M^{\rm GIT}_{(n,d)} \dashrightarrow \Gamma\backslash \D_{(n,d)}$ is exactly the following locus:
\begin{enumerate}
\item if $n$ is even, the locus of smooth hypersurfaces;
\item if $n$ is odd, the locus of hypersurfaces with simple ADE singularities.
\end{enumerate}
Furthermore, for Calabi--Yau hypersurfaces (i.e. $d=n+1$), the Stein factorization of the period map $\mathcal P$ to its period image is an open embedding on the domain of definition.
\end{cor}

Here, a simple ADE singularity is an isolated singularity locally analytically equivalent to a suspension of an ADE surface singularity. Note that (1) is proven in \cite[Corollary 8.1(1)]{Park25}, which is well known to experts. Additionally, the local Torelli theorem for smooth hypersurfaces implies that the Stein factorization of the period map is an open embedding on the smooth locus, for $n\ge 2$, $d\ge 3$ and $(n,d)\neq(3,3)$.

A key ingredient in determining the indeterminacy locus of the period map is provided by affirmative answers to the questions posed by Mordant \cite{Mordant25}. In \emph{loc. cit.}, Mordant studies families of hypersurfaces over curves and compares two natural height functions: the Griffiths height, defined as the degree of the Griffiths bundle, and the GIT height, defined as the degree of the GIT polarization pulled back from the GIT moduli space.

\smallskip
\noindent
{\bf Definition.} Let $H_\eta\subset \P^n_{C_\eta}$ be a smooth hypersurface of degree $d$ over the generic point $C_\eta$ of a smooth projective curve $C$. Consider a generically defined polarized $\Z$-Hodge structure $\mathbb H^{n-1}(H_\eta/C_\eta)$ on an open subset of $C$. The \emph{Griffiths height} $\mathrm{ht}_{GK}(\mathbb H^{n-1}(H_\eta/C_\eta))$ is the degree of the Griffiths bundle associated to $\mathbb H^{n-1}(H_\eta/C_\eta)$. Consider the GIT moduli morphism $C\to \overline \M^{\rm GIT}_{(n,d)}$. The \emph{GIT height} $\mathrm{ht}_{GIT}(H_\eta)$ is the degree of the GIT polarization $\O_{\overline \M^{\rm GIT}_{(n,d)}}(1)$ on $C$.

\smallskip

See Section \ref{sec:Griffiths and GIT height} for precise definitions.

For a vector bundle $E$ of rank $n+1$ on $C$ and a family of hypersurfaces $H$ over $C$ defined as a smooth divisor in a projective bundle $\P_C(E)$, Mordant \cite[Theorem 1.1]{Mordant25} computes the Griffiths height and the GIT height, and proves the equality
$$
\mathrm{ht}_{GK}(\mathbb H^{n-1}(H_\eta/C_\eta))=F(n,d)\,\mathrm{ht}_{GIT}(H_\eta)
$$
when every fiber of $H\to C$ has at worst nodal singularities. The constant $F(n,d)$ is defined by
\begin{equation}
\label{eqn:constant F}
F(n,d)=\begin{cases}
\frac{n+1}{24d^2}\left[(d-1)^n(d^2n-d^2-2dn-2)+2(d^2-1)\right] & \text{when}\; n \text{ is odd}\\
\frac{n+1}{24d^2}\left[(d-1)^n(d^2n+2d^2-2dn-2)-2(d^2-1)\right] & \text{when}\; n \text{ is even.}
\end{cases}
\end{equation}
We prove inequalities that attain equality as above when the period map is defined on the image of $C$ in the GIT moduli space, thereby answering questions in \cite[Section 5.3]{Mordant25}.

\begin{thm}\label{thm:Mordant question and indeterminacy}
Let $(n,d)\neq(3,3)$, where $n\geq 2$ and $d\geq 3$. Let $H_\eta\subset \P^n_{C_\eta}$ be a smooth hypersurface of degree $d$ over the generic point $C_\eta$ of a smooth projective curve $C$. Then we have
$$
\mathrm{ht}_{GK}(\mathbb H^{n-1}(H_\eta/C_\eta))\le F(n,d)\,\mathrm{ht}_{GIT}(H_\eta),
$$
and the equality holds if and only if the image of the moduli morphism $C\to\overline\M^{\rm GIT}_{(n,d)}$ is contained in the domain of definition of $\overline{\mathcal P}$.

Furthermore, there exists a positive constant $\epsilon(n,d)$ such that the converse inequality holds:
$$
\mathrm{ht}_{GK}(\mathbb H^{n-1}(H_\eta/C_\eta))\ge \epsilon(n,d)\,\mathrm{ht}_{GIT}(H_\eta).
$$
\end{thm}

It is an interesting question to find an optimal constant $\epsilon(n,d)$ for the converse inequality.

Throughout the text, a variety is a reduced separated scheme of finite type over $\C$. We use superscripts for a decreasing filtration $F^\bullet$ and subscripts for an increasing filtration $F_\bullet$, related by $F^p=F_{-p}$.

\smallskip
\noindent
{\bf Acknowledgements.} I thank Kenny Ascher, Ben Bakker, Brad Dirks, Hyunsuk Kim, János Kollár, Radu Laza, Yongnam Lee, Mircea Musta\c t\u a, Mihnea Popa, Duc Vo, Chenyang Xu, and Ruijie Yang for valuable discussions. Part of this work was completed during my visit to KIAS as a June E Huh Visiting Fellow. This research was supported by the Oswald Veblen Fund at IAS. AI tools were used only to improve the writing and to search for references.

\section{Preliminaries}
\label{sec:preliminaries}

\subsection{Deligne's canonical extension and Griffiths bundle}
\label{sec:Deligne's canonical extension and Griffiths}

Let $S$ be a smooth variety with a polarized variation of $\Z$-Hodge structures $\mathbb V=(\mathbb V_\Z,F^\bullet, Q)$ of weight $w$. Here, $\mathbb V_\Z$ is a $\Z$-local system on $S$, $F^\bullet$ is a (decreasing) Hodge filtration on $\mathcal V:=\mathbb V_\Z\otimes_\Z\O_S$, and $V_\Z$ denotes a fixed lattice underlying $\mathbb V_\Z$ equipped with the polarization $Q$. For convenience, we assume that the Hodge filtration is effective, namely $\gr_F^{p}\mathcal V=0$ for $p<0$.

Let $\overline S$ be a log smooth compactification of $S$, i.e., a smooth proper variety containing $S$ as an open subvariety such that $D:=\overline S\sm S$ is a simple normal crossing divisor. Then the flat bundle $(\mathcal V,\nabla)$ admits Deligne's lower canonical extension to a vector bundle $\overline{\mathcal V}^{[0,1)}$ on $\overline S$, equipped with a logarithmic connection
$$
\overline \nabla:\overline {\mathcal V}^{[0,1)}\to\overline {\mathcal V}^{[0,1)}\otimes_{\O_{\overline S}}\Omega_{\overline S}(\log D) 
$$
whose residues along $D$ have eigenvalues in $[0,1)$ and whose restriction to $S$ is $(\mathcal V,\nabla)$. We omit the superscript $[0,1)$ and write $\overline {\mathcal V}$ if the local monodromy along $D$ is unipotent. See \cite[Chapter II.5]{Deligne70} or \cite[Theorem 5.2.17]{HTT} for details.

By Schmid \cite{Schmid73}, the Hodge filtration $F^\bullet\mathcal V$ extends to a filtration by subbundles $F^\bullet\overline{\mathcal V}^{[0,1)}$ of Deligne's lower canonical extension $\overline{\mathcal V}^{[0,1)}$. Following \cite[(7.13)]{Griffiths70}, we define the lower Griffiths bundle.

\begin{defn}
\label{defn:lower Griffiths bundle}
In the above setting, we define the \emph{lower Griffiths bundle} on $\overline S$ by
$$
\G\left(\overline{\mathcal V}^{[0,1)},F^\bullet\right):=\bigotimes_{p=1}^w \det\left(F^p\overline{\mathcal V}^{[0,1)}\right)\isom\bigotimes_{p=1}^w \det\left(\gr_F^p\overline{\mathcal V}^{[0,1)}\right)^{\otimes p}.
$$
When the variation of Hodge structures is clear from the context, we simply denote this line bundle by $\G^{[0,1)}_{\overline S}$. If the local monodromy of $\mathbb V_\Z$ along $D$ is unipotent, we call it the \emph{Griffiths bundle} on $\overline S$ and denote it by $\G_{\overline S}$.
\end{defn}

By Kawamata's covering trick \cite[Theorem 17]{Kawamata81}, there exists a smooth variety $\overline T$ and a finite flat morphism $\pi:\overline T\to \overline S$ such that $\pi^{-1}(D)$ is a simple normal crossing divisor and the local monodromy of $\pi^*\mathbb V_\Z$ is unipotent. Due to the canonicity of Deligne's extension, the Griffiths bundle $\G_{\overline T}$ descends to a class
$$
\G_{\overline S}\in \mathrm{Pic}(\overline S)\otimes_\Z \Q,
$$
which we call the \emph{Griffiths $\Q$-bundle} (or simply, the \emph{Griffiths bundle}) on $\overline S$. This class is independent of the choice of the covering.

For later use, we record a duality between the graded Hodge pieces of Deligne's canonical extension induced by the polarization.

\begin{lem}
\label{lem:Griffiths bundle alternative}
Let $S$ be a smooth variety and $\mathbb V=(\mathbb V_\Z,F^\bullet, Q)$ be a polarized variation of Hodge structures of weight $w$. Assume $\gr^p_F\mathcal V=0$ for $p<0$ and the local monodromy of $\mathbb V_\Z$ along $D:=\overline S\sm S$ is unipotent. Then we have
$$
\Hom_{\O_{\overline S}}(\overline {\mathcal V},\O_{\overline S})\isom \overline{\mathcal V},\quad \Hom_{\O_{\overline S}}(\gr_F^p\overline {\mathcal V},\O_{\overline S})\isom \gr_F^{w-p}\,\overline{\mathcal V}
$$
and
$$
\G_{\overline S}\isom_\Q \bigotimes_{p=0}^{\lfloor\frac{w-1}{2}\rfloor}\det\left(\gr_F^p\overline {\mathcal V}\right)^{\otimes -w+2p}.
$$
\end{lem}

\begin{proof}
The polarization
$$
Q:\mathbb V_\Q\otimes_\Q\mathbb V_\Q\to \Q_S
$$
is a perfect pairing, hence induces an isomorphism
$$
\Hom_{\O_S}(\mathcal V,\O_S)\isom\mathcal V
$$
of vector bundles with flat connections.

In suitable local coordinates, suppose the local defining equation of $D$ near a point $s\in D$ is $x_1\cdots x_d=0$. Assume that the rank of $\mathbb V$ is $r$. From the local description of Deligne's canonical extension (for instance, see the proof of \cite[Theorem 5.2.17]{HTT}), there exist $r$ by $r$ nilpotent complex matrices $A^k=\{A^k_{ij}\}$ for all $1\le k\le d$ satisfying
$$
\overline \nabla e_j=\sum_{1\le i\le r, 1\le k\le d}A_{ij}^ke_i\otimes \frac{dx_k}{x_k},
$$
for some basis $\{e_1,\dots,e_r\}$ of $\overline{\mathcal V}$ near $s\in D$. The local monodromy matrix along $x_k=0$ is $\exp(-2\pi\sqrt{-1}A^k)$.

Recall that the polarization $Q=\langle\;,\;\rangle:\mathcal V\otimes\Hom_{\O_S}(\mathcal V,\O_S)\to \O_S$ is compatible with flat connections. For the dual basis $\{e_1^*,\dots,e_r^*\}$ of $\Hom_{\overline S}(\overline{\mathcal V},\O_{\overline S})$, we have
$$
d\langle e_i, e_j^*\rangle=\langle \overline \nabla e_i, e_j^*\rangle+\langle e_i, \overline \nabla e_j^*\rangle=0,
$$
since $\langle e_i, e_j^*\rangle$ is identically $0$ or $1$. Hence, we have
$$
\overline \nabla e_j^*=\sum_{1\le i\le r, 1\le k\le d}-A_{ji}^ke_i\otimes \frac{dx_k}{x_k}.
$$
This implies that $\Hom_{\overline S}(\overline{\mathcal V},\O_{\overline S})$ is Deligne's canonical extension of $\Hom_{\O_S}(\mathcal V,\O_S)$, with local residue matrix $-(A^k)^T$ along $x_k=0$. Therefore, we have an isomorphism
$$
\Hom_{\O_{\overline S}}(\overline {\mathcal V},\O_{\overline S})\isom \overline{\mathcal V}.
$$

Note that the isomorphism $\Hom_{\O_S}(\mathcal V,\O_S)\isom\mathcal V$ induced by the polarization $Q$ is also compatible with the Hodge filtration
$$
\Hom_{\O_S}(\gr_F^p\mathcal V,\O_S)\isom\gr_F^{w-p}\mathcal V.
$$
Since the Hodge filtration extends uniquely as a filtration by subbundles to Deligne's canonical extension, the above isomorphism extends to
$$
\Hom_{\O_{\overline S}}(\gr_F^p\overline {\mathcal V},\O_{\overline S})\isom \gr_F^{w-p}\,\overline{\mathcal V}.
$$
As an immediate consequence, the last isomorphism for the Griffiths bundle $\G_{\overline S}$ follows.
\end{proof}

\subsection{Baily--Borel compactification of period image}
\label{sec:Baily--Borel}

As in the setting of Section \ref{sec:Deligne's canonical extension and Griffiths}, consider a smooth variety $S$ with a polarized variation of $\Z$-Hodge structures $(\mathbb V_\Z,F^\bullet, Q)$. Associated to these are a period domain $\D$ and an arithmetic subgroup $\Gamma\subset \mathrm{Aut}(V_\Z,Q)$ containing the monodromy group. The associated period map is
$$
\Phi:S\to \Gamma\backslash \D.
$$

Period maps arise naturally in a wide range of moduli problems, and understanding their geometry has been a central theme. This includes, for instance, local and global Torelli theorems, as well as the study of the birational geometry of period maps. Although the results of this section apply more generally, we restrict our attention to the case
$$
\Gamma=\mathrm{Aut}(V_{\mathbb Z},Q),
$$
so that $\Gamma\backslash \D$ is the classifying space of Hodge structures of the prescribed type.

After establishing the complex analyticity of the period map, Griffiths \cite[Conjecture 10.5]{Griffiths70Q} conjectured that the closure of its image is algebraic, or more precisely, quasi-projective. This conjecture was proved by Bakker, Brunebarbe, and Tsimerman \cite{BBT23}. More recently, Bakker, Filipazzi, Mauri, and Tsimerman \cite{BFMT25} constructed a Baily--Borel compactification of the period image.

\begin{thm}[{\cite[Theorems 1.1 and 1.2]{BFMT25}}]
\label{thm:BFMT BB}
The period image $\overline{\Phi(S)}$ admits a Baily--Borel compactification $\overline M^{\rm BB}$, with an ample $\Q$-line bundle $\O_{\overline M^{\rm BB}}(1)$. For any log smooth compactification $\overline S$ of $S$, the period map extends to a morphism
$$
\overline \Phi:\overline S\to \overline M^{\rm BB},
$$
and the pullback $\overline \Phi^*\O_{\overline M^{\rm BB}}(1)$ is isomorphic to the Griffiths bundle $\G_{\overline S}$ on $\overline S$.
\end{thm}

Here, a $\mathbb Q$-line bundle on $X$ means a class in $\Pic(X)\otimes_{\mathbb Z}\mathbb Q$. This compactification is functorial and satisfies an analogue of the Borel extension property (see \emph{loc. cit.} for further details).

For a family $f:\X\to S$ of Calabi--Yau varieties with canonical singularities polarized by a relatively ample line bundle $\L$ on $\X$, the Hodge bundle $\Lambda_S:=f_*\omega_{\X/S}$ on $S$ extends naturally to the Hodge bundle $\Lambda_{\overline S}$ on $\overline S$. This bundle can be used to construct a different compactification of a moduli space of Calabi--Yau varieties.

\begin{thm}[{\cite[Corollaries 1.8 and 7.9]{BFMT25}}]
\label{thm:BFMT BBH}
Let $M$ be the coarse moduli space of polarized Calabi--Yau varieties with canonical singularities, and let $M^\nu$ denote its normalization. Then $M^\nu$ admits a unique normal compactification $\overline M^{\rm BBH}$ with the following properties:
\begin{enumerate}
\item The Hodge bundle $\Lambda_{M^\nu}$ extends to an ample $\Q$-line bundle $\O_{\overline M^{\rm BBH}}(1)$ on $\overline M^{\rm BBH}$.
\item For any family $f:\X\to S$ of polarized Calabi--Yau varieties with canonical singularities and any log smooth compactification $\overline S$ of $S$, the moduli map $\Phi_{\rm CY}: S\to M$ extends to a morphism $\overline \Phi_{\rm CY}:\overline S\to \overline M^{\rm BBH}$ and the pullback $\overline \Phi_{\rm CY}^*\O_{\overline M^{\rm BBH}}(1)$ is isomorphic to the Hodge bundle $\Lambda_{\overline S}$ on $\overline S$.
\end{enumerate}
\end{thm}

Throughout this paper, we write $\overline \M^{\rm BB}_{(n,d)}$ for the Baily--Borel compactification of the period image of the moduli space of degree $d$ hypersurfaces in $\P^n$, and $\overline \M^{\rm BBH}_{(n,n+1)}$ for the BBH compactification in the Calabi--Yau case. We analyze the difference between the two compactifications $\overline \M^{\rm BB}_{(n,n+1)}$ and $\overline \M^{\rm BBH}_{(n,n+1)}$ by giving a complete description of the indeterminacy loci in Section~\ref{sec:BB versus BBH for Calabi--Yau}.

\subsection{Griffiths height and GIT height for hypersurfaces}
\label{sec:Griffiths and GIT height}

Returning to the setting of Section \ref{sec:Deligne's canonical extension and Griffiths}, when $S$ is a smooth curve, the \emph{Griffiths height}, introduced and studied by Mordant \cite{Mordant26, Mordant25}, is defined as the degree of the Griffiths bundle. Recall that a smooth curve admits a unique smooth projective compactification.

\begin{defn}
Let $C$ be a smooth projective curve, and let $\mathbb V_\eta$ be a polarized variation of $\Z$-Hodge structures defined on a nonempty Zariski open subset of $C$. We define its \emph{Griffiths height} and \emph{lower Griffiths height}, respectively, by
$$
\mathrm{ht}_{GK}(\mathbb V_\eta):=\deg_{C}\G_{C},\quad \mathrm{ht}_{GK,-}(\mathbb V_\eta):=\deg_{C}\G_{C}^{[0,1)}.
$$
\end{defn}

The definitions of the Griffiths bundle and the lower Griffiths bundle immediately imply
$$
\mathrm{ht}_{GK}(\mathbb V_\eta)\ge \mathrm{ht}_{GK,-}(\mathbb V_\eta).
$$
In Section \ref{sec:LMHS and Griffiths}, we determine the difference between these two heights explicitly in terms of the limiting mixed Hodge structures of $\mathbb V_\eta$ at the boundary points.

For example, given a family of polarized smooth projective varieties over a nonempty Zariski open subset of $C$, we have a polarized variation of Hodge structures induced by their primitive middle cohomology. Hence, we may consider the Griffiths height associated to this family.

For a family of hypersurfaces, we may consider the \emph{GIT height} induced by the GIT moduli space, which was also studied by Mordant \cite{Mordant26, Mordant25} in comparison with the Griffiths height.

\begin{defn}
Let $C_\eta$ be the generic point of a smooth projective curve $C$, and let $H_\eta\subset \P^n_{C_\eta}$ be a smooth hypersurface of degree $d\ge 3$ over $C_\eta$. We define its \emph{GIT height} by
$$
\mathrm{ht}_{GIT}(H_\eta):=\deg_C\overline\Phi_{\rm GIT}^*\O_{\overline \M^{\rm GIT}_{(n,d)}}(1),
$$
where $\overline\Phi_{\rm GIT}:C\to \overline \M^{\rm GIT}_{(n,d)}$ is the GIT moduli morphism and $\O_{\overline \M^{\rm GIT}_{(n,d)}}(1)$ is the GIT polarization.
\end{defn}

Recall that the GIT moduli space of degree $d$ hypersurfaces in $\P^n$ is constructed by the GIT quotient
$$
\overline \M^{\rm GIT}_{(n,d)}:=\P|\O_{\P^n}(d)|/\!/ PGL(n+1)
$$
and the polarization $\O_{\P|\O_{\P^n}(d)|}(1)$ descends to the GIT polarization $\O_{\overline \M^{\rm GIT}_{(n,d)}}(1)$ as a $\Q$-line bundle. Since a smooth hypersurface of degree $d\ge 3$ is GIT stable by \cite{Mumford94}, we have the GIT moduli morphism
$$
\Phi_{\rm GIT}:C_\eta\to \overline \M^{\rm GIT}_{(n,d)}
$$
that extends regularly to $C$ by the properness of the GIT moduli space.

Note that a smooth hypersurface $H_\eta\subset \P^n_{C_\eta}$ induces a polarized variation of Hodge structure $\mathbb H^{n-1}_{\rm prim}(H_\eta/C_\eta)$ on a nonempty Zariski open subset of $C$ associated to the primitive middle cohomology. Since $\mathbb H^{n-1}_{\rm prim}(H_\eta/C_\eta)$ and $\mathbb H^{n-1}(H_\eta/C_\eta)$ are either equal or differ by a rank $1$ local system of pure Hodge-Tate type, we have
$$
\mathrm{ht}_{GK}(\mathbb H^{n-1}_{\rm prim}(H_\eta/C_\eta))=\mathrm{ht}_{GK}(\mathbb H^{n-1}(H_\eta/C_\eta)).
$$

In \cite[Theorem 1.1]{Mordant25}, under additional assumptions on the family $H\to C$ of hypersurfaces, Mordant proved the equality
$$
\mathrm{ht}_{GK}(\mathbb H^{n-1}(H_\eta/C_\eta))=F(n,d)\,\mathrm{ht}_{GIT}(H_\eta)
$$
for the constant $F(n,d)$ described in \eqref{eqn:constant F}.

\subsection{Minimal exponents and higher Du Bois singularities}
\label{sec:minimal exponent and higher Du Bois}

Let $X$ be a complex variety. Deligne \cite{Deligne74} and Du Bois \cite{DB81} defined the associated filtered de Rham complex $(\DB_X^\bullet, F^\bullet)$ quasi-isomorphic to the constant sheaf $\C_X$. When $X$ is proper, this complex defines the Hodge filtration of the mixed Hodge structure on the singular cohomology $H^\bullet(X,\Q)$ of $X$. The \emph{$p$-th Du Bois complex} $\DB^p_X$ of $X$ is defined as the shifted graded piece
$$
\DB^p_X : = \gr^p_F\, \DB^\bullet_X[p].
$$
Comparing this complex with the sheaf of $p$-forms allows us to define higher Du Bois singularities, which were developed recently as a refinement of Du Bois singularities. We focus our attention on the case of local complete intersection singularities, following \cite{JKSY, MOPW, MP22, FL24}.

\begin{defn}\label{defn:higher singularities}
Let $m\ge 0$ be an integer and let $X$ be a variety with local complete intersection singularities. We say $X$ has \emph{$m$-Du Bois singularities} if the natural morphisms $\Omega^p_X \to \DB^p_X$ are isomorphisms for all $0\le p\le m$.
\end{defn}

Here, $\Omega_X^p$ denotes the $p$-th exterior power of the sheaf of Kähler differentials on $X$.

Saito \cite{Saito93} defined the \emph{minimal exponent}, a singularity invariant which is closely related to the notion of higher Du Bois singularities. For a holomorphic function $f\in \O_{\C^n,0}$ on a germ $(\C^n,0)$ vanishing at the origin, there exist nonzero polynomials $b(s)\in \C[s]$ and $P(s)\in \D_{\C^n,0}[s]$ satisfying
$$
P(s)f^{s+1}=b(s)f^s
$$
formally in $\O_{\C^n,0}[\frac{1}{f},s]\cdot f^s$, where $\D_{\C^n,0}$ is the ring of differential operators on the germ $(\C^n,0)$. The \emph{Bernstein--Sato polynomial} $b_f(s)$ of $f$ is defined as the monic polynomial of minimal degree that generates the principal ideal in $\C[s]$ of all polynomials $b(s)$ satisfying the above equation. The \emph{minimal exponent} $\wt \alpha_{0}(f)$ is defined as the negative of the greatest root of the reduced Bernstein--Sato polynomial $\wt b_f(s):=b_f(s)/(s+1)$. It is well known that $s+1$ is always a factor of $b_f(s)$, and that $b_f(s)=s+1$ if and only if the hypersurface $\left\{f=0\right\}$ is smooth. In this case, we set $\wt \alpha_{0}(f)=+\infty$.

For a variety $X$ with hypersurface singularities, the minimal exponent of $X$ is defined as
$$
\wt \alpha(X):=\min_{x\in X} \wt\alpha_0(f)
$$
where $f$ is a local defining equation for the germ $(X,x)$ embedded in the germ $(\C^n,0)$. This invariant is independent of the choice of an embedding.

In \cite{JKSY, MOPW}, an equivalent criterion for higher Du Bois singularities was established in terms of the minimal exponent.

\begin{thm}
\label{thm:higher singularities via minimal exponent}
Let $X$ be a variety with hypersurface singularities. Then $X$ has $m$-Du Bois singularities if and only if the minimal exponent $\wt \alpha(X)\ge m+1$.
\end{thm}

For a proper family of varieties with Du Bois singularities, the higher direct images of the structure sheaf are locally free and compatible with base change. Friedman and Laza \cite{FL24} generalized this property for varieties with local complete intersection higher Du Bois singularities. We restrict our attention to hypersurface singularities.

\begin{thm}[{\cite[Theorem 1.2]{FL24}}]
\label{thm:higher Du Bois higher direct image}
Let $f:\X\to B$ be a flat proper family of varieties with $m$-Du Bois hypersurface singularities. Then the higher direct image sheaves
$$
R^qf_*\Omega_{\X/B}^p
$$
of the relative sheaf of $p$-th Kähler differentials are locally free and compatible with base change for all $0\le p\le m$ and $q\ge 0$.
\end{thm}

This is used crucially to prove the extendability of the compactified period map under the assumptions on higher Du Bois singularities of fibers. See Section \ref{sec:extendability for higher Du Bois}.

On the other hand, to prove that the extendability is precisely equivalent to a certain higher Du Bois singularity condition, we need the following technical description of the minimal exponent in terms of the vanishing cycle functor. It is formulated using Saito's theory of mixed Hodge modules \cite{Saito88, Saito90} and is stated in \cite{Saito17}; see also \cite[Proposition 1.3]{MSY24}.

\begin{thm}
\label{thm:minimal exponent via vanishing cycle}
Let $\Q^H_{\C^n,0}[n]$ be the Hodge module associated to the shifted constant sheaf $\Q_{\C^n,0}[n]$ on the germ $(\C^n,0)$. Denote by $\phi_{f}\Q^H_{\C^n,0}[n]$ the mixed Hodge module associated to the vanishing cycle $\phi_f\Q_{\C^n,0}[n]$. For $\lambda\in \C^*$ with absolute value $1$, let
$$
\phi_{f,\lambda}\Q^H_{\C^n,0}[n]
$$
be the filtered D-module associated to the generalized $\lambda$-eigenspace of $\phi_f\Q_{\C^n,0}[n]$ with respect to the monodromy operator induced by $f:(\C^n,0)\to (\C,0)$. Then we have
$$
\wt\alpha_0(f)=\min\left\{\alpha\in \Q\mid\lambda=e^{-2\pi i \alpha},
\gr_F^{\lfloor n-\alpha\rfloor}\DR\left(\phi_{f,\lambda}\Q^H_{\C^n,0}[n]\right)\neq 0\right\}.
$$
\end{thm}

Note that the underlying filtered D-module of $\phi_{f}\Q^H_{\C^n,0}[n]$ admits a direct sum decomposition
$$
\bigoplus_{|\lambda|=1}\phi_{f,\lambda}\Q^H_{\C^n,0}[n]
$$
into generalized eigenspaces with respect to the monodromy operator \cite[(2.2.6)]{Saito90}. The above description of the minimal exponent is a reformulation of \cite[(1.3.8)]{Saito17}, explained below.

\begin{proof}
By Saito \cite[(1.3.8)]{Saito17}, we have
$$
\wt\alpha_0(f)=\min\left\{\alpha\in \Q\mid\gr_{\wt V}^\alpha\O_{\C^n,0}\neq 0\right\}.
$$
By \cite[Corollaries 1.2 and 1.3]{MSY24}, we have
$$
\gr_{\wt V}^\alpha\O_{\C^n,0}=\gr^F_{\lceil \alpha\rceil}\mathcal V_f^{\{\lambda\}}\quad \mathrm{where}\quad \lambda=e^{-2\pi i \alpha}.
$$
Here, $\mathcal V_f^{\{\lambda\}}$ denotes the underlying left filtered D-module on $(\C^n,0)$ associated to $\phi_{f,\lambda}\Q^H_{\C^n,0}[n]$. For a fixed $\lambda$, we have
$$
\min\left\{p\in \Z\mid\gr^F_p\mathcal V_f^{\{\lambda\}}\neq 0\right\}=\min\left\{p\in \Z\mid\gr^F_{p-n}\DR\left(\phi_{f,\lambda}\Q^H_{\C^n,0}[n]\right)\neq 0\right\}
$$
from the definition of the graded de Rham functor for filtered D-modules. Therefore, we have
$$
\wt\alpha_0(f)=\min\left\{\alpha\in \Q\mid \lambda=e^{-2\pi i \alpha}, \gr^F_{\lceil \alpha\rceil}\mathcal V_f^{\{\lambda\}}\neq 0\right\}
$$
which is equivalent to the desired formula, since $\lceil \alpha\rceil-n=-\lfloor n-\alpha\rfloor$.
\end{proof}

In the main body of the paper, we repeatedly use the Thom--Sebastiani theorem for minimal exponents.

\begin{thm}[{\cite[Proposition 0.7]{Saito94}}]
\label{thm:Thom-Sebastiani for minimal exponents}
Let $f_1\in \O_{\C^{n_1},0}$ and $f_2\in \O_{\C^{n_2},0}$ be nonzero holomorphic functions on germs $(\C^{n_1},0)$ and $(\C^{n_2},0)$ vanishing at the origin. Then
$$
\wt\alpha_0(f_1\oplus f_2)=\wt\alpha_0(f_1)+\wt\alpha_0(f_2)
$$
where $f_1\oplus f_2\in \O_{\C^{n_1+n_2},0}$ is a holomorphic function on $(\C^{n_1+n_2},0)$ vanishing at the origin.
\end{thm}

\subsection{Condition $D_m$ and higher rational singularities}
\label{sec:D_m and higher rational}

Let $X$ be an equidimensional complex variety and $D^b{\rm MHM}(X)$ be the derived category of mixed Hodge modules on $X$, defined by Saito \cite{Saito88, Saito90}. By \cite{Saito00}, the $p$-th Du Bois complex of $X$ is naturally isomorphic to the (shifted) graded de Rham complex of the object $\Q_X^H[\dim X]\in D^b{\rm MHM}(X)$ associated to the constant sheaf $\Q_X$:
$$
\DB^p_X\isom\gr_F^{p}\DR(\Q^H_X[\dim X])[p-\dim X].
$$
Analogously, the $p$-th intersection Du Bois complex $I\DB^p_X$ of $X$ is the (shifted) graded de Rham complex of the pure Hodge module $\IC^H_X\in D^b{\rm MHM}(X)$ associated to the intersection complex of $X$:
$$
I\DB^p_X := \gr_F^{p}\DR(\IC_X^H)[p-\dim X].
$$
The natural morphism
$$
\gamma_X: \Q^H_X[\dim X]\to \IC^H_X
$$
in $D^b{\rm MHM}(X)$ induces a natural morphism
$$
\gamma_p: \DB^p_X \to I\DB^p_X
$$
in $D^b_{\rm coh}(X,\O_X)$.

\begin{defn}
We say $X$ satisfies \emph{Condition $D_m$} if $\gamma_p$ is an isomorphism for all $0\le p\le m$. 
\end{defn}

This condition was introduced and studied extensively in \cite{PP25a, PP25b}; see also \cite{DOR25}, where it is described in terms of the invariant $\rm{HRH}$. The notion of $m$-rational singularities was introduced by Friedman and Laza in \cite{FL24}, in terms of the Grothendieck dual of the Du Bois complex. Chen, Dirks, and Musta\c{t}\u{a} subsequently proved that $m$-rational lci singularities are equivalent to $m$-Du Bois lci singularities satisfying $D_m$. With a slight abuse of chronology, we take this characterization as the definition of $m$-rational singularities.

\begin{defn}[{\cite[Theorem 3.1]{CDM24}}] Let $X$ be a variety with local complete intersection singularities. We say $X$ has \emph{$m$-rational singularities} if it has $m$-Du Bois singularities and satisfies $D_m$.
\end{defn}

Lastly, we record a technical proposition on the vanishing of graded de Rham complexes of the vanishing cycle functor, under a suitable higher singularity condition.

\begin{prop}
\label{prop:vanishing grdr for vanishing cycle}
Let $f:\X\to \Delta$ be a projective one-parameter degeneration with a reduced special fiber $X$. If $X$ has $m$-Du Bois lci singularities and $\X\sm X$ has $m$-rational lci singularities, then
$$
\gr_F^p\DR(\phi_{f,1}\Q_\X^H)=0\quad\text{for all}\quad p\le m.
$$
\end{prop}

Here, $\phi_{f,1}:D^b{\rm MHM}(\X)\to D^b{\rm MHM}(X)$ is the generalized 1-eigenspace vanishing cycle functor in the derived category of mixed Hodge modules \cite[Theorem 0.1]{Saito90}.

\begin{proof}
By Chen's inversion of adjunction \cite[Theorem 1.2]{Chen25} for higher singularities, $\X$ has $m$-rational singularities. In particular, $\X$ satisfies $D_m$. Since $m$-Du Bois singularities imply $(m-1)$-rational singularities \cite[Corollary 1.2]{CDM24}, $X$ satisfies $D_{m-1}$. Therefore, \cite[Proposition 11.2]{Park25} applies and we obtain the desired vanishing.
\end{proof}

\section{Extendability of period map for higher Du Bois singularities}
\label{sec:extendability for higher Du Bois}

Consider a family $f:(\X,\L)\to B$ of polarized projective varieties over a normal variety $B$ whose general fiber is smooth of dimension $n-1$. Here, $\L$ is the class of a relatively ample line bundle on $\X$ over $B$. Over the open locus of $B$ where the fibers are smooth, we have a polarized variation of $\Z$-Hodge structures associated to the primitive middle cohomologies of the fibers. Hence, we have a generically defined period map
$$
\Phi_f: B\dashrightarrow\Gamma\backslash \D.
$$
Under the assumption that the fibers of $f$ have higher Du Bois singularities, the compactified period map to the Baily--Borel compactification of the period image
$$
\overline{\Phi_f}:B\dashrightarrow \overline{M_f}^{\rm BB}
$$
is regular:

\begin{thm}
\label{thm:extendability}
In the above setting, the period map $\overline{\Phi_f}$ to the Baily--Borel compactification of the period image is regular if every fiber has $(\lfloor\frac{n}{2}\rfloor-1)$-Du Bois hypersurface singularities.
\end{thm}

By Theorem \ref{thm:higher singularities via minimal exponent}, the condition of having $(\lfloor\frac{n}{2}\rfloor-1)$-Du Bois hypersurface singularities is equivalent to the condition that the minimal exponent is at least $\lfloor\frac{n}{2}\rfloor$. In fact, there is an explicit classification of such singularities.

\begin{lem}
\label{lem:classification of n/2 Du Bois}
For $n\ge 3$ and a hypersurface singularity $\left\{h=0\right\}\subset (\C^n,0)$ at the origin, the minimal exponent $\wt\alpha_0(h)\ge\frac{n-1}{2}$ if and only if $\{h=0\}$ is a suspension of a semi-log canonical surface singularity. Furthermore, $\wt\alpha_0(h)\ge\frac{n}{2}$ if and only if $\{h=0\}$ is an ordinary double point.
\end{lem}

\begin{proof} Firstly, we prove the equivalence for $\wt\alpha_0(h)\ge\frac{n-1}{2}$.
We proceed by induction on $n$. When $n=3$, we have $\wt\alpha_0(h)\ge 1$ if and only if the log canonical threshold of $h$ is $1$ near the origin (see \cite[Theorem 10.6]{Kollar97}). By inversion of adjunction, this is equivalent to $\{h=0\}$ having semi-log canonical singularities near the origin by \cite[Theorem 4.9]{Kollar13}.

Suppose the claim is true for a hypersurface singularity in $(\C^{n-1},0)$ at the origin and $n\ge 4$. Now, for a hypersurface singularity $\left\{h=0\right\}\subset (\C^n,0)$ at the origin, we have
$$
\wt\alpha_0(h)\le \frac{n}{\mathrm{mult}_0(h)}
$$
by \cite[Theorem E(3)]{MP20}. Hence, $\mathrm{mult}_0(h)=2$ and the rank of the Hessian matrix is at least $1$ at the origin. This implies that $\left\{h=0\right\}$ is locally analytically isomorphic to the hypersurface singularity
$$
\{H(x_1,\dots,x_{n-1})+x_{n}^2=0\}\subset(\C^n,0)
$$
at the origin. By the Thom--Sebastiani theorem (Theorem \ref{thm:Thom-Sebastiani for minimal exponents}) and the induction hypothesis, $\left\{H=0\right\}\subset (\C^{n-1},0)$ is a suspension of a semi-log canonical surface singularity, which completes the proof.

The equivalence for $\wt\alpha_0(h)\ge\frac{n}{2}$ is proved in the same way. See also \cite[Corollary 6.3]{DM23} for a slightly different proof.
\end{proof}

To show that some rational map is regular on $B$, it suffices to prove that the induced map on $\wt B$ is regular, where $\mu:\wt B \to B$ is a proper birational modification, and every fiber of $\mu$ is contracted. We include a short proof of this fact.

\begin{lem}
\label{lem:extendability via resolution}
Let $B$ and $M$ be varieties, with $B$ normal. Let $\Phi:B\dashrightarrow M$ be a rational map. Suppose there exists a proper birational morphism $\mu:\wt B \to B$ such that the induced map $$\wt\Phi:=\Phi\circ\mu:\wt B\to M$$ is regular. Then $\Phi$ is regular if and only if, for every closed point $b\in B$, the fiber $\mu^{-1}(b)$ is mapped by $\wt\Phi$ to a closed point in $M$.
\end{lem}

\begin{proof}
It suffices to prove the converse. Suppose the fiber $\mu^{-1}(b)$ is contracted to a point $m\in M$ by $\wt \Phi$. We prove that there exists a neighborhood $U$ of $b\in B$ such that $\Phi$ is regular on $U$.

Let $\mathrm{Spec}\, A$ be an affine neighborhood of the point $m\in M$. Since $\mu$ is proper, the image $\mu(\wt B\sm \wt\Phi^{-1}(\mathrm{Spec}\, A))$ is closed and does not contain $b$. Hence, there exists an affine neighborhood $\mathrm{Spec}\, R$ of $b\in B$ such that $$\mu^{-1}(\mathrm{Spec}\, R)\subset \wt\Phi^{-1}(\mathrm{Spec}\, A).$$
Note that the morphism $\wt\Phi:\mu^{-1}(\mathrm{Spec}\, R)\to \mathrm{Spec}\, A$ is equivalent to the ring homomorphism
$$
A\to \Gamma(\mu^{-1}(\mathrm{Spec}\, R)).
$$
Since $B$ is normal and $\mu$ is birational, the global section ring $\Gamma(\mu^{-1}(\mathrm{Spec}\, R))$ is isomorphic to $R$. Hence, we have a regular morphism $\mathrm{Spec}\, R\to \mathrm{Spec}\, A$, which implies that $\Phi$ is regular on $\mathrm{Spec}\, R$.
\end{proof}

As shown later in this section, Lemma \ref{lem:extendability via resolution} reduces the proof of Theorem \ref{thm:extendability} to the case where we have a semistable reduction over $B$. We prove the following key isomorphism of Hodge-graded pieces of Deligne's canonical extension under the assumption on higher Du Bois singularities of the fibers.

\begin{prop}
\label{prop:Hodge pieces for higher Du Bois family}
Let $f:\X\to B$ be a family of projective varieties over a smooth variety $B$, and $D\subset B$ be a simple normal crossing divisor such that $f$ is smooth over $B\sm D$. Suppose every fiber of $f$ has $m$-Du Bois local complete intersection singularities for some nonnegative integer $m\ge 0$. Denote by $(\overline{\mathcal V}^{[0,1)}, F^\bullet)$ Deligne's lower canonical extension of the polarizable variation of Hodge structure $(R^kf_*\Z_{\X})|_{B\sm D}$ to $B$. Then there exists an isomorphism
$$
\gr_F^p\overline{\mathcal V}^{[0,1)}\isom R^{k-p}f_*\Omega_{\X/B}^p
$$
for all $0\le p\le m$.
\end{prop}

Here, we do not specify the polarization, since Deligne's lower canonical extension is independent of the choice of a polarization. The hypersurface singularity assumption for $f$ is not necessary provided that $m$-Du Bois singularities are defined appropriately for the argument below.

\begin{proof}
By Theorem \ref{thm:higher Du Bois higher direct image}, the sheaf $R^{k-p}f_*\Omega_{\X/B}^p$ is locally free for all $0\le p\le m$. Since $\gr^p_F\overline{\mathcal V}^{[0,1)}$ is also locally free, it suffices to prove the isomorphism
$$
\gr_F^p\overline{\mathcal V}^{[0,1)}\isom R^{k-p}f_*\Omega_{\X/B}^p
$$
on the open subset $B^o\subset B$ with $\codim_B B\sm B^o\ge 2$. Indeed, the isomorphism on $B^o$ extends to $B$ by taking the reflexive hulls of both sides.

There exists an open subset $B^o\subset B$ with $\codim_B B\sm B^o\ge 2$ such that $D|_{B^o}$ is smooth. Consider a resolution
$$
\eta:\wt \X_{B^o}\to\X_{B^o}
$$
and the induced morphism
$$
\wt f=f\circ\eta:\widetilde{\mathcal X}_{B^o}\to B^o
$$
such that the set-theoretic preimage $\wt f^{-1}(D)=E$ is a reduced simple normal crossing divisor. We may assume that $\eta$ is an isomorphism over $B\sm D$. Furthermore, we may restrict $B^o\subset B$ with $\codim_B B\sm B^o\ge 2$ such that $\wt f|_{E_I}:E_I\to D$ is smooth for every stratum $E_I:=\cap_{i\in I} E_i$ of $E=\cup_{i} E_i$. Here, $E_i$ is an irreducible component of $E$. Hence, we assume $B^o=B$ and omit the subscript $B^o$ everywhere.

From the local analytic description, we have a natural inclusion of the logarithmic sheaves of Kähler differentials
$$
\wt f^*\Omega_B(\log D)\to \Omega_{\wt \X}(\log E)
$$
such that the cokernel sheaf of relative log differentials $\Omega_{\wt\X/B}(\log E)$ is locally free. By Steenbrink \cite{Steenbrink75} (see also Zucker \cite{Zucker84}), it is well known that
$$
\gr_F^p\overline{\mathcal V}^{[0,1)}\isom R^{k-p}\wt f_*\Omega_{\wt\X/B}^p(\log E)
$$
for all $p\ge 0$. Therefore, it suffices to prove the following claim, using the Leray spectral sequence for the composition $\wt f=f\circ \eta$.

\smallskip
\noindent
\textit{Claim.} There exists a natural isomorphism $R\eta_*\Omega_{\wt\X/B}^p(\log E)\isom\Omega_{\X/B}^p$ for all $0\le p\le m$.
\smallskip

This isomorphism is equivalent to two separate statements: $\eta_*\Omega_{\wt\X/B}^p(\log E)\isom\Omega_{\X/B}^p$ and $R^i\eta_*\Omega_{\wt\X/B}^p(\log E)=0$ for $i>0$.

\smallskip
\noindent
\textit{Step 1.} Isomorphism of sheaves $\eta_*\Omega_{\wt\X/B}^p(\log E)\isom\Omega_{\X/B}^p$ for all $0\le p\le m$.
\smallskip

Note that $\X$ has $m$-rational singularities. Indeed, for each point $b\in B$, the fiber $\X_b$ has $m$-Du Bois singularities. Consider nested general $i$-dimensional plane sections $H_i$ of the germ $(B,b)$:
$$
b\in H_1\subset\dots\subset H_{\dim B-1}\subset B.
$$
Here, $H_i$ is a general smooth Cartier divisor of $H_{i+1}$. Then the restricted family $\X_{H_1}\to H_1$ is smooth away from $b\in H_1$ and the fiber $\X_b$ has $m$-Du Bois singularities. Hence, $\X_{H_1}$ has $m$-rational singularities by Chen's inversion of adjunction \cite[Theorem 1.2]{Chen25}. Therefore, the iterative application of the inversion of adjunction \cite[Theorem 1.2]{CDMO24} implies that $\X$ has $m$-rational singularities.

The relative sheaf of $p$-th differentials $\Omega_{\X/B}^p$ is a reflexive sheaf on $\X$ for all $0\le p\le m$ by \cite[Corollary 2.3]{FL24} and the fact that 
$$
2\dim \Sing(\X_b)+1\le \dim \X_b
$$
for every fiber $\X_b$ (see \cite[Corollary 3.40]{MP22}). We have the following commutative diagram
\begin{displaymath}
\xymatrix{
{0} \ar[r] & {f^*\Omega_{B}} \ar[r] \ar[d] & {\Omega_{\X}} \ar[r] \ar[d]  & {\Omega_{\X/B}} \ar[r] \ar[d] & {0} \\
{0} \ar[r] & {\eta_*\wt f^*\Omega_{B}(\log D)} \ar[r] & {\eta_*\Omega_{\wt\X}(\log E)} \ar[r] & {\eta_*\Omega_{\wt\X/B}(\log E)} \ar[r] & 0
}
\end{displaymath}
Note that the first row is exact because $f$ is generically smooth, and the second row is exact because this is the pushforward of the short exact sequence
\begin{equation}
\label{eqn:log SES}
0\to \wt f^*\Omega_B(\log D)\to \Omega_{\wt \X}(\log E)\to \Omega_{\wt\X/B}(\log E)\to 0
\end{equation}
and $R\eta_*\wt f^*\Omega_{B}(\log D)=\eta_*\wt f^*\Omega_{B}(\log D)$. Recall that $\X$ has rational singularities. 

In particular, we have a map $\Omega_{\X/B}\to\eta_*\Omega_{\wt\X/B}(\log E)$ which is an isomorphism on $\X\sm \Sing(f)$. Here, $\Sing(f)$ denotes the singular locus of the morphism $f$. In fact, we may choose $\eta$ to be an isomorphism over $\X\sm \Sing(f)$.
Hence, for $0\le p\le m$, we have an inclusion
$$
\Omega^p_{\X/B}\subset\eta_*\Omega^p_{\wt\X/B}(\log E)
$$
which is an isomorphism away from a codimension $2$ subset of $\X$. The reflexivity of $\Omega^p_{\X/B}$ implies the desired isomorphism.

\smallskip
\noindent
\textit{Step 2.} Vanishing of higher direct images $R^i\eta_*\Omega_{\wt\X/B}^p(\log E)=0$ for all $i>0$ and $0\le p\le m$.
\smallskip

Consider the Koszul filtration associated to the short exact sequence \eqref{eqn:log SES}:
$$
\Koz^q \, \Omega_{\wt\X}^p(\log E):= \mathrm{image}\, \left[\wt f^*\Omega_{B}^{q}(\log D)\tensor \Omega_{\wt\X}^{p-q}(\log E)\to \Omega_{\wt\X}^p(\log E)\right].
$$
Then we have the isomorphism
$$
\Koz^q/\Koz^{q+1}\, \left(\Omega_{\wt\X}^p(\log E)\right)\isom \wt f^*\Omega_{B}^{q}(\log D)\tensor \Omega_{\wt\X/B}^{p-q}(\log E)
$$
for all integers $q\ge 0$. We proceed by induction on $m$ to prove the vanishing of higher direct images: we assume $R^i\eta_*\Omega_{\wt\X/B}^p(\log E)=0$ for all $i>0$ and $0\le p\le m-1$, and prove the vanishing for $p=m$.

This assumption implies that
$$
R^i\eta_*\left[\Koz^q/\Koz^{q+1}\, \left(\Omega_{\wt\X}^m(\log E)\right)\right]=0
$$
for all $q>0$ and $i>0$. Since this vanishing for $q=0$ is equivalent to
$$
R^i\eta_*\Omega_{\wt\X/B}^m(\log E)=0
$$
for all $i>0$, the spectral sequence argument shows that it suffices to prove
$$
R^i\eta_*\Omega_{\wt\X}^m(\log E)=0
$$
for all $i>0$.

Denote by $\wt j:\wt\X\sm E\to \wt\X$ and $j:\X\sm f^*(D)\to \X$ the open embeddings. Note that $\eta:\wt \X\to \X$ induces an isomorphism between $\wt\X\sm E$ and $\X\sm f^*(D)$. Let $\mathcal U:=\X\sm f^*(D)$. By Saito \cite{Saito00}, we have
$$
\Omega_{\wt\X}^m(\log E)\isom \gr_F^m\DR(\wt j_*\Q_{\mathcal U}^H)[m].
$$
Since the graded de Rham functor commutes with proper (derived) pushforward (see, for example, \cite[Proposition 2.8]{Saito00} or \cite[(2.3.7)]{Saito88}), we have
\begin{equation}
\label{eqn:grdr on U for resolution}
R\eta_*\Omega_{\wt\X}^m(\log E)\isom \gr_F^m\DR(j_*\Q_{\mathcal U}^H)[m].   
\end{equation}
Since the claim is a local statement, it suffices to prove the desired vanishing in a neighborhood of every point $x\in f^*(D)\subset \X$. Let $g=0$ be the local defining equation of the Cartier divisor $f^*(D)\subset \X$ near $x$.

Denote by $\iota:f^*(D)\to \X$ the closed embedding. Then we have the distinguished triangles
\begin{gather*}
\iota^*\Q_\X^H[-1]\to \psi_{g,1}\Q_{\X}^H\to\phi_{g,1}\Q_{\X}^H\xrightarrow{+1},\\
\iota^!\Q_\X^H\to\phi_{g,1}\Q_{\X}^H\to \psi_{g,1}\Q_{\X}^H(-1)\xrightarrow{+1}
\end{gather*}
by \cite[Corollary 2.24]{Saito90}. Note that $f:f^*(D)\to D$ is a family of varieties with $m$-Du Bois singularities. The same argument as at the beginning of Step 1 implies that $f^*(D)$ has $m$-Du Bois singularities.

Applying the graded de Rham functor and Proposition \ref{prop:vanishing grdr for vanishing cycle}, we obtain the following isomorphism
$$
\gr_F^{m-1}DR(\iota^*\Q_\X^H)[-1]\isom\gr_F^{m-1}\DR(\psi_{g,1}\Q_\X^H)\isom\gr_F^m\DR(\iota^!\Q_\X^H)[1].
$$
Recall that
$$
\gr_F^{m-1}DR(\iota^*\Q_\X^H)\isom\gr_F^{m-1}DR(\Q_{f^*(D)}^H)\isom\Omega^{m-1}_{f^*(D)}[-m+1]
$$
where the second isomorphism follows from the fact that $f^*(D)$ has $m$-Du Bois singularities. Hence, we obtain
$$
\gr_F^m\DR(\iota^!\Q_\X^H)\isom\Omega^{m-1}_{f^*(D)}[-m-1].
$$
Applying the graded de Rham functor to the distinguished triangle \cite[(4.4.1)]{Saito90}
$$
\iota_*\iota^!\Q_\X^H\to \Q_\X^H\to j_*\Q_{\mathcal U}^H\xrightarrow{+1},
$$
we obtain the following distinguished triangle
$$
\Omega^{m-1}_{f^*(D)}[-m-1]\to \Omega^{m}_{\X}[-m]\to \gr_F^m\DR(j_*\Q_{\mathcal U}^H)\xrightarrow{+1}.
$$
Therefore, we have the desired vanishing of higher direct images by \eqref{eqn:grdr on U for resolution}.
\end{proof}

\begin{proof}[Proof of Theorem \ref{thm:extendability}]
Let $Z:=\Sing(B)\cup \mathrm{Discr}(f)\subset B$ be the union of the singular locus of $B$ and the discriminant locus (i.e., the set of points whose fibers are singular) of the family $f$. Consider a resolution $\mu:\wt B\to B$ such that $D:=\mu^{-1}(Z)$ is a simple normal crossing divisor. Then the compactified period map 
$$
\wt\Phi_f:\wt B\to \overline{M_f}^{\rm BB}
$$
is everywhere regular by Theorem \ref{thm:BFMT BB} (or one may consider a resolution $\mu$ such that the period map $\overline{\Phi_f}$ is resolved).

By Lemma \ref{lem:extendability via resolution}, it suffices to show that the fiber $\mu^{-1}(b)$ is contracted to a point by $\wt\Phi_f$. By Theorem \ref{thm:BFMT BB}, this is equivalent to showing that the Griffiths bundle $\G_{\wt B}$ is numerically trivial when restricted to $\mu^{-1}(b)$.

Denote by $\wt f:\wt\X\to \wt B$ the family pulled back from $f:\X\to B$. By Kawamata's covering trick \cite[Theorem 17]{Kawamata81}, there exists a smooth variety $T$ and a finite flat morphism $\pi:T\to \wt B$ such that $\pi^{-1}(D)$ is a simple normal crossing divisor and the local monodromy of 
$$
\pi^*\left(R^{n-1}f_*\Z_\X\right)|_{T\sm \pi^{-1}(D)}
$$
is unipotent along $\pi^{-1}(D)$. Denote by $g:\X_T\to T$ the family pulled back from $f$, and $(\overline{\mathcal V}, F^\bullet)$ Deligne's canonical extension to $T$.
By Proposition \ref{prop:Hodge pieces for higher Du Bois family}, we have
$$
\gr_F^p\overline{\mathcal V}\isom R^{n-1-p}g_*\Omega_{\X_T/T}^p
$$
for all $0\le p\le \lfloor\frac{n}{2}\rfloor-1$. Hence, the associated Griffiths bundle is
$$
\G_{T}\isom_\Q\bigotimes_{p=0}^{\lfloor\frac{n}{2}\rfloor-1}\det\left(R^{n-1-p}g_*\Omega_{\X_T/T}^p\right)^{\otimes -n+1+2p}
$$
by Lemma \ref{lem:Griffiths bundle alternative}. From the definition of the Griffiths bundle, we have
$$
\G_T\isom_\Q\pi^*\G_{\wt B}
$$
and by Theorem \ref{thm:higher Du Bois higher direct image}, we have
$$
\G_{T}\isom_\Q\bigotimes_{p=0}^{\lfloor\frac{n}{2}\rfloor-1}(\mu\circ\pi)^*\det\left(R^{n-1-p}f_*\Omega_{\X/B}^p\right)^{\otimes -n+1+2p}.
$$
This implies that $\G_T$ is numerically trivial on $(\mu\circ\pi)^{-1}(b)$, and thus, $\G_{\wt B}$ is numerically trivial on $\mu^{-1}(b)$. This completes the proof.
\end{proof}

The same strategy as in the proof above shows that the period map associated to the second cohomology is regular if every fiber has Du Bois singularities. In particular, this applies to KSB-moduli spaces.

\begin{cor}
Let $\overline M_{n,v}^\nu$ be the normalization of an irreducible component of the KSB-moduli space of varieties with dimension $n$ and volume $v$. If the general KSB-stable variety $X$ of $\overline M_{n,v}^\nu$ is smooth, then the associated period map to the Baily--Borel compactification of the period image associated to the primitive second cohomology $H^{2}_{\rm prim}(X,\mathbb Z)$
$$
\overline {\mathcal P}_{n,v}^2:\overline M_{n,v}^\nu \to \overline M^{\rm BB}
$$
is regular.
\end{cor}

\begin{proof}
Note that the KSB-moduli space $\overline M_{n,v}$ is the geometric quotient $B/\!/ PGL(N+1)$ (see e.g. \cite[Theorem 8.62]{Kollar23}). Here, $B$ parametrizes pluricanonically embedded KSB-stable varieties in $\P^N$. Denote by $B^\nu$ the normalization of the parameter space associated to the irreducible component of the KSB-moduli space. Then $\overline M_{n,v}^\nu$ is the geometric quotient $B^\nu/\!/ PGL(N+1)$. From the proof of Theorem \ref{thm:extendability} and from the fact that KSB-stable varieties have Du Bois singularities (see e.g. \cite{KK10}), the compactified period map
$$
\overline \Phi^2_{n,v}: B^\nu\to\overline M^{\rm BB}
$$
is regular. Indeed, Proposition \ref{prop:Hodge pieces for higher Du Bois family} applies to families with (not necessarily lci) Du Bois singularities, since the general form of inversion of adjunction can be applied (see \cite{KS16} and \cite[Corollary 3.10]{MSS17}). Then the proof of Theorem \ref{thm:extendability} works verbatim in the setting of second cohomology. Additionally, this map is $PGL(N+1)$-invariant, and therefore descends to $\overline M_{n,v}^\nu$.
\end{proof}

Given an explicit moduli problem, it is an interesting question to give a concrete description of the compactified period map $\overline {\mathcal P}_{n,v}^2$. For instance, the period map from the KSB-moduli compactification for degree $d\ge 5$ surfaces in $\P^3$ is birational by the generic Torelli \cite{Donagi83}. Hence, $\overline {\mathcal P}_{2,d(d-4)^2}^2$ is a regular birational map on the irreducible component generically parametrizing surfaces in $\P^3$, but we do not have a good understanding of this period map.

\section{Numerical criterion for indeterminacy: Questions of Mordant}
\label{sec:numerical criterion}

By construction, the period map
$$
\mathcal P: \overline\M^{\rm GIT}_{(n,d)} \dashrightarrow \Gamma\backslash \D_{(n,d)}
$$
is regular on the open subset $\M^{\rm sm}_{(n,d)}$ parametrizing smooth hypersurfaces. Theorem \ref{thm:extendability} and Lemma \ref{lem:classification of n/2 Du Bois} imply that the compactified period map
$$
\overline {\mathcal P}:\overline\M^{\rm GIT}_{(n,d)} \dashrightarrow \overline \M^{\rm BB}_{(n,d)}
$$
is regular on the locus described in Theorem \ref{thm:indeterminacy}. Hence, it remains to prove the converse: $\overline{\mathcal P}$ does not extend beyond this locus. In this section, we prove Theorem \ref{thm:Mordant question and indeterminacy}, which is later used to determine the indeterminacy locus by computing the Griffiths and GIT heights for families of hypersurfaces over curves (see Section \ref{sec:Griffiths and GIT height} for definitions).

The proof of Theorem \ref{thm:Mordant question and indeterminacy} combines a key geometric property of the GIT moduli space with a numerical criterion for the indeterminacy locus. The former is established in the following.

\begin{prop}
\label{prop:GIT is Q factorial picard rank 1}
For $n\ge 2$ and $d\ge 3$, the GIT moduli space $\overline\M^{\rm GIT}_{(n,d)}$ is a $\Q$-factorial klt Fano variety with Picard rank $1$.
\end{prop}

\begin{proof}
We prove that $\overline\M^{\rm GIT}_{(n,d)}$ is $\Q$-factorial with Picard rank $1$. Denote by $V:=H^0(\P^n,\O(d))$ the vector space of homogeneous degree $d$ polynomials in $n+1$ variables. By \cite[Theorem A]{Park25} (see also \cite{Mordant24, He25}), every nodal hypersurface (i.e., a hypersurface with at worst ordinary double point singularities) is GIT stable. In particular, the moduli space $\M^{\rm nodal}_{(n,d)}$ of (at worst) nodal hypersurfaces has quotient singularities, and we have an open embedding
$$
\M^{\rm nodal}_{(n,d)}\subset\overline\M^{\rm GIT}_{(n,d)}.
$$
Note that every singular hypersurface is a limit of hypersurfaces with a single node. Indeed, let $\left\{P=0\right\}\subset\P^n$ be a hypersurface singular at a point $x\in \P^n$. Choose a general hypersurface $\left\{Q=0\right\}\subset\P^n$ of the same degree having a node at $x\in \P^n$. Then the pencil
$$
\left\{P+tQ=0\right\}\subset \P^n\times\A^1_t
$$
has $\left\{P=0\right\}$ as a special fiber. For general $t$, the hypersurface $\{P+tQ=0\}$ has a unique singularity, which is an ordinary double point.

This implies that the set of nodal hypersurfaces is an open dense subset in
$$
\overline\M^{\rm GIT}_{(n,d)}\sm\M^{\rm sm}_{(n,d)}
$$
where $\M^{\rm sm}_{(n,d)}$ is the moduli space of smooth hypersurfaces. Hence, $\M^{\rm nodal}_{(n,d)}$ contains every codimension $1$ point of $\overline\M^{\rm GIT}_{(n,d)}$.

Consider the following diagram:
\begin{displaymath}
\xymatrix{
{(\P V)^{\rm nodal}} \ar[r] \ar[d] & {(\P V)^{\rm ss}} \ar[d]^-{\pi}\\
{\M^{\rm nodal}_{(n,d)}} \ar[r] & {\overline\M^{\rm GIT}_{(n,d)}}
}
\end{displaymath}
where $(\P V)^{\rm ss}$ (resp. $(\P V)^{\rm nodal}$) is the parameter space of GIT semistable (resp. nodal) hypersurfaces, and
$$
\pi:(\P V)^{\rm ss}\to \overline\M^{\rm GIT}_{(n,d)}
$$
is the GIT quotient morphism. Note that the horizontal maps are open embeddings. For any effective divisor $D\subset \overline\M^{\rm GIT}_{(n,d)}$, there exists a sufficiently divisible $N>0$ such that $ND$ is Cartier on $\M^{\rm nodal}_{(n,d)}$. Hence, the section
$$
\O_{\M^{\rm nodal}_{(n,d)}}\to \O_{\M^{\rm nodal}_{(n,d)}}(ND)
$$
pulls back to a section of a line bundle
$$
\O_{(\P V)^{\rm nodal}}\to \O_{(\P V)^{\rm nodal}}(N\pi^*D).
$$
Additionally, the line bundle $\O_{(\P V)^{\rm nodal}}(N\pi^*D)$ is isomorphic to a line bundle
$$
\O_{(\P V)^{\rm nodal}}(N')
$$
for some positive integer $N'$. This line bundle extends to $\O_{\P V}(N')$, and thus, the divisor $N\pi^*D$ is a $PGL(n+1)$-invariant degree $N'$ section of $\P V$. This section descends to $\overline\M^{\rm GIT}_{(n,d)}$, that is, $ND$ is a divisor $\Q$-linearly equivalent to
$$
\O_{\overline\M^{\rm GIT}_{(n,d)}}(N').
$$
In other words, any (effective) divisor $D$ is $\Q$-linearly equivalent to $\O_{\overline\M^{\rm GIT}_{(n,d)}}(\alpha)$ for some rational number $\alpha$. Therefore, the GIT moduli space $\overline\M^{\rm GIT}_{(n,d)}$ is a $\Q$-factorial variety with Picard rank $1$. Consequently, it is klt Fano by \cite{BGLM24} or \cite{Zhuang24}.
\end{proof} 

Given a rational map of projective varieties from $X$ to $Y$, the next two lemmas establish inequalities relating the degrees of maps $C\to X$ and $C\to Y$ for a curve $C$. We apply these to the compactified period map $\overline {\mathcal P}$.

\begin{lem}
\label{lem:Mordant Q1 numerical criterion}
Let $X$ be a normal projective variety and $Y$ be a projective variety with an ample $\Q$-line bundle $\L_Y$. Let $f:X\dashrightarrow Y$ be a rational map and $X^o\subset X$ be the domain of definition of $f$. Suppose the pullback $\Q$-line bundle $f^*\L_Y|_{X^o}$ on $X^o$ is $\Q$-Cartier on $X$, and denote by
$$
\L_X\in \Pic(X)\otimes_{\mathbb Z}\mathbb Q
$$
its reflexive hull. Then for any smooth projective curve $C$ and a morphism $\phi:C\to X$ such that the generic point $C_\eta$ is mapped to $X^o$, we have
$$
\deg_C(f\circ\phi)^*\L_Y\le \deg_C \phi^*\L_X
$$
and the equality holds if and only if $\phi(C)\subset X^o$.
\end{lem}

A priori, the induced map $f\circ\phi:C\dashrightarrow Y$ is only a rational map, which extends regularly.

\begin{proof}
Let $\wt X$ be the normalization of the graph of $f$. Denote by $\mu:\wt X\to X$ and $\wt f:\wt X\to Y$ the associated projective morphisms. Note that $\mu$ is an isomorphism over $X^o$. Hence,
$$
\mu^*\L_X \quad\text{and}\quad \wt f^*\L_Y
$$
are isomorphic over $X^o$, and the difference
$$
D:= \wt f^*L_Y-\mu^*L_X
$$
is a $\Q$-Cartier divisor supported on $\wt X\sm\mu^{-1}(X^o)$ (by abuse of notation, we also use $L_X$ and $L_Y$ to denote $\Q$-divisors that are $\Q$-linearly equivalent to $\L_X$ and $\L_Y$). Since $L_Y$ is an ample divisor on $Y$, the divisor $D$ is $\mu$-nef. Therefore, by the negativity lemma \cite[Lemma 3.39]{KM98}, $-D$ is effective and for every closed point $x\in X\sm X^o$, either
$$
\mu^{-1}(x)\subset \Supp (-D) \quad\text{or}\quad \mu^{-1}(x)\cap \Supp (-D)=\emptyset.
$$

If $\mu^{-1}(x)\cap \Supp (-D)=\emptyset$, then $\mu^{-1}(x)$ is contracted to a point by $\wt f$. Indeed, $-D$ is numerically trivial on $\mu^{-1}(x)$, and thus, $\wt f^*L_Y$ is numerically trivial on $\mu^{-1}(x)$. Then by Lemma~\ref{lem:extendability via resolution}, the point $x$ is contained in the domain of definition of $f$: $x\in X^o$. This is a contradiction. Therefore, we have
$$
\Supp (-D)=\mu^{-1}(X\sm X^o).
$$

Let $\wt\phi:C\to \wt X$ be the morphism that lifts $\phi$. Since the generic point of $C$ is mapped to $X^o$, we have
$$
\deg_C(f\circ\phi)^*\L_Y-\deg_C \phi^*\L_X=\deg_C\wt\phi^*D\le0
$$
and the equality holds if and only if $\wt\phi(C)$ does not intersect $\Supp (-D)$. This completes the proof.
\end{proof}

\begin{lem}
\label{lem:Mordant Q2}
Let $X$ be a $\Q$-factorial normal projective variety with Picard rank $1$, and let $Y$ be a projective variety with an ample $\Q$-line bundle $\L_Y$. Let $f:X\dashrightarrow Y$ be a birational map, and $\L_X$ be the reflexive hull of the pullback $f^*\L_Y$. Suppose there exists an open subset $U\subset X$ such that $f|_U:U\to Y$ is an open embedding. Then there exists $\epsilon>0$ such that for any smooth projective curve $C$ and a morphism $\phi:C\to X$ with $\phi(C)\cap U\neq \emptyset$, we have
$$
\epsilon\deg_C \phi^*\L_X\le\deg_C(f\circ\phi)^*\L_Y.
$$
\end{lem}

\begin{proof}
As in the proof of Lemma \ref{lem:Mordant Q1 numerical criterion}, we consider the normalization $\wt X$ of the graph of $f$ and use the same notations for induced maps. By abuse of notation, we consider $U$ as an open subset of $X$, $Y$, and $\wt X$. Let $Z:=Y\sm U$, equipped with the reduced scheme structure, and let $\I_Z$ be the ideal sheaf of $Z$ in $Y$. From the proof of Lemma \ref{lem:Mordant Q1 numerical criterion}, we have
$$
\Supp(-D)\subset \wt f^{-1}(Z)
$$
where $D:= \wt f^*L_Y-\mu^*L_X$. Upon taking further blow-ups, we may assume
$$
\wt f^*\I_Z=\O_{\wt X}(-E)
$$
for some effective Cartier divisor $E$, satisfying $\Supp(E)=\wt X\sm U$.

Let $N$ be a sufficiently large integer such that
$$
\L_Y^N\otimes\I_Z
$$
is globally generated. This implies that
$$
\wt f^*\L_Y^N\otimes\O_{\wt X}(-E)
$$
is nef. We choose an integer $k\ge 0$ such that $kE\ge -D$. Then, for any morphism $\phi:C\to X$ with $\phi(C)\cap U\neq \emptyset$, we have
$$
(kN+1)\deg_C(f\circ \phi)^*\L_Y\ge k\deg_C \wt\phi^*(E)+\deg_C(f\circ \phi)^*\L_Y
$$
where $\wt\phi:C\to \wt X$ is the induced morphism that lifts $\phi$. Furthermore,
$$
k\deg_C \wt\phi^*(E)+\deg_C(f\circ \phi)^*\L_Y\ge \deg_C \wt\phi^*(-D)+\deg_C(f\circ \phi)^*\L_Y=\deg_C \phi^*\L_X.
$$
Taking $\epsilon=\frac{1}{kN+1}$ completes the proof.
\end{proof}

Now, we prove Theorem \ref{thm:Mordant question and indeterminacy}, thereby answering questions of Mordant \cite[Section 5.3]{Mordant25}.

\begin{proof}[Proof of Theorem \ref{thm:Mordant question and indeterminacy}]
Note that the compactified period map
$$
\overline {\mathcal P}:\overline\M^{\rm GIT}_{(n,d)} \dashrightarrow \overline \M^{\rm BB}_{(n,d)}
$$
is defined over a codimension $1$ point. Consider a general pencil
$$
\X:=\{P+tQ=0\}\subset \P^n\times\P_t^1
$$
over a rational curve $\P^1_t$. Then the GIT moduli image $\overline\Phi_{\rm GIT}(\P^1_t)$ is contained in the domain of definition of $\overline {\mathcal P}$.

By Proposition \ref{prop:GIT is Q factorial picard rank 1}, the reflexive hull of the pullback of the Griffiths bundle $\O_{\overline \M^{\rm BB}_{(n,d)}}(1)$ is a rational multiple of the GIT polarization $\O_{\overline\M^{\rm GIT}_{(n,d)}}(1)$:
$$
\O_{\overline\M^{\rm GIT}_{(n,d)}}(\alpha)\isom_\Q \left[\overline {\mathcal P}^*\O_{\overline \M^{\rm BB}_{(n,d)}}(1)\right]^{**} \quad \text{for}\quad \alpha\in\Q.
$$
From the construction of $\overline \M^{\rm BB}_{(n,d)}$ (see Theorem \ref{thm:BFMT BB}) and the computation of Mordant \cite[Theorem 1.1]{Mordant25}
$$
\mathrm{ht}_{GK}(\mathbb H^{n-1}(\X_\eta/\P^1_\eta))=F(n,d)\,\mathrm{ht}_{GIT}(\X_\eta)
$$
for a general pencil, we obtain
$$
\alpha=F(n,d).
$$
Hence, Lemma \ref{lem:Mordant Q1 numerical criterion} proves the first part of the theorem.

The second part of the theorem follows from Lemma \ref{lem:Mordant Q2} and the local Torelli theorem \cite[Theorem 9.8]{Griffiths69}. Indeed, denote by $\M^{\rm sm}_{(n,d)}$ the moduli space of smooth hypersurfaces. Then by the local Torelli, the period map restricted to $\M^{\rm sm}_{(n,d)}$ is an open embedding upon taking a Stein factorization of $\overline {\mathcal P}$. Therefore, Lemma \ref{lem:Mordant Q2} applies to the Stein factorization of $\overline {\mathcal P}$ with respect to the open subset $\M^{\rm sm}_{(n,d)}$, proving the existence of the desired positive constant $\epsilon(n,d)$.
\end{proof}

We record a numerical criterion for the extension of the period map in a form that is more suitable for explicit computation.

\begin{cor}
\label{cor:numerical criterion for pencil}
Let $(n,d)\neq(3,3)$, where $n\geq 2$ and $d\geq 3$. Let $X\subset \P^n$ be a degree $d$ hypersurface with a defining equation $\{P=0\}$. Then the following are equivalent:

\noindent
(1) The period map $\overline \Phi:\P |\O_{\P^n}(d)|\dashrightarrow\overline \M^{\rm BB}_{(n,d)}$ extends at $[X]\in\P |\O_{\P^n}(d)|$.

\noindent
(2) The hypersurface $X$ is GIT semistable and the period map $\overline {\mathcal P}:\overline\M^{\rm GIT}_{(n,d)}\dashrightarrow\overline \M^{\rm BB}_{(n,d)}$ extends at $[X]\in\overline\M^{\rm GIT}_{(n,d)}$.

\noindent 
(3) For a general degree $d$ hypersurface $\{Q=0\}\subset \P^n$, the pencil
$$
\X:=\{P+tQ=0\}\subset \P^n\times\P_t^1
$$
satisfies
$$
\mathrm{ht}_{GK}(\mathbb H^{n-1}(\X_\eta/\P^1_\eta))=F(n,d).
$$
\end{cor}

\begin{proof}
We first prove the equivalence between (1) and (3). As in the proof of Theorem \ref{thm:Mordant question and indeterminacy}, we have
$$
\O_{\P |\O_{\P^n}(d)|}(F(n,d))\isom_\Q \left[\overline \Phi^*\O_{\overline \M^{\rm BB}_{(n,d)}}(1)\right]^{**}
$$
upon taking the reflexive hull. Note that the compactified period map extends over the locus of nodal hypersurfaces by Theorem \ref{thm:extendability}. Since every fiber $\X_t$ is a nodal hypersurface when $t\neq0$, the moduli image
$$
\P^1_t\sm\{t=0\}\to \P |\O_{\P^n}(d)|
$$
is contained in the domain of definition of $\overline \Phi$. Therefore, by Lemma \ref{lem:Mordant Q1 numerical criterion}, the hypersurface $[X]$ is contained in the domain of definition of $\overline \Phi$ if and only if
$$
\mathrm{ht}_{GK}(\mathbb H^{n-1}(\X_\eta/\P^1_\eta))=F(n,d).
$$
It is immediate from the construction of a pencil that $\P^1_t$ is a line in $\P |\O_{\P^n}(d)|$, and thus, has degree $1$ with respect to the ample line bundle $\O(1)$ on $\P |\O_{\P^n}(d)|$.

Next, we prove the equivalence between (2) and (3). It is immediate that (2) implies (1), so it suffices to show that (3) implies (2). By \cite[Theorem 3.4]{Mordant25}, we have
$$
\mathrm{ht}_{GIT}(\X_\eta)\le \int_{\P^n\times\P^1_t}h_1^n\cdot(dh_1+h_2)
$$
where $h_1$ is a class of $\O_{\P^n}(1)$ and $h_2$ is a class of $\O_{\P^1_t}(1)$ pulled back to $\P^n\times\P^1_t$. In particular,
$$
\mathrm{ht}_{GIT}(\X_\eta)\le1
$$
and the equality holds if and only if every fiber $\X_t$ is GIT semistable. Additionally, Theorem~\ref{thm:Mordant question and indeterminacy} implies that
$$
\mathrm{ht}_{GK}(\mathbb H^{n-1}(\X_\eta/\P^1_\eta))\le F(n,d)\,\mathrm{ht}_{GIT}(\X_\eta)
$$
where the equality holds if and only if the moduli image of $\P^1_t$ in $\overline\M_{(n,d)}^{\rm GIT}$ is contained in the domain of definition of $\overline{\mathcal P}$. Therefore, the equality
$$
\mathrm{ht}_{GK}(\mathbb H^{n-1}(\X_\eta/\P^1_\eta))=F(n,d)
$$
implies that both preceding inequalities must be equalities, and thus, we have (2).
\end{proof}

\begin{rmk}
It is easy to check the following inequality
$$
\lfloor\frac{n}{2}\rfloor\ge \frac{n+1}{d}
$$
for $n\ge 2,d\ge 3$ and $(n,d)\neq (3,3)$. From the GIT stability criterion via minimal exponent \cite[Theorem A]{Park25}, every hypersurface with $(\lfloor\frac{n}{2}\rfloor-1)$-Du Bois singularities is GIT semistable for such $n$ and $d$. This can also be obtained as a consequence of Corollary \ref{cor:numerical criterion for pencil}.
\end{rmk}

\section{Limiting mixed Hodge structures and Griffiths heights}
\label{sec:LMHS and Griffiths}

Let $\mathbb V=(\mathbb V_\Z,F^\bullet, Q)$ be a polarized variation of $\Z$-Hodge structures on a punctured disk
$$
\Delta_t^*:=\Delta_t\sm\{t=0\}.
$$
Denote by $j:\Delta_t^*\to \Delta_t$ the open embedding. Schmid \cite[Theorem 6.16]{Schmid73} defined the limiting (complex) mixed Hodge structure of $\mathbb V$. Although \emph{loc. cit.} is stated in the setting of geometric families, the same construction applies to an abstract polarized variation of $\Z$-Hodge structure.

The construction of the limiting mixed Hodge structure is as follows. After taking a cyclic cover 
$$
\rho:\Delta_s\to \Delta_t
$$
with $t=s^m$ for a sufficiently divisible $m$, the pullback $\mathbb V_s:=\rho^*\mathbb V$ is a polarized variation of $\Z$-Hodge structure on $\Delta_s^*$ with a unipotent local monodromy. Denote by $\overline {\mathcal V}_s$ (resp. $\overline {\mathcal V}_t^{[0,1)}$) Deligne's (lower) canonical extension of $\mathbb V_s$ (resp. $\mathbb V$). Then the Hodge filtration extends to a filtration by subbundles
$$
F^\bullet \overline {\mathcal V}_s \quad\text{and}\quad F^\bullet\overline {\mathcal V}_t^{[0,1)}
$$
and the limiting (complex) mixed Hodge structure of $\mathbb V$ is defined by
$$
(\overline {\mathcal V}_s|_{0}, F^\bullet \overline {\mathcal V}_s|_{0}, W)
$$
where $\overline {\mathcal V}_s|_{0}$ and $F^\bullet \overline {\mathcal V}_s|_{0}$ are the fibers of the bundles at $\{s=0\}$, and $W$ is the monodromy weight filtration induced by the action of the nilpotent operator $N:=\log T_u$ on this fiber. Here, $T_u$ is the unipotent part of the local monodromy. In this paper, the monodromy weight filtration will not play a role.

Alternatively, the limiting mixed Hodge structure is constructed using Saito's theory of mixed Hodge modules \cite{Saito88, Saito90}. The construction is intrinsic to $\mathbb V$, without taking a cyclic cover. By \cite[Theorem 3.21]{Saito90}, we have a Hodge module
$$
\mathbb V^H[1]\in {\rm MHM}(\Delta_t^*)
$$
associated to $\mathbb V$, and its open pushforward
$$
j_*\mathbb V^H[1]\in {\rm MHM}(\Delta_t).
$$
Applying the nearby cycle functor (see \cite[Theorem 0.1]{Saito90}), we obtain a rational mixed Hodge structure
$$
\psi_t\left(j_*\mathbb V^H[1]\right)\in {\rm MHM}(\{0\}).
$$
Here, the category of mixed Hodge modules on a point $\{0\}$ is equivalent to the category of (graded polarizable) rational mixed Hodge structures. It follows from Saito's construction of the nearby cycle functor that the underlying complex mixed Hodge structure agrees with Schmid's construction above.

Note that the underlying complex mixed Hodge structure admits a direct sum decomposition
$$
\psi_t\left(j_*\mathbb V^H[1]\right)=\bigoplus_{|\lambda|=1}\psi_{t,\lambda}\left(j_*\mathbb V^H[1]\right)
$$
into generalized eigenspaces with respect to the monodromy operator \cite[(2.2.6)]{Saito90}. The following lemma describes the discrepancy between $F^\bullet \overline {\mathcal V}_s$ and $F^\bullet\overline {\mathcal V}_t^{[0,1)}$ using the above decomposition.

This is known to some experts, but we include a proof for completeness. See, for example, \cite[Section 2]{EFiMM21} for a description in the setting of families of varieties over a punctured disk.

\begin{lem}
\label{lem:Deligne's canonical extension and finite cover}
In the above setting, we have a natural inclusion
$$
\rho^*F^p\overline {\mathcal V}_t^{[0,1)}\subset F^p\overline {\mathcal V}_s
$$
which is an isomorphism over $\Delta_s^*$. The dimension of the cokernel is
$$
\dim_\C \left[F^p\overline {\mathcal V}_s/\rho^*F^p\overline {\mathcal V}_t^{[0,1)}\right]
=\sum_{k=0}^{m-1}k\dim_\C F^p\left[\psi_{t,e^{-2\pi i\frac{k}{m}}}\left(j_*\mathbb V^H[1]\right)\right].$$
\end{lem}

\begin{proof}
Note that the cyclic group $G:=\Z/m\Z$ acts on $\Delta_s$ over $\Delta_t$. We choose a generator $g\in G$ so that $g(s)=\zeta s$ where $\zeta:=e^{\frac{2\pi i}{m}}$ and $s$ is a coordinate function. The pullback
$$
\mathbb V_s:=\rho^*\mathbb V
$$
has an invariant $G$-action over $\Delta_t$, and thus, Deligne's canonical extension
$$
(\overline {\mathcal V}_s, F^\bullet \overline {\mathcal V}_s)
$$
has an invariant $G$-action, and its fiber at $\{s=0\}$ inherits a $G$-action.

From the local analytic description of Deligne's lower canonical extension (see, for example, the proof of Lemma \ref{lem:Griffiths bundle alternative}), it is easy to obtain the following isomorphism of vector bundles
$$
\left(\rho_*\overline {\mathcal V}_s\right)^G\isom\overline {\mathcal V}_t^{[0,1)}
$$
where the left hand side is the sheaf of $G$-invariant sections of the pushforward $\rho_*\overline {\mathcal V}_s$. Since
$$
\left(\rho_*F^p\overline {\mathcal V}_s\right)^G \quad {\rm and}\quad F^p\overline {\mathcal V}_t^{[0,1)}
$$
are subbundles of $\left(\rho_*\overline {\mathcal V}_s\right)^G$ and $\overline {\mathcal V}_t^{[0,1)}$ respectively (i.e., each quotient sheaf is also a vector bundle), we obtain an isomorphism
\begin{equation}
\label{eqn:G-invariant canonical extension}
\left(\rho_*F^p\overline {\mathcal V}_s\right)^G \isom F^p\overline {\mathcal V}_t^{[0,1)} 
\end{equation}
as subbundles. Indeed, they are naturally isomorphic over $\Delta_t^*$, and they both uniquely extend as subbundles. This establishes a natural inclusion
$$
\rho^*F^p\overline {\mathcal V}_t^{[0,1)}\subset F^p\overline {\mathcal V}_s.
$$

Under the identification of Saito's and Schmid's limiting mixed Hodge structures
$$
\psi_t\left(j_*\mathbb V^H[1]\right)\isom(\overline {\mathcal V}_s|_{0}, F^\bullet \overline {\mathcal V}_s|_{0}, W),
$$
the generalized $\lambda$-eigenspace for the local monodromy operator $T$ on the left hand side is equal to the $\lambda$-eigenspace for the action of the generator $g\in G$ on the right hand side. Denote by
$$
(\overline {\mathcal V}_s|_{0}, F^\bullet \overline {\mathcal V}_s|_{0}, W)_\lambda
$$
the $\lambda$-eigenspace with respect to the action of $g$. This is identified with the complex mixed Hodge structure
$$
\psi_{t,\lambda}\left(j_*\mathbb V^H[1]\right).
$$
When $\lambda=e^{-2\pi i\frac{k}{m}}$ for $0\le k\le m-1$, choose a basis
$$
\overline v_{k,1},\dots,\overline v_{k,a_k}
$$
of the $\lambda$-eigenspace $(F^p\overline {\mathcal V}_s|_{0})_\lambda=(F^p\overline {\mathcal V}_s/s\cdot F^p\overline {\mathcal V}_s)_\lambda$, and a lift
$$
v_{k,1},\dots, v_{k,a_k}
$$
as sections of $F^p\overline {\mathcal V}_s$. In particular, we have
$$
a_k=\dim_\C F^p\left[\psi_{t,e^{-2\pi i\frac{k}{m}}}\left(j_*\mathbb V^H[1]\right)\right].
$$
Let
$$
V_{k,i}:=\frac{1}{m}\sum_{b=0}^{m-1}\zeta^{bk}g^b(v_{k,i})
$$
be a section of $F^p\overline {\mathcal V}_s$ for all $1\le i\le a_k$. Then 
$$
g(V_{k,i})=\zeta^{-k}V_{k,i}
$$
and the image of $V_{k_i}$ in $F^p\overline {\mathcal V}_s|_{0}$ is $v_{k,i}$. In other words, for each $0\le k\le m-1$, the images of
$$
V_{k,1},\dots, V_{k,a_k}
$$
in $F^p\overline {\mathcal V}_s|_{0}$ form a basis of its $\zeta^{-k}$-eigenspace. Hence, by Nakayama's Lemma, we have the local isomorphism
$$
F^p\overline {\mathcal V}_s=\bigoplus_{k=0}^{m-1}\bigoplus_{i=1}^{a_k}V_{k,i}\O_{\Delta_s}
$$
near the origin $\{s=0\}$. Note that $g(s^jV_{k,i})=\zeta^{j-k}s^jV_{k,i}$. This implies that
$$
\left(\rho_*F^p\overline {\mathcal V}_s\right)^G=\left(\bigoplus_{k=0}^{m-1}\bigoplus_{i=1}^{a_k}V_{k,i}\O_{\Delta_s}\right)^G=\bigoplus_{k=0}^{m-1}\bigoplus_{i=1}^{a_k}s^kV_{k,i}\O_{\Delta_t}.
$$
From the isomorphism \eqref{eqn:G-invariant canonical extension}, we have 
$$
\dim_\C\left[F^p\overline {\mathcal V}_s/\rho^*F^p\overline {\mathcal V}_t^{[0,1)}\right]=\sum_{k=0}^{m-1}ka_k,
$$
which is the desired equality.
\end{proof}

This lemma allows us to compute the difference between the Griffiths height and the lower Griffiths height in terms of limiting mixed Hodge structures.

\begin{prop}
\label{prop:Griffiths height via LMHS}
Let $C$ be a smooth projective curve, and let $\mathbb V_\eta$ be a polarized variation of (effective) Hodge structures of weight $w$, defined on an open subset $C\sm\{c_1,\dots,c_s\}$ of $C$. Then
$$
\mathrm{ht}_{GK}(\mathbb V_\eta)-\mathrm{ht}_{GK,-}(\mathbb V_\eta)=\sum_{\gamma=1}^s\sum_{\alpha\in[0,1)}\sum_{p=1}^w\alpha\dim_\C F^p\left[\psi_{t_\gamma,e^{-2\pi i\alpha}}\left(j_*\mathbb V^H[1]\right)\right],
$$
where $t_\gamma$ is a local coordinate function near the point $c_\gamma$ and $j:C\sm\{c_1,\dots,c_s\}\to C$ is an open embedding.
\end{prop}

Here, \emph{effective} means that $\gr_F^{p}\mathcal V=0$ for $p<0$, where $F^\bullet$ is a Hodge filtration on $\mathcal V:=\mathbb V_\Z\otimes_\Z\O$. Note that the sum is taken over finitely many nonzero terms, since only finitely many $e^{-2\pi i\alpha}$ appear as eigenvalues of local monodromies.

\begin{proof}
Suppose the local monodromy at $c_\gamma$ becomes unipotent after taking a cyclic cover of degree $m$ for all $1\le\gamma\le s$. Choose a sufficiently ample line bundle $A\in \Pic(C)$ such that
$$
A^m(-c_1-\dots-c_s)
$$
is very ample. In particular, we have
$$
A^m\isom\O_C(c_1+\dots+c_s+R)
$$
for a reduced divisor $R$ disjoint from $c_1,\dots,c_s$. Consequently, consider the induced $m$-cyclic cover
$$
\rho:C'\to C
$$
branched over $c_1+\dots+c_s+R$. Write $\rho^{-1}(c_i)=\{c_i'\}$ and set $\mathbb V':=\rho^*\mathbb V$ on $C'\sm\{c_1',\dots,c_s'\}$. Let $\overline{\mathcal V}'$ be Deligne's canonical extension of $\mathbb V'$. Then we have
$$
\rho^*F^p\overline{\mathcal V}^{[0,1)}\subset F^p\overline{\mathcal V}'
$$
by Lemma \ref{lem:Deligne's canonical extension and finite cover}, and from the definition of Griffiths heights, we have
$$
\deg_{C'}\bigotimes_{p=1}^w \det\left(F^p\overline{\mathcal V}'\right)=m\,\mathrm{ht}_{GK}(\mathbb V_\eta)
$$
and
$$
\deg_{C'}\bigotimes_{p=1}^w \rho^* \det\left(F^p\overline{\mathcal V}^{[0,1)}\right)=m\,\mathrm{ht}_{GK,-}(\mathbb V_\eta).
$$
Therefore, applying Lemma  \ref{lem:Deligne's canonical extension and finite cover} near each point $c_\gamma$ for $1\le\gamma\le s$, we have
$$
m\left(\mathrm{ht}_{GK}(\mathbb V_\eta)-\mathrm{ht}_{GK,-}(\mathbb V_\eta)\right)=\sum_{\gamma=1}^s\sum_{k=0}^{m-1}\sum_{p=1}^w k\dim_\C F^p\left[\psi_{t_\gamma,e^{-2\pi i\frac{k}{m}}}\left(j_*\mathbb V^H[1]\right)\right]
$$
which is equivalent to the desired equality.
\end{proof}

\section{Griffiths heights for pencils with isolated singularities}
\label{sec:pencil with isolated singularities}

Given degree $d$ homogeneous polynomials $P,Q\in\O_{\P^n}(d)$, consider a pencil
$$
\X:=\{P+tQ=0\}\subset \P^n\times\P_t^1.
$$
In this section, we provide a formula to compute the Griffiths height
$$
\mathrm{ht}_{GK}(\mathbb H^{n-1}(\X_\eta/\P^1_\eta))
$$
when $\X$ is smooth. Note that this smoothness assumption on $\X$ implies that every fiber $\X_t$ has isolated singularities. Indeed, the singular locus of $\X$ is determined by the vanishing of the Jacobian of the equation $P+tQ=0$. Hence, if a fiber $\X_0$ has a positive dimensional singular locus, then $\X$ is singular along 
$$
\Sing(\X_0)\cap\{Q=0\}\subset \X_0\subset \X.
$$

Denote by $S:=\Sing(\X/\P^1)$ the set of critical points of the morphism
$$
f:\X\to \P^1,
$$
that is, the set of singular points $\Sing(\X_t)$ for all $t\in \P^1$. For $x\in S\subset \X$, let $t$ be the local coordinate function of a disk $\Delta_t\subset\P^1$ centered at $f(x)$, and let
$$
f_t:\X_{\Delta_t}\to \Delta_t
$$
be the restriction of the morphism $f$ to $\X_{\Delta_t}:=f^{-1}(\Delta_t)$. Denote by
\begin{equation}
\label{eqn:V_lambda}
V_\lambda(x):=\phi_{f_t,\lambda}\left(\Q^H_\X[n]\right)|_{\{x\}}
\end{equation}
the complex mixed Hodge structure given by the generalized $\lambda$-eigenspace of the vanishing cycle mixed Hodge structure at $\{x\}$ 
$$
\phi_{f_t}\left(\Q^H_\X[n]\right)|_{\{x\}}\in {\rm MHM}(\{x\})
$$
with respect to the local monodromy operator at $f(x)$. Recall that the category of mixed Hodge modules on a closed point is equivalent to the category of (graded polarizable) rational mixed Hodge structures, and the vanishing cycle $\phi_{f_t}\left(\Q^H_\X[n]\right)$ is supported on the finite set of critical points.

In other words, the complex mixed Hodge structure $V_\lambda(x)$ is the generalized $\lambda$-eigenspace of the middle cohomology of the Milnor fiber of $f_t$ near $x$, with respect to the local monodromy operator.

Applying Lemma \ref{lem:Mordant Q1 numerical criterion} to the period map 
$$
\overline \Phi:\P |\O_{\P^n}(d)|\dashrightarrow\overline \M^{\rm BB}_{(n,d)}
$$
(see the proof of Corollary \ref{cor:numerical criterion for pencil}), we obtain
$$
\mathrm{ht}_{GK}(\mathbb H^{n-1}(\X_\eta/\P^1_\eta))\le F(n,d).
$$
The following theorem allows us to compute the difference between the two sides of the above inequality, in terms of the cohomology of the Milnor fiber at each critical point $x\in S$.

\begin{thm}
\label{thm:Griffiths height of pencil}
Let $(n,d)\neq(3,3)$, where $n\geq 2$ and $d\geq 3$. In the setting above, let
$$
A_{\alpha}(p):=(n-1)+\dots+(n-p)+\alpha p
$$
and
$$
h_{\lambda}^p(x):=\dim_\C\gr_F^pV_\lambda(x)
$$
for an integer $p\ge 0$. In other words, $h_{\lambda}^p(x)$ is the dimension of the $p$-th graded Hodge piece of the generalized $\lambda$-eigenspace of the middle cohomology of the Milnor fiber at $x$. Then the difference
$$
\mathrm{ht}_{GK}(\mathbb H^{n-1}(\X_\eta/\P^1_\eta))-F(n,d)
$$
between the Griffiths height and $F(n,d)$ is equal to
$$
\sum_{x\in S}\sum_{\alpha\in[0,1)}\sum_{p\in \Z_{\ge 0}} A_{\alpha}(p)\,h_{e^{-2\pi i\alpha}}^p(x)-\frac{(n+1)(d-1)^n}{2} \lfloor\frac{n(3n-2)}{4}\rfloor
$$
where $S:=\Sing(\X/\P^1)$ is the set of critical points of $f:\X\to \P^1$.
\end{thm}

The effective range of $p$ is $0\le p\le n-1$, since $h_{\lambda}^p(x)=0$ for $p\ge n$ (see Theorem~\ref{thm:minimal exponent via vanishing cycle} and apply the positivity of the minimal exponent).

In the rest of this section, denote by $\mathbb V_\eta:=\mathbb H^{n-1}(\X_\eta/\P^1_\eta)$ the variation of Hodge structures on $\P^1\sm f(S)$. The proof of Theorem \ref{thm:Griffiths height of pencil} follows from several intermediate lemmas, which we prove next.

\begin{lem}
\label{lem:Griffiths minus lower for pencil}
In the setting of Theorem \ref{thm:Griffiths height of pencil}, we have
$$
\mathrm{ht}_{GK}(\mathbb V_\eta)-\mathrm{ht}_{GK,-}(\mathbb V_\eta)=\sum_{x\in S}\sum_{\alpha\in[0,1)}\sum_{p=0}^{n-1}\alpha p\, h_{e^{-2\pi i\alpha}}^p(x).
$$
\end{lem}

\begin{proof}
By Proposition \ref{prop:Griffiths height via LMHS}, we have
$$
\mathrm{ht}_{GK}(\mathbb V_\eta)-\mathrm{ht}_{GK,-}(\mathbb V_\eta)=\sum_{\gamma\in f(S)}\sum_{\alpha\in[0,1)}\sum_{p=1}^{n-1}\alpha\dim_\C F^p\left[\psi_{t_\gamma,e^{-2\pi i\alpha}}\left(j_*\mathbb V^H[1]\right)\right],
$$
where $j:\P^1\sm f(S)\to \P^1$ is an open embedding and $t_\gamma$ is a local coordinate function near each point $\gamma\in f(S)$. In other words, we have a disk $\Delta_{t_\gamma}\subset \P^1$ centered at the point $\gamma$. Note that
$$
\sum_{p=1}^{n-1}\alpha\dim_\C F^p\left[\psi_{t_\gamma,e^{-2\pi i\alpha}}\left(j_*\mathbb V^H[1]\right)\right]=\sum_{p=0}^{n-1}\alpha p\dim_\C \gr_F^p\left[\psi_{t_\gamma,e^{-2\pi i\alpha}}\left(j_*\mathbb V^H[1]\right)\right],
$$
and thus, it suffices to prove
\begin{equation}
\label{eqn:alpha h^p(x)}
\alpha\dim_\C \gr_F^p\left[\psi_{t_\gamma,e^{-2\pi i\alpha}}\left(j_*\mathbb V^H[1]\right)\right]=\sum_{x\in S\cap f^{-1}(\gamma)}\alpha\,h_{e^{-2\pi i\alpha}}^p(x). 
\end{equation}
When $\alpha=0$, this is immediate. Suppose $\alpha\in (0,1)$. By the commutativity of the nearby and vanishing cycle functors with projective pushforwards (see \cite[Theorem 2.14]{Saito90} and \cite[Proposition 3.3.17]{Saito88}), we have
\begin{align*}
\psi_{t_\gamma,e^{-2\pi i\alpha}}\left(j_*\mathbb V^H[1]\right)&=\psi_{t_\gamma,e^{-2\pi i\alpha}}(\mathcal H^0(f_*\Q_\X^H[n]))\\
&=\mathcal H^0\left(f_*\left(\psi_{f_t,e^{-2\pi i\alpha}}\left(\Q_\X^H[n]\right)\right)\right)
\end{align*}
where
$$
f_t:\X_{\Delta_{t_\gamma}}\to \Delta_{t_\gamma}
$$
is the restriction of the morphism $f$ to $\X_{\Delta_{t_\gamma}}:=f^{-1}(\Delta_{t_\gamma})$. Since $\alpha\in (0,1)$, we have
$$
\psi_{f_t,e^{-2\pi i\alpha}}\left(\Q_\X^H[n]\right)=\phi_{f_t,e^{-2\pi i\alpha}}\left(\Q_\X^H[n]\right)
$$
by the definition of nearby and vanishing cycle functors (see, for example, \cite[(2.1.6)]{Saito17a}). Since the set of critical points is finite, we have
$$
\phi_{f_t,e^{-2\pi i\alpha}}\left(\Q_\X^H[n]\right)=\bigoplus_{x\in S\cap f^{-1}(\gamma)}\phi_{f_t,e^{-2\pi i\alpha}}\left(\Q_\X^H[n]\right)|_{\{x\}}
$$
Therefore, we obtain \eqref{eqn:alpha h^p(x)} (see \eqref{eqn:V_lambda}), which completes the proof.
\end{proof}

\begin{lem}
\label{lem:sum of Milnor number}
In the setting of Theorem \ref{thm:Griffiths height of pencil}, denote by $\mu(x)$ the Milnor number of $f$ at $x\in \X$. Then we have
$$
\mu(x)=\sum_{\alpha\in[0,1)}\sum_{p\in \Z_{\ge 0}}h^p_{e^{-2\pi i\alpha}}(x)
$$
and
$$
\sum_{x\in S} \mu(x)=(n+1)(d-1)^n.
$$
\end{lem}

\begin{proof}
The first equality is immediate from the definition. Indeed, $\mu(x)$ is the dimension of the middle cohomology of the Milnor fiber of $f$ near $x$, and $V_\lambda(x)$ is its generalized $\lambda$-eigenspace. Since
$$
h_{\lambda}^p(x):=\dim_\C\gr_F^pV_\lambda(x)
$$
by the definition, the sum of $h_{\lambda}^p(x)$ is equal to $\mu(x)$.

It remains to prove the second equality. Denote by $Z\subset\P^n$ a general smooth hypersurface of degree $d$. By Ehresmann's fibration theorem, the pencil $f$ is a topological fibration away from the discriminant locus $f(S)$. Hence, by the Mayer--Vietoris sequence, we have
$$
e(\X)=e(Z)e(\P^1)+\sum_{\gamma\in f(S)}\left[e(\X_\gamma)-e(Z)\right]
$$
where $e(\,\cdot\,)$ is the topological Euler characteristic. From the nearby and vanishing distinguished triangle
$$
\Q_{\X_\gamma}^H[n-1]\to \psi_{f_t}\Q_{\X}^H[n]\to \phi_{f_t}\Q_\X^H[n]\xrightarrow{+1}
$$
(see, for example, \cite[Corollary 2.24]{Saito90}), we have
$$
e(Z)-e(\X_\gamma)=(-1)^{n-1}\sum_{x\in S\cap f^{-1}(\gamma)}\mu(x).
$$
In the above, $f_t$ is a restriction of $f$ to the neighborhood of the fiber $\X_\gamma$ in $\X$, mapping to a disk $\Delta_{t_\gamma}\subset \P^1$ centered at $\gamma$. This implies
$$
e(\X)-e(Z)e(\P^1)=(-1)^n\sum_{x\in S}\mu(x).
$$

Now, we consider a deformation of $\X$:
$$
\X(\epsilon):=\{P_\epsilon+tQ_\epsilon=0\}\subset \P^n\times\P_t^1
$$
where $P_\epsilon$ and $Q_\epsilon$ are general small deformations of $P$ and $Q$, respectively. Recall that a general pencil has $(n+1)(d-1)^n$ non-degenerate critical points. Indeed, the discriminant divisor in $\P|\O_{\P^n}(d)|$ is a degree $(n+1)(d-1)^n$ hypersurface and its general member has a single node. Hence, we have
$$
e(\X(\epsilon))-e(Z)e(\P^1)=(-1)^n(n+1)(d-1)^n.
$$
By Ehresmann's fibration theorem, $\X(\epsilon)$ is homeomorphic to $\X$, and thus, we have
$$
e(\X(\epsilon))=e(\X).
$$
This implies that
$$
\sum_{x\in S}\mu(x)=(n+1)(d-1)^n
$$
as desired.
\end{proof}

\begin{lem}
\label{lem:Griffiths lower for pencil}
There exists a function $A(n,d)$ which only depends on the pair $(n,d)$ such that
$$
\mathrm{ht}_{GK,-}(\mathbb V_\eta)=A(n,d)+\sum_{x\in S}\sum_{\alpha\in[0,1)}\sum_{p=0}^{n-1}\left[(n-1)+\dots+(n-p)\right] h^p_{e^{-2\pi i\alpha}}(x).
$$
\end{lem}

\begin{proof}
By Lemma \ref{lem:sum of Milnor number}, the set $f(S)$ has at most $\delta:=(n+1)(d-1)^n$ points. Let $D\subset \P^1$ be a set of $\delta$ points containing $f(S)$. Denote by $\X_D:=f^{-1}(D)$, and consider a resolution
$$
\eta:\wt \X\to\X,
$$
which is an isomorphism over $\X\sm S$ and $E:=(f\circ\eta)^{-1}(D)$ is a simple normal crossing divisor. Set $\wt f:=f\circ\eta$.

As in the proof of Proposition \ref{prop:Hodge pieces for higher Du Bois family}, consider a short exact sequence
$$
0\to \wt f^*\omega_{\P^1}(D)\to \Omega_{\wt \X}(\log E)\to \Omega_{\wt\X/\P^1}(\log E)\to 0.
$$
By Steenbrink \cite{Steenbrink75} (see also Zucker \cite{Zucker84}), we have
$$
\gr_F^p\overline{\mathcal V}^{[0,1)}_k\isom R^{k-p}\wt f_*\Omega_{\wt\X/\P^1}^p(\log E)
$$
for all $p\ge 0$, where $\overline{\mathcal V}^{[0,1)}_k$ is Deligne's lower canonical extension of $\mathbb H^{k}(\X_\eta/\P^1_\eta)$.

Note that, since fibers are connected, the variation of Hodge structures $\mathbb H^{0}(\X_\eta/\P^1_\eta)$ on $\P^1\sm f(S)$ is isomorphic to the trivial variation
$$
\Q_{\P^1\sm f(S)}.
$$
By the relative hard Lefschetz theorem, when $2k\neq n-1$, the variation $\mathbb H^{2k}(\X_\eta/\P^1_\eta)$ is isomorphic to
$$
\Q_{\P^1\sm f(S)}(-k),
$$
that is, the $(-k)$-Tate twist of the trivial variation. Additionally, when $2k+1\neq n-1$, we have
$$
\mathbb H^{2k+1}(\X_\eta/\P^1_\eta)=0.
$$
This implies that
\begin{equation}
\label{eqn:degree of Hodge piece as derived pushforward}
\deg \left[R\wt f_*\Omega_{\wt\X/\P^1}^p(\log E)\right]=(-1)^{n-1-p}\deg \left[R^{n-1-p}\wt f_*\Omega_{\wt\X/\P^1}^p(\log E)\right].    
\end{equation}
Note that the degree is well defined in the Grothendieck group of $\P^1$.

Now, we consider the following Koszul resolution of $\Omega_{\wt\X/\P^1}^p(\log E)$:
$$
0\to (\wt f^*\omega_{\P^1}(D))^p\to\dots\to \wt f^*\omega_{\P^1}(D)\otimes\Omega^{p-1}_{\wt\X}(\log E)\to \Omega^{p}_{\wt\X}(\log E)\to \Omega_{\wt\X/\P^1}^p(\log E)\to 0.
$$
This implies that
$$
\deg \left[R\wt f_*\Omega_{\wt\X/\P^1}^p(\log E)\right]=\sum_{i=0}^p(-1)^{p-i}\deg\left[R\wt f_*\Omega^i_{\wt \X}(\log E)\otimes\O_{\P^1}((p-i)(\delta-2))\right].
$$
Note that $R\wt f_*\Omega^i_{\wt \X}(\log E)=Rf_*R\eta_*\Omega^i_{\wt \X}(\log E)$, and we have a distinguished triangle
$$
\Omega^i_{\X}\to R\eta_*\Omega^i_{\wt \X}(\log E)\to R\Hom_{\O_{\X_D}}(\DB_{\X_D}^{n-i},\omega_{\X_D})\xrightarrow{+1}.
$$
Indeed, this is easily induced by applying the graded de Rham functor $\gr_F^i\DR(\,\cdot\,)$ to the distinguished triangle
$$
\iota_*\iota^!\Q^H_\X\to\Q_\X^H\to j_*\Q^H_{\X\sm\X_D}\xrightarrow{+1}
$$
via Saito's theory of mixed Hodge modules, where $\iota:\X_D\to \X$ and $j:\X\sm\X_D\to \X$ are closed and open embeddings. Additionally, we have the short exact sequence of vector bundles
$$
0\to\O_{\X}(-\X)\to \Omega_{\P^n\times\P^1}|_{\X}\to\Omega_\X\to 0,
$$
and the classes $[\O_{\X}(-\X)]$ and $[\Omega_{\P^n\times\P^1}|_{\X}]$ in the Grothendieck group of $\P^n\times\P^1$ are independent of the choice of a pencil $\X$. Consequently, the degree
$$
\deg\left[Rf_*\Omega^i_{\X}\otimes\O_{\P^1}((p-i)(\delta-2))\right]
$$
is independent of the choice of a pencil $\X$. Hence, as a result of Grothendieck--Serre duality, we have
$$
\deg \left[R\wt f_*\Omega_{\wt\X/\P^1}^p(\log E)\right]=A+\sum_{i=0}^p(-1)^{p-i}(-1)^{n-1}e(\DB_{\X_D}^{n-i})
$$
for some constant $A$, independent of the choice of $\X$. Here, $e(\DB_{\X_D}^{n-i})$ is the Euler characteristic of the cohomology of the coherent complex $\DB_{\X_D}^{n-i}$.

From \eqref{eqn:degree of Hodge piece as derived pushforward}, we have
$$
\mathrm{ht}_{GK,-}(\mathbb V_\eta)=\sum_{p=0}^{n-1}(-1)^{n-1-p}p\deg\left[R\wt f_*\Omega_{\wt\X/\P^1}^p(\log E)\right],
$$
and thus, from the argument above, we have
\begin{align*}
\mathrm{ht}_{GK,-}(\mathbb V_\eta)&=A'+\sum_{p=0}^{n-1}\sum_{i=0}^p(-1)^ip\,e(\DB_{\X_D}^{n-i})\\&=A'+\sum_{p=0}^{n-1}(-1)^{n-p}\left[(n-1)+\dots+(n-p)\right]e(\DB_{\X_D}^{p})
\end{align*}
for some constant $A'$, independent of the choice of $\X$. For a general smooth hypersurface $Z\subset \P^n$ of degree $d$, we may therefore write
$$
\mathrm{ht}_{GK,-}(\mathbb V_\eta)=A(n,d)+\sum_{p=0}^{n-1}(-1)^{n-p}\left[(n-1)+\dots+(n-p)\right]\left(e(\DB_{\X_D}^{p})-\delta\,e(\Omega_{Z}^{p})\right)
$$
for some constant $A(n,d)$. It suffices to prove
\begin{equation}
\label{eqn:difference of euler char of DB}
(-1)^{n-p}\left(e(\DB_{\X_\gamma}^{p})-e(\Omega_{Z}^{p})\right)=\sum_{x\in S\cap f^{-1}(\gamma)}\sum_{\alpha\in[0,1)} h^p_{e^{-2\pi i\alpha}}(x).  
\end{equation}
Recall that we have the equalities
$$
\gr_F^p\DR(\Q_{\X_\gamma}^H)=\DB_{\X_\gamma}^p[-p], \quad \gr_F^p\DR(\phi_{f_t}\Q_\X^H[n])=\bigoplus_{x\in S\cap f^{-1}(\gamma)}\bigoplus_{|\lambda|=1}\gr_F^pV_\lambda(x).
$$
Applying the graded de Rham functor $\gr_F^p\DR(\,\cdot\,)$ to the distinguished triangle
$$
\Q_{\X_\gamma}^H[n-1]\to \psi_{f_t}\Q_{\X}^H[n]\to \phi_{f_t}\Q_\X^H[n]\xrightarrow{+1}
$$
appearing in the proof of Lemma \ref{lem:sum of Milnor number}, we obtain \eqref{eqn:difference of euler char of DB} and complete the proof.
\end{proof}

Combining the above lemmas, we prove Theorem \ref{thm:Griffiths height of pencil}.

\begin{proof}[Proof of Theorem \ref{thm:Griffiths height of pencil}]
Combining Lemma \ref{lem:Griffiths minus lower for pencil} and Lemma \ref{lem:Griffiths lower for pencil}, we have
$$
\mathrm{ht}_{GK}(\mathbb H^{n-1}(\X_\eta/\P^1_\eta))=A(n,d)+\sum_{x\in S}\sum_{\alpha\in[0,1)}\sum_{p=0}^{n-1}A_\alpha(p)\, h^p_{e^{-2\pi i\alpha}}(x).
$$

From the classical computation of the vanishing cycle $\phi_f$ for nodal singularities
$$
\{f:=x_1^2+\dots+x_n^2=0\}\subset\C^n,
$$
we have
$$
\phi_f=\Q\left(-\frac{n-1}{2}\right)
$$
if $n$ is odd with eigenvalue $(-1)$, and
$$
\phi_f=\Q\left(-\frac{n}{2}\right)
$$
if $n$ is even with eigenvalue $1$. Indeed, the Milnor fiber is homotopy equivalent to an $(n-1)$-sphere, and thus, the vanishing cycle is $1$-dimensional and of Hodge--Tate type. Since the minimal exponent of $f$ is $\frac{n}{2}$, we can easily conclude the above from Theorem \ref{thm:minimal exponent via vanishing cycle}.

Consider a general pencil $\mathcal Y\to \P^1$. Then the set of critical points consists of $(n+1)(d-1)^n$ nodal singularities. Therefore, when $n$ is odd, we have
$$
\mathrm{ht}_{GK}(\mathbb H^{n-1}(\mathcal Y_\eta/\P^1_\eta))=A(n,d)+(n+1)(d-1)^n\left((n-1)+\dots+(n-\frac{n-1}{2})+\frac{n-1}{4}\right),
$$
and when $n$ is even, we have
$$
\mathrm{ht}_{GK}(\mathbb H^{n-1}(\mathcal Y_\eta/\P^1_\eta))=A(n,d)+(n+1)(d-1)^n\left((n-1)+\dots+(n-\frac{n}{2})\right).
$$
This can be alternatively written as
$$
\mathrm{ht}_{GK}(\mathbb H^{n-1}(\mathcal Y_\eta/\P^1_\eta))=A(n,d)+\frac{(n+1)(d-1)^n}{2} \lfloor\frac{n(3n-2)}{4}\rfloor,
$$
and \cite[Theorem 1.1]{Mordant25} implies $\mathrm{ht}_{GK}(\mathbb H^{n-1}(\mathcal Y_\eta/\P^1_\eta))=F(n,d)$. Therefore, we obtain the desired formula.
\end{proof}

The formula in Theorem \ref{thm:Griffiths height of pencil} is used to determine the domain of definition of the period map over the locus of hypersurfaces with isolated singularities. We first define some additional invariants associated to the mixed Hodge structure $V_\lambda(x)$:
$$
h_\lambda^{p,q}(x):=\dim_\C\gr_F^p\gr^W_{p+q}V_\lambda(x).
$$
In particular, the sum of $h_\lambda^{p,q}(x)$ over $q\in \Z_{\ge 0}$ is $h_\lambda^{p}(x)$. This invariant satisfies interesting symmetries.

\begin{lem}
\label{lem:symmetry of h_lambda}
When $\lambda\neq 1$, we have
$$
h_\lambda^{p,q}(x)=h_{\overline\lambda}^{q,p}(x)=h_\lambda^{n-1-q,n-1-p}(x)=h_{\overline\lambda}^{n-1-p,n-1-q}(x).
$$
When $\lambda=1$, we have
$$
h_1^{p,q}(x)=h_1^{q,p}(x)=h_1^{n-q,n-p}(x)=h_1^{n-p,n-q}(x).
$$
\end{lem}

\begin{proof}
If the local monodromy operator is denoted by $T$, then the operator
$$
(T-\lambda)(T-\overline\lambda)
$$
has real coefficients. This implies that $V_\lambda(x)\oplus V_{\overline \lambda}(x)$ is a real mixed Hodge structure with
$$
\overline{V_\lambda(x)}=V_{\overline \lambda}(x).
$$
Hence, by the definition of mixed Hodge structure, we have
$$
\overline{\gr_F^p\gr^W_{p+q}\left(V_\lambda(x)\oplus V_{\overline \lambda}(x)\right)}=\gr_F^q\gr^W_{p+q}\left(V_\lambda(x)\oplus V_{\overline \lambda}(x)\right),
$$
which implies
$$
\overline{\gr_F^p\gr^W_{p+q}\left(V_\lambda(x)\right)}=\gr_F^q\gr^W_{p+q}\left(V_{\overline \lambda}(x)\right) \quad\text{and}\quad h_\lambda^{p,q}(x)=h_{\overline\lambda}^{q,p}(x).
$$
By \cite[Section 5.1.6]{Saito88}, we additionally have the nilpotent operator
$$
N:V_\lambda(x)\to V_\lambda(x)(-1)
$$
where $(-1)$ denotes the Tate twist, with an isomorphism
$$
N^{n-1-p-q}:\gr^W_{p+q}V_\lambda(x)\xrightarrow{\sim}\gr^W_{2(n-1)-p-q}V_\lambda(x)(n-1-p-q)
$$
when $\lambda\neq1$ and an isomorphism
$$
N^{n-p-q}:\gr^W_{p+q}V_1(x)\xrightarrow{\sim}\gr^W_{2n-p-q}V_1(x)(n-p-q)
$$
when $\lambda=1$. Here, we assume $p+q\le n-1$ without loss of generality. As a consequence, we have
$$
h_\lambda^{p,q}(x)=h_\lambda^{n-1-q,n-1-p}(x)\quad\text{when}\quad \lambda\neq 1
$$
and
$$
h_1^{p,q}(x)=h_1^{n-q,n-p}(x)\quad\text{when}\quad \lambda= 1,
$$
from which we obtain the desired symmetry.
\end{proof}

\begin{cor}
\label{cor:indeterminacy for isolated singularities}
Let $(n,d)\neq(3,3)$, where $n\geq 2$ and $d\geq 3$. Let $X\subset \P^n$ be a degree $d$ hypersurface with isolated singularities. Then $[X]\in\P |\O_{\P^n}(d)|$ is contained in the domain of definition of the period map $$\overline \Phi:\P |\O_{\P^n}(d)|\dashrightarrow\overline \M^{\rm BB}_{(n,d)}$$ if and only if $X$ has $(\lfloor\frac{n}{2}\rfloor-1)$-Du Bois hypersurface singularities.
\end{cor}

\begin{proof}
By Theorem \ref{thm:extendability}, it suffices to prove ``only if."
Let $\{P=0\}$ be the defining equation of $X\subset \P^n$. By Corollary \ref{cor:numerical criterion for pencil}, it suffices to prove that for a general degree $d$ hypersurface $\{Q=0\}\subset \P^n$, the pencil 
$$
\X:=\{P+tQ=0\}\subset \P^n\times\P_t^1
$$
satisfies
$$
\mathrm{ht}_{GK}(\mathbb H^{n-1}(\X_\eta/\P^1_\eta))=F(n,d)
$$
only if $X$ has $(\lfloor\frac{n}{2}\rfloor-1)$-Du Bois hypersurface singularities. Away from $\X_0\isom X$, every fiber $\X_t$ is smooth or has one nodal singularity.

When $\alpha\neq 0$, denote by
$$
A_\alpha(p,q):=\frac{1}{4}\left(A_\alpha(p)+A_{1-\alpha}(q)+A_{1-\alpha}(n-1-p)+A_\alpha(n-1-q)\right),
$$
and by
$$
A_0(p,q):=\frac{1}{4}\left(A_0(p)+A_0(q)+A_0(n-p)+A_0(n-q)\right).
$$
By Lemma \ref{lem:symmetry of h_lambda}, we have
$$
\sum_{x\in S}\sum_{\alpha\in[0,1)}\sum_{(p,q)\in \Z^2_{\ge 0}} A_{\alpha}(p)\,h_{e^{-2\pi i\alpha}}^{p,q}(x)=\sum_{x\in S}\sum_{\alpha\in[0,1)}\sum_{(p,q)\in \Z^2_{\ge 0}}A_\alpha(p,q)\,h_{e^{-2\pi i\alpha}}^{p,q}(x).
$$
Recall from Lemma \ref{lem:sum of Milnor number} that
$$
\sum_{x\in S}\sum_{\alpha\in[0,1)}\sum_{(p,q)\in \Z^2_{\ge 0}}h_{e^{-2\pi i\alpha}}^{p,q}(x)=(n+1)(d-1)^n.
$$
By Theorem \ref{thm:Griffiths height of pencil}, the difference
$$
\mathrm{ht}_{GK}(\mathbb H^{n-1}(\X_\eta/\P^1_\eta))-F(n,d)
$$
is equal to
$$
\sum_{x\in S}\sum_{\alpha\in [0,1)}\sum_{(p,q)\in \Z^2_{\ge 0}}\left(A_\alpha(p,q)-\frac{1}{2}\lfloor\frac{n(3n-2)}{4}\rfloor\right)h_{e^{-2\pi i\alpha}}^{p,q}(x).
$$

We further simplify the terms. Note that
$$
A_\alpha(p)=\frac{p(2n-p-1)}{2}+\alpha p.
$$
Thus, after expanding, we obtain
$$
A_\alpha(p,q)=\frac{1}{4}\left(n^2-1+p(n-p-1)+q(n-q-1)+(2\alpha-1)(p-q)\right)
$$
when $\alpha\neq 0$, and
$$
A_0(p,q)=\frac{1}{4}\left(n^2-n+p(n-p)+q(n-q)\right).
$$
We split the argument into two cases depending on the parity of $n$.

\textbf{Case 1.} $n$ is even. Then
$$
A_\alpha(p,q)\le \frac{1}{2}\lfloor\frac{n(3n-2)}{4}\rfloor
$$
and the equality holds if and only if $\alpha=0,p=q=\frac{n}{2}$. Indeed, when $\alpha\neq 0$, then from the inequality $|(2\alpha-1)(p-q)|<|p-q|$ when $p\neq q$, we have
$$
A_\alpha(p,q)<\frac{1}{4}\left(n^2-1+\frac{n^2}{4}+\frac{(n-2)^2}{4}\right)= \frac{1}{2}\lfloor\frac{n(3n-2)}{4}\rfloor.
$$
When $\alpha=0$, then $A_0(p,q)$ is maximized with value equal to $\frac{1}{2}\lfloor\frac{n(3n-2)}{4}\rfloor$ only when $p=q=\frac{n}{2}$. Therefore, if
$$
\mathrm{ht}_{GK}(\mathbb H^{n-1}(\X_\eta/\P^1_\eta))=F(n,d),
$$
then $h_{e^{-2\pi i\alpha}}^{p,q}(x)\neq0$ only if $\alpha=0,p=q=\frac{n}{2}$. In particular, this implies that the minimal exponent of $X$ satisfies
$$
\wt\alpha(X)\ge\frac{n}{2}
$$
by Theorem \ref{thm:minimal exponent via vanishing cycle}, which completes the proof. 

\textbf{Case 2.} $n$ is odd. Then
$$
A_\alpha(p,q)\le \frac{1}{2}\lfloor\frac{n(3n-2)}{4}\rfloor
$$
and the equality holds if and only if $\alpha\neq0,p=q=\frac{n-1}{2}$ or $\alpha=0, (p,q)=\left(\frac{n\pm 1}{2},\frac{n\pm 1}{2}\right)$. Indeed, when $\alpha\neq 0$, the inequality is equivalent to
$$
p(n-p-1)+q(n-q-1)+(2\alpha-1)(p-q)\le\frac{(n-1)^2}{2},
$$
which is equivalent to
$$
(2\alpha-1)(p-q)\le\left(\frac{n-1}{2}-p\right)^2+\left(\frac{n-1}{2}-q\right)^2.
$$
It is easy to verify that $|p-q|$ is greater (resp. less) than or equal to the left (resp. right) hand side. Hence, the equality holds if and only if $p=q=\frac{n-1}{2}$. When $\alpha=0$, $A_0(p,q)$ is maximized with value equal to $\frac{1}{2}\lfloor\frac{n(3n-2)}{4}\rfloor$ only when $(p,q)=\left(\frac{n\pm 1}{2},\frac{n\pm 1}{2}\right)$. Therefore, as in the argument above, the equality
$$
\mathrm{ht}_{GK}(\mathbb H^{n-1}(\X_\eta/\P^1_\eta))=F(n,d)
$$
implies that the minimal exponent of $X$ satisfies
$$
\wt\alpha(X)\ge\frac{n-1}{2}
$$
by Theorem \ref{thm:minimal exponent via vanishing cycle}, which completes the proof.
\end{proof}

\section{Reduction to isolated singularities}
\label{sec:reduction to isolated singularities}

In this section, we prove that every hypersurface that has singularities worse than $\left(\lfloor\frac{n}{2}\rfloor-1\right)$-Du Bois singularities can be approximated by those with isolated singularities. Using Corollary~\ref{cor:indeterminacy for isolated singularities}, this allows us to precisely determine the indeterminacy locus of the period map.

\begin{prop}
\label{prop:reduction to isolated sing}
Let $(n,d)\neq(3,3),(3,4),(5,3)$, where $n\geq 2$ and $d\geq 3$. Let $X\subset\P^n$ be a degree $d$ hypersurface satisfying
$$
\wt\alpha(X)<\lfloor\frac{n}{2}\rfloor.
$$
Then there exists a deformation
$$
f:\X\to \Delta
$$
of $X\isom\X_0$ as hypersurfaces in $\P^n$ such that the general fiber $\X_t$ of $f$ has isolated singularities and satisfies
$$
\wt\alpha(\X_t)<\lfloor\frac{n}{2}\rfloor.
$$
\end{prop}

\begin{rmk}
\label{rmk:reduction to isolated singularities}
The lower semicontinuity of the minimal exponent \cite[Theorem E(2)]{MP20} implies that the set of hypersurfaces $X\subset \P^n$ satisfying $\wt\alpha(X)<\lfloor\frac{n}{2}\rfloor$ is a Zariski closed subset $Z$ in the parameter space $\P|\O_{\P^n}(d)|$. Additionally, the set of hypersurfaces in $Z$ with isolated singularities is open in $Z$. Proposition \ref{prop:reduction to isolated sing} is equivalent to proving that the closure of this open set is all of $Z$.
\end{rmk}

We first prove the following technical lemma to be used later in the proof.

\begin{lem}
\label{lem:corank 2 analytic equivalent form}
Let $P\in \C[x_1,\dots,x_{n-2},z_1,z_2]$ be a polynomial such that
$$
P=x_1^2+\dots+x_{n-2}^2+R
$$
where $R\in (x_1,\dots,x_{n-2},z_1,z_2)^3$. We write
$$
R=F_0+\sum_{i=1}^{n-2}F_ix_i+\sum_{i,j=1}^{n-2}F_{ij}x_ix_j+\sum_{i,j,k=1}^{n-2}F_{ijk}x_ix_jx_k+\cdots
$$
where $F_I\in\C[z_1,z_2]$ for every index $I$. By convention, we assume $F_I=F_{I'}$, if $I$ and $I'$ differ by a permutation. Then, at the origin, $P$ is locally analytically equivalent to
$$
u_1^2+\dots+u_{n-2}^2+G+R'
$$
where
$$
G:=F_0-\frac{1}{4}\sum_{i=1}^{n-2}F_i^2+\frac{1}{4}\sum_{i,j=1}^{n-2}F_{ij}F_iF_j-\frac{1}{4}\sum_{i,j,k=1}^{n-2}F_{ij}F_{ik}F_jF_k-\frac{1}{8}\sum_{i,j,k=1}^{n-2}F_{ijk}F_iF_jF_k
$$
and $R'\in (u_1,\dots,u_{n-2})^2+(u_1,\dots,u_{n-2})(z_1,z_2)^4+(z_1,z_2)^7$ with multiplicity at least $3$ at the origin.
\end{lem}

\begin{proof}
Consider a change of coordinates
$$
u_i:=x_i+\frac{F_i}{2}-\sum_{j=1}^{n-2}\frac{F_{ij}F_j}{2}
$$
for $i\in \{1,\dots, n-2\}$. By substituting $u_i-\frac{F_i}{2}+\sum_{j=1}^{n-2}\frac{F_{ij}F_j}{2}$ for $x_i$ in the polynomial $R$, we obtain the desired form. We omit this elementary computation. Note that $F_i\in (z_1,z_2)^2$ and $F_{ij}\in (z_1,z_2)$, so $\{u_1,\dots,u_{n-2},z_1,z_2\}$ is a local coordinate system at the origin.
\end{proof}

The proof of Proposition \ref{prop:reduction to isolated sing} consists of lengthy but elementary calculations, involving a case-by-case analysis and a series of reductions. After constructing a suitable deformation, we make crucial use of the following key inequality on the minimal exponent, which is well known to experts (see also \cite[Proposition 4.3]{Park25}).

\begin{prop}[{\cite[Proposition 2.1]{CDM25}}]
\label{prop:weight-minimal exponent inequality}
Let $P\in \C[x_1,\dots,x_n]$ be a nonzero polynomial with $P(0)=0$, and let $w=(w_1,\dots,w_n)$ be a nonnegative nonzero weight system associated to $(x_1,\dots,x_n)$. If the hypersurface $\{P=0\}\subset \C^{n}$ is singular at $0$, then the minimal exponent of $P$ at $0$ satisfies
$$
\wt \alpha_0(P)\le\frac{w_1+\cdots+w_n}{\mathrm{wt}_{w}(P)}.
$$
\end{prop}

Here, the weight $\mathrm{wt}_w(P)$ of a polynomial $P$ is the minimum of $\sum_{i=1}^nw_ie_i$ for all monomials $x_1^{e_1}\cdots x_n^{e_n}$ appearing in $P$ with nonzero coefficients. 

The proof of Proposition \ref{prop:reduction to isolated sing} is divided into two cases, according to whether $n$ is even or odd. When $n$ is odd, we reduce to the case where the Hessian of a singular point is of corank $2$. Proved using Lemma~\ref{lem:corank 2 analytic equivalent form} and Proposition \ref{prop:weight-minimal exponent inequality}, the following two lemmas are crucial for handling this case.

\begin{lem}
\label{lem:P3=0}
Let $X\subset \P^n$ be a hypersurface of degree $d$ defined by an equation $\{P=0\}$. Let $x:=[1:0:\dots:0]\in X$ be a singular point such that the Hessian of $(X,x)$ is of rank $n-2$. As in Lemma \ref{lem:corank 2 analytic equivalent form}, suppose we have
$$
P(1,x_1,\dots,x_{n})=x_1^2+\dots+x_{n-2}^2+
\left(F_0+\sum_{i=1}^{n-2}F_ix_i+\sum_{i,j=1}^{n-2}F_{ij}x_ix_j+\sum_{i,j,k=1}^{n-2}F_{ijk}x_ix_jx_k+\cdots\right)
$$
where $F_I\in\C[x_{n-1},x_{n}]$ for every index $I$. Denote by $P_3\in \C[x_{n-1},x_{n}]$ the cubic term of $F_0$ and $P_4\in \C[x_{n-1},x_{n}]$ the quartic term of $F_0$. If $P_3=0$, then we have the following equivalence
$$
\wt\alpha_x(\{P=0\})<\frac{n-1}{2}\quad\Longleftrightarrow\quad P_4-\frac{1}{4}\sum_{i=1}^{n-2}Q_i^2\quad \text{has a triple root or is identically zero}
$$
where $Q_i\in \C[x_{n-1},x_{n}]$ is the quadratic term of $F_i$.
\end{lem}

\begin{proof}
Firstly, assume $\wt\alpha_x(X)<\frac{n-1}{2}$. Considering the weight system 
$$
w_1=\dots=w_{n-2}=\frac{1}{2}, w_{n-1}=\frac{1}{4}, w_n=\frac{1}{4},
$$
we obtain a $\mathbb G_m$-equivariant degeneration of the singularity $(X,x)$ to
$$
\sum_{i=1}^{n-2}x_i^2+\sum_{i=1}^{n-2}Q_ix_i+P_4=0.
$$
By the change of variables $u_i=x_i+\frac{Q_i}{2}$ for $i=1,\dots,n-2$, the above singularity is analytically equivalent to the singularity
$$
u_1^2+\dots+u_{n-2}^2+P_4-\frac{1}{4}\sum_{i=1}^{n-2}Q_i^2=0.
$$
The lower semicontinuity of the minimal exponent \cite[Theorem E(2)]{MP20} and the Thom--Sebastiani theorem (see Theorem \ref{thm:Thom-Sebastiani for minimal exponents}) imply
$$
\wt\alpha_0\left(P_4-\frac{1}{4}\sum_{i=1}^{n-2}Q_i^2\right)<\frac{1}{2}
$$
which is equivalent to the condition that $P_4-\frac{1}{4}\sum_{i=1}^{n-2}Q_i^2$ has a triple root, if nonzero.

Now, we prove the converse. Suppose that $P_4-\frac{1}{4}\sum_{i=1}^{n-2}Q_i^2$ is zero or has a triple root. By a linear change of coordinates within the variables $x_{n-1}$ and $x_n$, we may assume  
$$
P_4-\frac{1}{4}\sum_{i=1}^{n-2}Q_i^2=x_{n-1}^3x_n \quad \text{or} \quad x_{n-1}^4 \quad \text{or} \quad 0.
$$
Set $z_1:=x_{n-1}$ and $z_2:=x_n$ and apply Lemma \ref{lem:corank 2 analytic equivalent form}. For small enough $\epsilon>0$, consider a weight system
$$
w_1=\dots=w_{n-2}=\frac{1}{2}, w_{n-1}=\frac{1}{4}+\epsilon, w_n=\frac{1}{4}-3\epsilon
$$
associated to the variables $(u_1,\dots,u_{n-2},z_1,z_2)$. Since the weights of degree $\ge 5$ terms of $G$ are greater than $1$ and the degree $4$ term of $G$ (if not zero) is either $x_{n-1}^3x_n$ or $x_{n-1}^4$, we obtain
$$
\wt\alpha_x(X)\le \frac{n-1}{2}-2\epsilon
$$
by Proposition \ref{prop:weight-minimal exponent inequality} as desired.
\end{proof}

\begin{lem}
\label{lem:P3=x3^3}
In the setting of Lemma \ref{lem:P3=0}, denote by
$$
G(x_{n-1},x_n):=F_0-\frac{1}{4}\sum_{i=1}^{n-2}F_i^2+\frac{1}{4}\sum_{i,j=1}^{n-2}F_{ij}F_iF_j-\frac{1}{4}\sum_{i,j,k=1}^{n-2}F_{ij}F_{ik}F_jF_k-\frac{1}{8}\sum_{i,j,k=1}^{n-2}F_{ijk}F_iF_jF_k.
$$
If $P_3=x_{n-1}^3$, then we have the following equivalence
$$
\wt\alpha_x(\{P=0\})<\frac{n-1}{2}\quad\Longleftrightarrow\quad \mathrm{wt}_w(G)=1 \text{ and } G_{\mathrm{wt}_w=1}(x_{n-1},x_n)=(x_{n-1}+\alpha x_n^2)^3
$$
for some $\alpha\in \C$, with respect to the weight system $w:=(\frac{1}{3},\frac{1}{6})$ associated to $(x_{n-1},x_n)$. In other words, the right hand side means that $G$ does not contain nonzero monomials of weight $<1$, and the weight $1$ term $G_{\mathrm{wt}_w=1}(x_{n-1},x_n)$ consisting of monomials
$$x_{n-1}^3,x_{n-1}^2x_n^2,x_{n-1}x_n^4,x_n^6$$
is of the form $(x_{n-1}+\alpha x_n^2)^3$.
\end{lem}

\begin{proof}
Set $z_1:=x_{n-1}$ and $z_2:=x_n$. Applying Lemma \ref{lem:corank 2 analytic equivalent form} to $(X,x)$, this singularity is locally analytically equivalent to
\begin{equation}
\label{eqn:P3=x_n-1^3 reduction}
u_1^2+\dots+u_{n-2}^2+G(x_{n-1},x_n)+R'=0 
\end{equation}
where $R'\in (u_1,\dots,u_{n-2})^2+(u_1,\dots,u_{n-2})(x_{n-1},x_n)^4+(x_{n-1},x_n)^7$. Note that every term of $F_i$ (resp. $F_{ij}$) has degree $\ge 2$ (resp. $\ge 1$).

We first prove the backward implication. Suppose the condition on the right hand side is satisfied. Consider a change of coordinates
$$
y_{n-1}:=x_{n-1}+\alpha x_n^2,\quad y_n:=x_n,
$$
and a weight system
$$
w_1=\dots=w_{n-2}=\frac{1}{2}, w_{n-1}=\frac{1}{3}, w_n=\frac{1}{6}-\epsilon
$$
associated to the variables $(u_1,\dots,u_{n-2},y_{n-1},y_n)$ for small enough $\epsilon>0$. Then by Proposition~\ref{prop:weight-minimal exponent inequality}, we obtain
$$
\wt\alpha_x(X)\le \frac{n-1}{2}-\epsilon
$$
as desired.

For the forward implication, suppose $G$ contains nonzero monomials
$$
x^I:=x_n^4 \quad\text{or}\quad x_{n-1}x_n^3\quad \text{or}\quad x_n^5.
$$
Note that they are the only possible nonzero monomials of $G$ with weight $<1$ with respect to the weight system $w:=(\frac{1}{3},\frac{1}{6})$. Considering the weight system 
$$
w_1=\dots=w_{n-2}=\frac{1}{2}, w_{n-1}=\frac{1}{3}, w_n=\frac{1}{4} \quad\text{or}\quad \frac{2}{9}\quad \text{or}\quad \frac{1}{5},
$$
we obtain a $\mathbb G_m$-equivariant degeneration of the singularity \eqref{eqn:P3=x_n-1^3 reduction} to
$$
u_1^2+\dots+u_{n-2}^2+x_{n-1}^3+cx^I=0
$$
for some nonzero $c\in \C$. By the lower semicontinuity and the Thom--Sebastiani theorem, this implies that
$$
\wt\alpha_x(X)\ge\frac{n-1}{2}.
$$
Hence, if the minimal exponent is $<\frac{n-1}{2}$, then we have $\mathrm{wt}_w(G)=1$. Likewise, considering the weight system
$$
w_1=\dots=w_{n-2}=\frac{1}{2}, w_{n-1}=\frac{1}{3}, w_n=\frac{1}{6},
$$
we obtain a $\mathbb G_m$-equivariant degeneration to
$$
u_1^2+\dots+u_{n-2}^2+G_{\mathrm{wt}_w=1}(x_{n-1},x_n)=0.
$$
Note that $G_{\mathrm{wt}_w=1}(x_{n-1},x_n)$ is a cubic polynomial in $x_{n-1}$ and $x_n^2$, with the leading term $x_{n-1}^3$. Hence, we have a factorization
$$
G_{\mathrm{wt}_w=1}(x_{n-1},x_n)=(x_{n-1}+\alpha x_n^2)(x_{n-1}+\beta x_n^2)(x_{n-1}+\gamma x_n^2)
$$
for some $\alpha,\beta,\gamma\in \C$. Again, by the lower semicontinuity and the Thom--Sebastiani theorem, this has the minimal exponent $<\frac{1}{2}$, or equivalently, the log canonical threshold $<\frac{1}{2}$. One can easily check that
$$
\wt\alpha_0(x_{n-1}^2(x_{n-1}+x_n^2))=\frac{1}{2},
$$
which is the most degenerate case up to a change of variables, unless $\alpha=\beta=\gamma$. Hence, we conclude $\alpha=\beta=\gamma$, as desired.
\end{proof}

We now prove Proposition \ref{prop:reduction to isolated sing}.

\begin{proof}[Proof of Proposition \ref{prop:reduction to isolated sing}]
Without loss of generality, let $x:=[1:0:\dots:0]\in X$ be a singular point such that
$$
\wt\alpha_x(X)<\lfloor\frac{n}{2}\rfloor.
$$
Denote by $\{P=0\}$ the defining equation of $X\subset \P^n$.

When $n$ is even, $\wt\alpha_x(X)<\frac{n}{2}$ if and only if $(X,x)$ is not a nodal singularity, that is, the Hessian at $(X,x)$ is not of full rank (see Lemma \ref{lem:classification of n/2 Du Bois}). Consider the affine open chart $\{x_0\neq 0\}$ and write
$$
P(1,x_1,\dots,x_n)=P_2(x_1,\dots,x_n)+R
$$
where $P_2$ is the quadratic term of $P(1,x_1,\dots,x_n)$. For general $c_1,\dots,c_n\in \C$, the hypersurface
$$
X_c:=\{P+c_1x_1^d+\dots+c_nx_n^d=0\}\subset\P^n
$$
is smooth away from $x$ by the Bertini theorem, and the rank of its Hessian at $x$ is equal to the rank of the form $P_2$. Therefore, there exists a small deformation $X_c$ of $X$ such that
$$
\wt\alpha_x(X_c)<\frac{n}{2}.
$$

From now on, we assume $n$ is odd and
$$
\wt\alpha_x(X)<\frac{n-1}{2}.
$$
As above, consider the affine open chart $\{x_0\neq 0\}$ and write
$$
P(1,x_1,\dots,x_n)=P_2(x_1,\dots,x_n)+R.
$$
Let $r$ be the rank of the quadratic form $P_2$. By a linear change of coordinates within the variables $x_1,\dots,x_n$, we may assume
$$
P_2=x_1^2+\dots+x_r^2.
$$
Note that, by \cite[Theorem E(3)]{MP20}, we have
$$
\frac{r}{2}\le\wt\alpha_x(X)<\frac{n-1}{2},
$$
and thus, $r\le n-2$. If $r\le n-4$, then for general $c_1,\dots,c_n\in \C$, the hypersurface
$$
X_c:=\{P+c_1x_1^d+\dots+c_nx_n^d=0\}\subset\P^n
$$
is smooth away from $x$ by the Bertini theorem, and
$$
\wt\alpha_x(X_c)\le\frac{r}{2}+\frac{n-r}{3}<\frac{n-1}{2}.
$$
Indeed, the first inequality is obtained by applying Proposition \ref{prop:weight-minimal exponent inequality} to the weight system
\begin{equation}
\label{eqn:weight 1/2 and 1/3}
w_1=\dots=w_r=\frac{1}{2},w_{r+1}=\dots=w_n=\frac{1}{3}.  
\end{equation}
Hence, it remains to consider the cases $r=n-3$ and $r=n-2$.

When $r=n-3$, we write
$$
P(1,x_1,\dots,x_n)=x_1^2+\dots+x_{n-3}^2+P_3(x_{n-2},x_{n-1},x_n)+R
$$
where $R\in (x_1,\dots,x_n)^3$ does not contain nonzero degree $3$ monomials in $x_{n-2},x_{n-1},x_n$. Considering the weight system \eqref{eqn:weight 1/2 and 1/3}, we obtain a $\mathbb G_m$-equivariant degeneration of the singularity $(X,x)$ to
$$
x_1^2+\dots+x_{n-3}^2+P_3(x_{n-2},x_{n-1},x_n)=0,
$$
and the lower semicontinuity of the minimal exponent \cite[Theorem E(2)]{MP20} and the Thom--Sebastiani theorem (see Theorem \ref{thm:Thom-Sebastiani for minimal exponents}) imply
$$
\wt\alpha_0(P_3)<1.
$$
This is equivalent to the inequality $\wt\alpha(\{P_3=0\}\subset\P^2)<1$, where $\{P_3=0\}\subset\P^2$ is the associated plane cubic (see \cite[Theorem 6.1]{Park25}). From the classical analysis of plane cubics in $\P^2$ (see, for example, \cite[Figure 2.7]{3264}), this is only possible when $P_3$ (if not zero) is cuspidal, a union of a conic and a tangent line, or a union of three lines passing through a point. In other words, we have
$$
P_3=x_{n-2}^3+x_{n-1}^2x_n \quad \text{or}\quad x_{n-1}(x_{n-1}x_n+x_{n-2}^2)\quad \text{or}\quad P_3(x_{n-2},x_{n-1}) \quad \text{or}\quad 0
$$
by a linear change of coordinates within the variables $x_{n-2},x_{n-1},x_{n}$. For general $c_1,\dots,c_n\in \C$, the hypersurface
\begin{equation}
\label{eqn:x_1x_n^d-1}
X':=\{P+c_1x_1^d+\dots+c_{n-1}x_{n-1}^d+c_nx_1x_n^{d-1}=0\}\subset\P^n
\end{equation}
is smooth away from $x_1=\dots=x_{n-1}=0$ by the Bertini theorem, and the Jacobian criterion implies that
$$
\Sing(X')\subset\{x_1=\dots=x_{n-1}=0,\partial_{x_1}P+c_nx_n^{d-1}=0\}.
$$
Hence, $X'$ has isolated singularities. Additionally, we obtain
$$
\wt\alpha_x(X')\le\frac{n-1}{2}-\epsilon<\frac{n-1}{2}
$$
by applying Proposition \ref{prop:weight-minimal exponent inequality} to the weight system
$$
w_1=\dots=w_{n-3}=\frac{1}{2},w_{n-2}=\frac{1}{3},w_{n-1}=\frac{1}{3}+\epsilon,w_n=\frac{1}{3}-2\epsilon
$$
for small enough $\epsilon>0$. Therefore, the only remaining case is when $r=n-2$.

To summarize the above, we have reduced to the case where $n$ is odd and
$$
P(1,x_1,\dots,x_n)=x_1^2+\dots+x_{n-2}^2+P_3(x_{n-1},x_n)+R    
$$
where $P_3$ is a homogeneous cubic form in $x_{n-1}, x_n$, and $R\in (x_1,\dots,x_n)^3$ does not contain nonzero degree $3$ monomials in $x_{n-1},x_n$.
As above, from the weight system \eqref{eqn:weight 1/2 and 1/3}, we obtain a $\mathbb G_m$-equivariant degeneration of the singularity $(X,x)$ to
$$
x_1^2+\dots+x_{n-2}^2+P_3(x_{n-1},x_n)=0,
$$
and the lower semicontinuity and the Thom--Sebastiani theorem imply
$$
\wt\alpha_0(P_3)<\frac{1}{2}
$$
if $P_3\neq 0$. In particular, if $P_3\neq 0$, the subscheme $\{P_3=0\}\subset\P^1$ must be a triple point. Hence, through a linear change of coordinates within the variables $x_{n-1}$ and $x_n$, we may assume
$$
P_3=x_{n-1}^3 \quad\text{or}\quad 0.
$$
We treat these in two separate cases.

\textbf{Case 1.} $P_3=0$.

We apply Lemma \ref{lem:P3=0} and follow its notation. We may assume
$$
P_4-\frac{1}{4}\sum_{i=1}^{n-2}Q_i^2
$$
is a multiple of $x_{n-1}^3$, by a linear change of coordinates within the variables $x_{n-1}$ and $x_n$.

When $d\ge 4$, for general $c_1,\dots,c_n\in \C$, the hypersurface \eqref{eqn:x_1x_n^d-1} still satisfies the above divisibility condition. Thus we have $\wt\alpha_x(X')<\frac{n-1}{2}$ by Lemma \ref{lem:P3=0}. As in the argument following \eqref{eqn:x_1x_n^d-1}, we conclude that $X'$ has isolated singularities.

When $d=3$, the assumption implies that $n$ is an odd integer at least $7$. We additionally have $P_4=0$, and $\sum_{i=1}^{n-2}Q_i^2$ is a multiple of $x_{n-1}^3$. For simplicity, we write
$$
P(1,x_1,\dots,x_n)=\sum_{i=1}^{n-2}x_i^2+\sum_{i=1}^{n-2}x_iQ_i(x_{n-1},x_n)+R,
$$
where $R$ is a cubic in the ideal $(x_1,\dots,x_{n-2})^2$. For general $c_1,\dots,c_{n-2}\in \C$, the hypersurface
$$
X':=\{P+c_1x_1^3+\dots+c_{n-2}x_{n-2}^3=0\}\subset\P^n
$$
satisfies $\wt\alpha_x(X')<\frac{n-1}{2}$ by Lemma \ref{lem:P3=0}, and is smooth away from $x_1=\dots=x_{n-2}=0$ by the Bertini theorem. The Jacobian criterion implies that 
$$
\Sing(X')\subset\{x_1=\dots=x_{n-2}=0,Q_1=\dots=Q_{n-2}=0\},
$$
and thus, if $Q_1,\dots,Q_{n-2}$ do not have a common linear factor, then $\Sing(X')=\{x\}$. Hence, $X$ satisfies our conclusion unless $Q_1,\dots,Q_{n-2}$ have a common linear factor.

If $Q_1,\dots,Q_{n-2}$ have a common linear factor, we vary coefficients of $Q_1,\dots,Q_{n-2}$. Let $A$ be the affine parameter space of coefficients of $Q_1,\dots,Q_{n-2}$. In particular, $A$ has dimension $3(n-2)$. The condition that $\sum_{i=1}^{n-2}Q_i^2$ is a multiple of $x_{n-1}^3$ is given by the vanishing of the coefficients of $x_{n-1}^2x_n^2$, $x_{n-1}x_n^3$, and $x_n^4$. Hence, the subset $B$ satisfying this condition is given by three equations on $A$. In particular, every irreducible component of $B$ is of dimension $\ge 3(n-3)$.

On the other hand, consider the subset $C$ of $A$ such that $Q_1,\dots,Q_{n-2}$ have a common linear factor. When $Q_1$ is fixed, the space of quadrics that share a factor with $Q_1$ is of dimension $2$. Hence, the subspace $C$ has dimension $3+2(n-3)$. Note that
$$
3(n-3)>3+2(n-3),
$$
which implies that the closure of $B\setminus C$ is all of $B$. Since a hypersurface $\{P=0\}$ with the coefficients of the quadrics in the set $B\setminus C$ satisfies the desired conclusion, every hypersurface in $B$ satisfies the desired conclusion (see Remark \ref{rmk:reduction to isolated singularities}). This completes the proof for Case 1.

\textbf{Case 2.} $P_3=x_{n-1}^3$.

We apply Lemma \ref{lem:P3=x3^3} and follow its notation. In particular, the condition on the right hand side of the equivalence in Lemma \ref{lem:P3=x3^3} holds for $X$. Suppose $d\ge 4$. For general $c_1,\dots,c_{n-1}\in \C$, the hypersurface
$$
\{P+c_1x_1^d+\dots+c_{n-1}x_{n-1}^d=0\}\subset\P^n
$$
is smooth away from $x_1=\dots=x_{n-1}=0$, and the Jacobian criterion implies that this hypersurface has isolated singularities unless $x_{n-1}$ is a factor of $F_i$ for all $i=1,\dots, n-2$. Additionally, we have
$$
\wt\alpha_x(\{P+c_1x_1^d+\dots+c_{n-1}x_{n-1}^d=0\})<\frac{n-1}{2}
$$
by Lemma \ref{lem:P3=x3^3}. Hence, it suffices to assume that $x_{n-1}$ is a factor of $F_i$ for all $i=1,\dots, n-2$. In particular, the weight $\mathrm{wt}_w(F_i)\ge\frac{1}{2}$.

If $d\ge 5$, then \eqref{eqn:x_1x_n^d-1} has isolated singularities and Lemma \ref{lem:P3=x3^3} concludes the proof. Indeed, the term $F_1$ is changed to $F_1+c_nx_n^{d-1}$, and thus, the term $G'$ of $X'$ in Lemma \ref{lem:P3=x3^3} satisfies $\mathrm{wt}_w(G')=1$ and $G_{\mathrm{wt}_w=1}=G_{\mathrm{wt}_w=1}'$.

If $d=4$, then $n\ge 5$ from the assumption. For general $c_1,\dots,c_{n-1}\in \C$, consider a hypersurface
$$
X'':=\{P+c_1x_1^4+\dots+c_{n-1}x_{n-1}^4+\alpha_1x_1x_n^3+\dots+\alpha_{n-2}x_{n-2}x_n^3=0\}\subset\P^n
$$
for $\alpha_1,\dots,\alpha_{n-2}\in \C$. When $(\alpha_1,\dots,\alpha_{n-2})\neq(0,\dots,0)$, this hypersurface $X''$ has isolated singularities by the argument above. Note that the term $G''$ of $X''$ in Lemma \ref{lem:P3=x3^3} still satisfies $\mathrm{wt}_w(G'')=1$, because $\mathrm{wt}_w(F_i+\alpha_i x_n^{3})\ge\frac{1}{2}$. The condition
$$
G_{\mathrm{wt}_w=1}''(x_{n-1},x_n)=(x_{n-1}+\alpha x_n^2)^3
$$
for some $\alpha\in \C$ is given by two equations in terms of $\alpha_1,\dots,\alpha_{n-2}$. Indeed, the condition that a monic cubic polynomial have a triple root imposes two equations on its coefficients. On the other hand, the space of $\alpha_1,\dots,\alpha_{n-2}$ is $(n-2)$-dimensional, which implies that there exists a $1$-dimensional family of $(\alpha_1,\dots,\alpha_{n-2})$ through $(0,\dots,0)$ satisfying the right hand side condition in Lemma \ref{lem:P3=x3^3}. Therefore, $X$ is a limit of $X''$ with isolated singularities satisfying $\wt\alpha_x(X'')<\frac{n-1}{2}$.

The only remaining case is $d=3$, and we have $n\ge 7$ from the assumption. Then we write $P$ as in Lemma \ref{lem:cubic P3=x^3} below. A priori, for $1\le i\le n-2$, the quadratic form $Q_i$ may have a nonzero $x_{n-1}^2$ term:
$$
Q_i=a_ix_n^2+b_ix_{n-1}x_n+c_ix_{n-1}^2.
$$
However, after the linear change of coordinates
$$
x_{n-1}\mapsto x_{n-1}+\frac{1}{3}(c_1x_1+\dots+c_{n-2}x_{n-2}),
$$
we may assume that $c_i=0$ for all $1\le i \le n-2$.

Now, we apply Lemma \ref{lem:cubic P3=x^3} below and follow its notation. Firstly, when $a_i\neq 0$ for some $1\le i\le n-2$, we prove that $X_\gamma$ has isolated singularities. It suffices to prove that
$$
\Sing(X_\gamma)\cap\{x_1=\dots=x_{n-2}=0\}=\{[1:0:\dots:0]\}.
$$
When $x_1=\dots=x_{n-2}=0$, the Jacobian criterion $\partial_{x_{n-1}}P_\gamma=0$
implies $x_{n-1}=0$. Furthermore, $\partial_{x_i}P_\gamma=0$ implies $a_ix_n^2=0$. This completes the proof when $a_i\neq 0$ for some $i$.

Consider the space of all polynomials $P$ of the form in Lemma \ref{lem:cubic P3=x^3}, that is, the affine space $A$ of coefficients $a_i, b_i, a_{ij}, b_{ij}, c_{ijk}$. Among these, the space $B$ of hypersurfaces satisfying
$$
\wt\alpha_x(\{P=0\})<\frac{n-1}2{}
$$
is the vanishing locus of $5$ equations in \eqref{eqn:cubic condition for < n-1/2}. In particular, every irreducible component of $B$ has codimension $\le 5$ in the affine space $A$. On the other hand, the complement $B\setminus C$ of the locus $C\subset B$ defined by $a_1=\dots=a_{n-2}=0$ satisfies our conclusion by the above argument. Note that the locus $C$ satisfies the additional vanishing
$$
\sum_{1\le i\le n-2}b_i^2=0
$$
by \eqref{eqn:cubic condition for < n-1/2}. Hence, $C$ has codimension $\ge n-1$ in the affine space $A$. Since $n-1>5$, the closure of $B\setminus C$ is all of $B$. The set of hypersurfaces satisfying the desired conclusion is closed (see Remark \ref{rmk:reduction to isolated singularities}), and thus, every hypersurface in $B$ satisfies the desired conclusion. This completes the proof for Case 2 and the proposition.
\end{proof}

\begin{lem}
\label{lem:cubic P3=x^3}
For $n\ge 7$, let $X\subset \P^n$ be a cubic hypersurface defined by an equation $\{P=0\}$. Suppose
$$
P=x_0(x_1^2+\dots+x_{n-2}^2)+x_{n-1}^3+\sum_{1\le i\le n-2}Q_ix_i+\sum_{1\le i,j\le n-2}L_{ij}x_ix_j+\sum_{1\le i,j,k\le n-2}c_{ijk}x_ix_jx_k.
$$
Here, $Q_i$ is a quadratic form in $x_{n-1}$ and $x_n$, $L_{ij}$ is a linear form in $x_{n-1}$ and $x_n$, and $c_{ijk}\in \C$, such that $L_{ij}=L_{ji}$ and $c_{ijk}$ is invariant under all permutations of $i,j,k$. Suppose
$$
Q_i=a_ix_n^2+b_ix_{n-1}x_n\quad\text{and}\quad L_{ij}=a_{ij}x_n+b_{ij}x_{n-1}.
$$
Then $\wt\alpha_x(X)<\frac{n-1}{2}$ if and only if
\begin{equation}
\label{eqn:cubic condition for < n-1/2}
\begin{gathered}
\sum_{1\le i\le n-2}a_i^2=0,\quad\sum_{1\le i\le n-2}a_ib_i=0,\quad\sum_{1\le i,j\le n-2}a_{ij}a_ia_j=0,\\
\sum_{1\le i,j\le n-2}(a_{ij}b_ia_j+a_{ij}a_ib_j+b_{ij}a_ia_j)=\frac{1}{12}\left(\sum_i b_i^2\right)^2,\\
2\sum_{1\le i,j,k\le n-2}a_{ij}a_{ik}a_ja_k+\sum_{1\le i,j,k\le n-2}c_{ijk}a_ia_ja_k=\frac{1}{216}\left(\sum_ib_i^2\right)^3.
\end{gathered}
\end{equation}
Furthermore, suppose that $\wt\alpha_x(X)<\frac{n-1}{2}$. For general $\gamma_{ijk}\in \C$ invariant under permutations of the indices and satisfying
\begin{equation}
\label{eqn:general gamma condition}
\sum_{1\le i,j,k\le n-2}\gamma_{ijk}a_ia_ja_k=0,
\end{equation}
the hypersurface
$$
X_\gamma:=\{P+\sum_{1\le i,j,k\le n-2}\gamma_{ijk}x_ix_jx_k=0\}\subset\P^n
$$
satisfies $\wt\alpha_x(X_\gamma)<\frac{n-1}{2}$ and
$$
\Sing(X_\gamma)\subset\{x_1=\dots=x_{n-2}=0\}\cup\{\text{finite points}\}.
$$
\end{lem}

\begin{proof}
We apply Lemma \ref{lem:P3=x3^3} and follow its notation. Note that the only possible nonzero monomials of $G$ with weight $<1$ with respect to the weight system $w$ are
$$
x_n^4,\quad x_{n-1}x_n^3,\quad x_n^5
$$
and their coefficients of $G(x_{n-1},x_n)$ are
$$
-\frac{1}{4}\sum_{i}a_i^2,\quad-\frac{1}{2}\sum_{i}a_ib_i,\quad\frac{1}{4}\sum_{i,j}a_{ij}a_ia_j,
$$
respectively (we omit the index ranges $1\le i,j,k\le n-2$ for brevity). Hence, the condition $\mathrm{wt}_w(G)=1$ is equivalent to the vanishing of the above coefficients, as expressed by the three equalities in the first line of \eqref{eqn:cubic condition for < n-1/2}. Additionally, the coefficients of the monomials
$$
x_{n-1}^2x_n^2,\quad x_{n-1}x_n^4,\quad x_n^6
$$
of $G(x_{n-1},x_n)$ are
$$
-\frac{1}{4}\sum_{i}b_i^2,\quad \frac{1}{4}\sum_{i,j}(a_{ij}b_ia_j+a_{ij}a_ib_j+b_{ij}a_ia_j),\quad -\frac{1}{4}\sum_{ i,j,k}a_{ij}a_{ik}a_ja_k-\frac{1}{8}\sum_{i,j,k}c_{ijk}a_ia_ja_k,
$$
respectively. In particular,
$$
G_{\mathrm{wt}_w=1}(x_{n-1},x_n)=(x_{n-1}+\alpha x_n^2)^3
$$
is satisfied for some $\alpha\in \C$ if and only if
$$
\alpha=-\frac{1}{12}\sum_{i}b_i^2
$$
and the equalities in the second and third lines of \eqref{eqn:cubic condition for < n-1/2} are satisfied. Therefore, Lemma \ref{lem:P3=x3^3} implies the equivalence between $\wt\alpha_x(X)<\frac{n-1}{2}$ and \eqref{eqn:cubic condition for < n-1/2}.

We now prove the last statement. Note that if $X$ satisfies \eqref{eqn:cubic condition for < n-1/2}, then so does $X_\gamma$. When $a_1=\dots=a_{n-2}=0$, the Bertini theorem implies
$$
\Sing(X_\gamma)\subset\{x_1=\dots=x_{n-2}=0\}\subset\P^n.
$$
It suffices to consider when $(a_1,\dots,a_{n-2})\neq 0$. From the condition that
$$
\sum_{1\le i\le n-2}a_i^2=0
$$
in \eqref{eqn:cubic condition for < n-1/2}, we may assume $a_1$ and $a_2$ are nonzero, without loss of generality. The linear system of cubics
$$
\sum_{1\le i,j,k\le n-2}\gamma_{ijk}x_ix_jx_k
$$
satisfying \eqref{eqn:general gamma condition} is the linear system of cubics in $\P^{n-3}$ with coordinates $x_1,\dots,x_{n-2}$ vanishing at the point 
$$
[a_1:\dots:a_{n-2}]\in \P^{n-3}.
$$
By the Bertini theorem, we have
$$
\Sing(X_\gamma)\subset\{[x_1:\dots:x_{n-2}]=[a_1:\dots:a_{n-2}]\}\cup\{x_1=\dots=x_{n-2}=0\}\subset\P^n
$$
for general such $\gamma_{ijk}\in \C$.

Suppose $[x_0:a_1:\dots:a_{n-2}:x_{n-1}:x_n]\in \Sing(X_\gamma)$ and denote by 
$$
P_\gamma:=P+\sum_{1\le i,j,k\le n-2}\gamma_{ijk}x_ix_jx_k.
$$
Then by \eqref{eqn:cubic condition for < n-1/2}, we have
$$
P_\gamma(x_0,a_1,\dots,a_{n-2},x_{n-1},x_n)=x_{n-1}^3+\sum_{1\le i,j\le n-2}a_ia_jb_{ij}x_{n-1}+\sum_{1\le i,j,k\le n-2}c_{ijk}a_ia_ja_k=0
$$
and
$$
\partial_{x_{n-1}}P_\gamma(x_0,a_1,\dots,a_{n-2},x_{n-1},x_n)=3x_{n-1}^2+\sum_{1\le i,j\le n-2}a_ia_jb_{ij}=0.
$$
Therefore, $x_{n-1}$ is uniquely determined as a double root of the cubic polynomial
$$
x^3+\left(\sum_{1\le i,j\le n-2}a_ia_jb_{ij}\right)x+\sum_{1\le i,j,k\le n-2}c_{ijk}a_ia_ja_k=0.
$$
Hence, we denote the double root by $a_{n-1}$, and set $x_{n-1}=a_{n-1}$. The Jacobian criterion implies
$$
\partial_{x_i}P_\gamma(x_0,a_1,\dots,a_{n-2},a_{n-1},x_n)=0
$$
for $1\le i\le n-2$. By an elementary calculation, we have
$$
\partial_{x_i}P_\gamma=2a_ix_0+a_ix_n^2+\left(b_ia_{n-1}+2\sum_{j}a_{ij}a_j\right)x_n+2\sum_{j}b_{ij}a_ja_{n-1}+3\sum_{j,k}(c_{ijk}+\gamma_{ijk})a_ja_k
$$
(we omit the index ranges $1\le i,j,k\le n-2$ for brevity) at the point $(x_0,a_1,\dots,a_{n-2},a_{n-1},x_n)$.

Recall that $a_1$ and $a_2$ are nonzero. Consider
$$
a_2\partial_{x_1}P_\gamma-a_1\partial_{x_2}P_\gamma=0
$$
at the point $(x_0,a_1,\dots,a_{n-2},a_{n-1},x_n)$. This is a linear equation for $x_n$. The coefficients of $\gamma_{111}$ and $\gamma_{222}$ in this equation are $3a_2a_1^2$ and $-3a_1a_2^2$. However, the coefficients for \eqref{eqn:general gamma condition} are $a_1^3$ and $a_2^3$. In particular, by the generality of $\gamma_{ijk}$ satisfying \eqref{eqn:general gamma condition}, the constant term of the linear equation for $x_n$ is nonzero unless
$$
\frac{a_1^3}{a_2^3}=-\frac{3a_2a_1^2}{3a_1a_2^2},
$$
or equivalently $a_1^2=-a_2^2$. Thus, unless $a_1^2=-a_2^2$, the value of $x_n$ is uniquely determined. The equation $\partial_{x_1}P_\gamma=0$ then uniquely determines $x_0$. Therefore, it suffices to consider $a_i^2=-a_j^2$ whenever $i\neq j$ and $a_i, a_j$ are nonzero.

If $a_i=0$ for some $1\le i \le n-2$, then the above equation $\partial_{x_i}P_\gamma=0$ is a linear equation for $x_n$ and the coefficient of $\gamma_{111}$ is zero. Hence, we likewise conclude that $x_n$ and $x_0$ are uniquely determined. Consequently, we may assume $a_i\neq 0$ for all $1\le i \le n-2$, and thus,
$$
a_i^2=-a_j^2
$$
for all $1\le i,j\le n-2$ with $i\neq j$. This is not possible under the assumption that $(a_1,\dots,a_{n-2})\neq 0$, so there is no singular point of the form $[x_0:a_1:\dots:a_{n-2}:x_{n-1}:x_n]$ in this case.
\end{proof}

\section{Proof of the main theorem}
\label{sec:Proof of main theorem}

In this section, we prove our main result, Theorem \ref{thm:indeterminacy}, on the indeterminacy locus of the compactified period map $\overline {\mathcal P}$. 

\begin{proof}[Proof of Theorem \ref{thm:indeterminacy}]
For cubic fourfolds $(n,d)=(5,3)$, Laza established the indeterminacy locus in \cite{Laza10}. For quartic K3s $(n,d)=(3,4)$, Laza and O'Grady \cite{LOG18, LOG19} predicted the indeterminacy locus, and Ascher, DeVleming, and Liu \cite{ADL23} resolved this indeterminacy locus via the K-moduli wall-crossing of pairs. See also the proof of Theorem \ref{thm:BBH indeterminacy} in Section \ref{sec:BB versus BBH for Calabi--Yau}.

Now, we prove the theorem for $(n,d)\neq (3,3),(3,4),(5,3)$, where $n\ge 2$ and $d\ge 3$.
We consider the period map
$$
\overline \Phi:\P |\O_{\P^n}(d)|\dashrightarrow\overline \M^{\rm BB}_{(n,d)}.
$$
Denote by $U$ the domain of definition of $\overline \Phi$, and by $V$ the locus of hypersurfaces with $(\lfloor\frac{n}{2}\rfloor-1)$-Du Bois singularities. By Theorem \ref{thm:extendability}, $U$ contains $V$. By Corollary \ref{cor:numerical criterion for pencil}, it suffices to prove that $U=V$.

Denote by $W\subset \P |\O_{\P^n}(d)|$ the locus of hypersurfaces with isolated singularities. By Corollary~\ref{cor:indeterminacy for isolated singularities}, we have $U\cap W=V\cap W$. Additionally, by Proposition \ref{prop:reduction to isolated sing}, the complement $V^c$ is the closure of $W\setminus V$. Consequently, $U\setminus V$ is the closure of $(U\cap W)\setminus V$, which is an empty set. Therefore, $U=V$ as desired.
\end{proof}

\begin{proof}[Proof of Corollary \ref{cor:indeterminacy}]
In \cite[Corollary 8.1]{Park25}, part (1) is established when $n$ is even. When $n$ is odd, the domain of definition of $\mathcal P$ contains the locus of hypersurfaces with simple ADE singularities. We prove that they are equal.

Note that the domain of definition of $\mathcal P$ is contained in the domain of definition of $\overline {\mathcal P}$. By Theorem \ref{thm:indeterminacy}, it suffices to prove that $\mathcal P$ is not defined for a hypersurface $X\subset \P^n$ with $(\frac{n-3}{2})$-Du Bois singularities but not simple ADE singularities. In terms of \cite{Park25}, this hypersurface has $(\frac{n-3}{2})$-liminal singularities (see \cite[Lemma 8.2]{Park25} for the fact that simple ADE singularities are equivalent to $(\frac{n-3}{2})$-rational singularities).

By Griffiths' Removable Singularity Theorem (see \cite[Theorem 9.5]{Griffiths70} or \cite[Application 16]{GT84}), it suffices to prove the following:

\smallskip
\noindent
\textit{Claim.} Let $n$ be an odd integer, and $X\subset\P^n$ be a hypersurface of degree $d$ with $(\frac{n-3}{2})$-liminal singularities. For any one-parameter smoothing $f:\X\to \Delta$ of $X$, the local monodromy of the middle cohomology $\mathbb H^{n-1}(\X_\eta/\Delta_\eta)$ at the origin is not finite.
\smallskip

Recall that the local monodromy is finite if and only if the limiting mixed Hodge structure $H^{n-1}(\X_\infty,\Q)$ is pure of weight $n-1$. By \cite[Theorem 11.1]{Park25}, the specialization map
$$
\mathrm{sp}^{n-1}:H^{n-1}(X,\Q)\to H^{n-1}(\X_\infty,\Q)
$$
induces isomorphisms of graded Hodge pieces
\begin{equation}
\label{eqn:cohomological insignificance}
\gr_F^p(\sp^{n-1}): \gr_F^pH^{n-1}(X,\C)\xrightarrow{\sim} \gr_F^pH^{n-1}(\X_\infty,\C) 
\end{equation}
for all $p\le \frac{n-3}{2}$. Suppose $H^{n-1}(\X_\infty,\Q)$ is pure of weight $n-1$. Then the weight $\le n-2$ subspace
$$
W_{n-2}H^{n-1}(X,\Q)
$$
is contained in the kernel of $\mathrm{sp}^{n-1}$. Hence, if $W_{n-2}H^{n-1}(X,\Q)$ is nonzero, then
$$
\gr_F^p\left(W_{n-2}H^{n-1}(X,\C)\right)\neq0
$$
for some $p\le \frac{n-3}{2}$, which is not possible by \eqref{eqn:cohomological insignificance}. This implies that $H^{n-1}(X,\Q)$ is also pure of weight $n-1$.

Hence, it suffices to prove
$$
\dim_\C\gr_F^{\frac{n-3}{2}}H^{n-1}(X,\Q)\neq\dim_\C\gr_F^{\frac{n+1}{2}}H^{n-1}(X,\Q),
$$
for a hypersurface $X\subset \P^n$ with $(\frac{n-3}{2})$-liminal singularities. This implies that $H^{n-1}(X,\Q)$ is not pure of weight $n-1$. Using the terminology of \cite{PP25a}, this statement is equivalent to
$$
\h^{\frac{n-3}{2},\frac{n+1}{2}}(X)\neq\h^{\frac{n+1}{2},\frac{n-3}{2}}(X).
$$
Note that $X$ has $(\frac{n-5}{2})$-rational singularities (see \cite[Appendix]{FL24} or \cite[Theorem E]{MP25}), which implies
$$
\h^{\frac{n+1}{2},\frac{n-3}{2}}(X)=I\h^{\frac{n+1}{2},\frac{n-3}{2}}(X)=I\h^{\frac{n-3}{2},\frac{n+1}{2}}(X)
$$
by \cite[Theorem 7.1 and Remark 7.2]{PP25a}. We instead prove
\begin{equation}
\label{eqn:h neq Ih}
\h^{\frac{n-3}{2},\frac{n+1}{2}}(X)\neq I\h^{\frac{n-3}{2},\frac{n+1}{2}}(X).
\end{equation}

Consider the RHM defect object $\K^\bullet_X\in D^b{\rm MHM}(X)$ sitting in the distinguished triangle
$$
\K^\bullet_X\to \Q_X^H[n-1]\to \IC_X^H\xrightarrow{+1}
$$
(see \cite[Section 6]{PP25a} for details). Since $X$ is a hypersurface, $\K^\bullet_X$ is a single mixed Hodge module. Since $X$ is $(\frac{n-3}{2})$-liminal, the index of the first nonzero (increasing) Hodge filtration of its dual $\dual\K^\bullet_X$ is $\frac{n-3}{2}$ (see discussion following \cite[Definition 3.1]{Park25}).

We apply \cite[Proposition 10.6]{Park25}, ignoring the assumption $\frac{n+1}{d}=m+1$ and setting $m=\frac{n-3}{2}$. The same proof applies in our case, and we obtain a surjective composition map
\begin{equation}
\label{eqn:composition map}
\O_{X}\left(-n-1+\frac{n-1}{2}d\right)\to R\Hom_{\O_X}\left(\Omega_X^\frac{n-3}{2},\w_X\right)\left[\frac{n-3}{2}\right]\to F_\frac{n-3}{2}\dual\K_X^\bullet.
\end{equation}
In particular, $F_{\frac{n-3}{2}}\dual\K_X^\bullet\isom \O_S(-n-1+\frac{n-1}{2}d)$ where $S$ is the $\frac{n-3}{2}$-liminal locus of $X$. Since we can always find a nonzero hypersurface section on $S$, the map of global sections
$$
H^0\left(\O_{X}\left(-n-1+\frac{n-1}{2}d\right)\right)\to H^0\left(\O_S\left(-n-1+\frac{n-1}{2}d\right)\right)
$$
is nonzero. This is the composition map obtained by applying the hypercohomology $\HH^0(\,\cdot\,)$ to \eqref{eqn:composition map}. Note that, after applying $\HH^0(\,\cdot\,)$, the dual of the second map in \eqref{eqn:composition map} is the natural map
$$
\gr_F^\frac{n-3}{2}\HH^0(\K_X^\bullet)\to \gr_F^\frac{n-3}{2}H^{n-1}(X,\C)
$$
(see, for example, the proof of \cite[Corollary 10.7]{Park25}). Hence, we conclude that this map is nonzero.

Recall that $F_{\frac{n-3}{2}}\dual\K_X^\bullet$ is a sheaf, and thus, we have the vanishing
$$
H^{<0}\left(F_{\frac{n-3}{2}}\dual\K_X^\bullet\right)=0.
$$
This implies the vanishing of its dual
$$
\gr_F^\frac{n-3}{2}\HH^{>0}(\K_X^\bullet)=0.
$$
Therefore, we have an exact sequence
$$
\gr_F^\frac{n-3}{2}\HH^0(\K_X^\bullet)\to \gr_F^\frac{n-3}{2}H^{n-1}(X,\C)\to\gr_F^\frac{n-3}{2}\IH^{n-1}(X,\C)\to 0
$$
and the first map being nonzero implies \eqref{eqn:h neq Ih}.

Finally, we prove the last statement of Corollary \ref{cor:indeterminacy}.
For Calabi--Yau hypersurfaces, the birational map
$$
\overline {\mathcal P}^{\rm BBH}:\overline\M^{\rm GIT}_{(n,n+1)} \dashrightarrow \overline \M^{\rm BBH}_{(n,n+1)}
$$
factors through the Stein factorization
$$
\overline\M^{\rm GIT}_{(n,n+1)}\dashrightarrow \left(\overline \M^{\rm BBH}_{(n,n+1)}\right)^\nu
$$
of the period map $\overline {\mathcal P}$. Recall that $\overline {\mathcal P}^{\rm BBH}$ is an isomorphism over the locus $\M^{\rm can}_{(n,n+1)}$ of Calabi--Yau hypersurfaces with canonical singularities. We have proven that the domain of definition of $\mathcal P$ is contained in $\M^{\rm can}_{(n,n+1)}$. Therefore, the Stein factorization of the period map $\mathcal P$ to its period image is an open embedding on the domain of definition of $\mathcal P$.
\end{proof}

\section{Baily--Borel versus BBH for Calabi--Yau hypersurfaces}
\label{sec:BB versus BBH for Calabi--Yau}

In this section, we prove Theorem \ref{thm:BBH indeterminacy} and compare the Baily--Borel and BBH compactifications for Calabi--Yau hypersurfaces. The strategy of the proof follows that of Theorem \ref{thm:indeterminacy}, with the Griffiths height replaced by the \emph{CY height}.

\begin{defn}
Let $C$ be a smooth projective curve, and let $\mathbb V_\eta$ be a polarized Calabi--Yau variation of $\Z$-Hodge structures defined on a nonempty Zariski open subset of $C$. We define its \emph{CY height}, \emph{lower CY height}, and \emph{upper CY height}, respectively, by
$$
\mathrm{ht}_{CY}(\mathbb V_\eta):=\deg_{C}\Lambda_C,\quad \mathrm{ht}_{CY,-}(\mathbb V_\eta):=\deg_{C}\Lambda_{C}^{[0,1)},\quad \mathrm{ht}_{CY,+}(\mathbb V_\eta):=\deg_{C}\Lambda_{C}^{(-1,0]}.
$$
\end{defn}

Here, a polarized variation of $\Z$-Hodge structures $\mathbb V=(\mathbb V_\Z,F^\bullet, Q)$ of weight $w$ is \emph{Calabi--Yau} if $F^0\mathcal V=\mathcal V$ and $\gr_F^{0}\mathcal V$ is a line bundle. The Hodge line bundle of $\mathbb V$ is $F^w\mathcal V$, the last nonzero graded piece of the Hodge filtration. As explained in Section \ref{sec:Deligne's canonical extension and Griffiths}, $\mathbb V_\eta$ on a nonempty Zariski open subset of $C$ admits Deligne's lower and upper canonical extensions
$$
\overline{\mathcal V}^{[0,1)} \quad{\rm and}\quad \overline{\mathcal V}^{(-1,0]}
$$
on $C$, and we have the induced lower and upper Hodge bundles
$$
\Lambda_{C}^{[0,1)}:=F^w\overline{\mathcal V}^{[0,1)} \quad{\rm and}\quad \Lambda_{C}^{(-1,0]}:=F^w\overline{\mathcal V}^{(-1,0]}.
$$
As explained after Definition \ref{defn:lower Griffiths bundle}, we may consider a finite flat morphism $\pi:T\to C$ such that the local monodromy of $\pi^*\mathbb V_\eta$ is unipotent. Then the Hodge bundle $\Lambda_{T}^{[0,1)}=\Lambda_{T}^{(-1,0]}$ descends to a class
$$
\Lambda_C\in \mathrm{Pic}(C)\otimes_\Z \Q
$$
which we call the Hodge $\Q$-bundle.

To determine the indeterminacy locus of the birational map
$$
\overline {\mathcal P}^{\rm BBH}:\overline\M^{\rm GIT}_{(n,n+1)} \dashrightarrow \overline \M^{\rm BBH}_{(n,n+1)},
$$
we establish an analogue of Theorem \ref{thm:Mordant question and indeterminacy}. The proof is essentially the same.

\begin{thm}\label{thm:Mordant question for Calabi Yau}
Let $n\geq 2$ and let $H_\eta\subset \P^n_{C_\eta}$ be a smooth hypersurface of degree $n+1$ over the generic point $C_\eta$ of a smooth projective curve $C$. Then we have
$$
\mathrm{ht}_{CY}(\mathbb H^{n-1}(H_\eta/C_\eta))\le \mathrm{ht}_{GIT}(H_\eta),
$$
and the equality holds if and only if the image of the moduli morphism $C\to\overline\M^{\rm GIT}_{(n,n+1)}$ is contained in the domain of definition of $\overline{\mathcal P}^{\rm BBH}$.
\end{thm}

\begin{proof}
By Proposition \ref{prop:GIT is Q factorial picard rank 1}, the reflexive hull of the pullback of the Hodge bundle $\O_{\overline \M^{\rm BBH}_{(n,n+1)}}(1)$ is a rational multiple of the GIT polarization $\O_{\overline\M^{\rm GIT}_{(n,d)}}(1)$:
$$
\O_{\overline\M^{\rm GIT}_{(n,d)}}(\alpha)\isom_\Q \left[\overline {\mathcal P}^*\O_{\overline \M^{\rm BBH}_{(n,n+1)}}(1)\right]^{**} \quad \text{for}\quad \alpha\in\Q.
$$
Then Lemma \ref{lem:Mordant Q1 numerical criterion} completes the proof if we additionally have $\alpha=1$.

Note that $\overline{\mathcal P}^{\rm BBH}$ is defined over a codimension $1$ point. Consider a general pencil
$$
\X:=\{P+tQ=0\}\subset \P^n\times\P_t^1
$$
over a rational curve $\P^1_t$. Then the GIT moduli image $\overline\Phi_{\rm GIT}(\P^1_t)$ is contained in the domain of definition of $\overline{\mathcal P}^{\rm BBH}$. Therefore, it suffices to prove
$$
\mathrm{ht}_{CY}(\mathbb H^{n-1}(\X_\eta/\P^1_\eta))=1.
$$
Note that $\X$ is smooth. Hence, by \cite[Theorem 2.6]{Kollar86}, we have
$$
\Lambda_{\P^1}^{(-1,0]}\isom f_*\w_{\X/\P^1}
$$
where $f:\X\to \P^1$. We have $\w_\X\isom \w_{\P^n\times\P^1}(n+1,1)|_\X$ by the adjunction formula, which implies the isomorphism
$$
f_*\w_{\X/\P^1}\isom\O_{\P^1}(1).
$$
Therefore, we have
$$
\mathrm{ht}_{CY,+}(\mathbb H^{n-1}(\X_\eta/\P^1_\eta))=\deg_{\P^1}f_*\w_{\X/\P^1}=1.
$$

Now, it suffices to prove
$$
\mathrm{ht}_{CY}(\mathbb H^{n-1}(\X_\eta/\P^1_\eta))=\mathrm{ht}_{CY,+}(\mathbb H^{n-1}(\X_\eta/\P^1_\eta)).
$$
For any disk $\Delta\subset \P^1$ centered at $c\in \P^1$, we have the isomorphism of the graded specialization map
$$
\gr_F^0(\sp^{n-1}): \gr_F^0H^{n-1}(\X_c,\C)\xrightarrow{\sim} \gr_F^0H^{n-1}(\X_\infty,\C)_1
$$
where $H^{n-1}(\X_\infty,\Q)_1$ is the unipotent monodromy eigenspace of the limiting mixed Hodge structure over $\Delta^*$ at $c$ (see \cite[Remark 11.6]{Park25}). Since the fiber $\X_c$ has Du Bois singularities, we have
$$
\dim_\C\gr_F^0H^{n-1}(\X_c,\C)=1
$$
(see \cite{Dolgachev81,Steenbrink81}), which implies
$$
\dim_\C\gr_F^0H^{n-1}(\X_\infty,\C)_1=1.
$$
By the invariance of Hodge numbers for limiting mixed Hodge structures, we have
$$
\dim_\C\gr_F^0H^{n-1}(\X_\infty,\C)=1,
$$
and thus, the monodromy $\lambda$-eigenspace $H^{n-1}(\X_\infty,\C)_\lambda$ satisfies
$$
\dim_\C\gr_F^0H^{n-1}(\X_\infty,\C)_\lambda=0
$$
for all $\lambda\neq 1$. By the symmetry of Hodge numbers (see \cite[Theorem 6.16]{Schmid73} or \cite[Section 5.1.6]{Saito88}), we obtain
$$
\dim_\C F^{n-1}H^{n-1}(\X_\infty,\C)_\lambda=0.
$$
Therefore, Proposition \ref{prop:CY height via LMHS} applied to $\mathbb V_\eta:=\mathbb H^{n-1}(\X_\eta/\P^1_\eta)$ implies the desired conclusion.
\end{proof}

As in Section \ref{sec:LMHS and Griffiths}, we describe the difference between the CY height and the upper CY height in terms of limiting mixed Hodge structures.

\begin{prop}
\label{prop:CY height via LMHS}
Let $C$ be a smooth projective curve, and let $\mathbb V_\eta$ be a Calabi--Yau variation of $\Z$-Hodge structures of weight $w$, defined on an open subset $C\sm\{c_1,\dots,c_s\}$ of $C$. Then
$$
\mathrm{ht}_{CY,+}(\mathbb V_\eta)-\mathrm{ht}_{CY}(\mathbb V_\eta)=\sum_{\gamma=1}^s\sum_{\alpha\in[0,1)}\alpha\dim_\C F^w\left[\psi_{t_\gamma,e^{2\pi i\alpha}}\left(j_*\mathbb V^H[1]\right)\right],
$$
where $t_\gamma$ is a local coordinate function near the point $c_\gamma$ and $j:C\sm\{c_1,\dots,c_s\}\to C$ is an open embedding.
\end{prop}

\begin{proof}
The proof of Proposition \ref{prop:Griffiths height via LMHS} applies with the Griffiths height replaced by the CY height, using Lemma \ref{lem:Deligne's upper canonical extension and finite cover} below.
\end{proof}

\begin{lem}
\label{lem:Deligne's upper canonical extension and finite cover}
In the setting of Lemma \ref{lem:Deligne's canonical extension and finite cover}, we have a natural inclusion
$$
F^p\overline {\mathcal V}_s\subset\rho^*F^p\overline {\mathcal V}_t^{(-1,0]} 
$$
and the dimension of the cokernel is
$$
\dim_\C \left[\rho^*F^p\overline {\mathcal V}_t^{(-1,0]}/F^p\overline {\mathcal V}_s\right]
=\sum_{k=0}^{m-1}k\dim_\C F^p\left[\psi_{t,e^{2\pi i\frac{k}{m}}}\left(j_*\mathbb V^H[1]\right)\right].$$
\end{lem}

\begin{proof}
We follow the proof of Lemma \ref{lem:Deligne's canonical extension and finite cover} for Deligne's upper canonical extension. From the local analytic description, it is easy to deduce the following isomorphism of vector bundles
$$
(\rho_*(s^{-m+1}\overline{\mathcal V_s}))^G\isom \overline{\mathcal V_t}^{(-1,0]}.
$$
Consequently, we obtain an isomorphism
$$
(\rho_*(s^{-m+1}F^p\overline{\mathcal V_s}))^G\isom F^p\overline{\mathcal V_t}^{(-1,0]}.
$$

As in the proof of Lemma \ref{lem:Deligne's canonical extension and finite cover}, we consider the sections
$$
V_{k,1},\dots, V_{k,a_k}
$$
of $F^p\overline{\mathcal V}_s$ for each $0\le k\le m-1$, which together form a basis of $F^p\overline{\mathcal V}_s$ as an $\O_{\Delta_s}$-module. Recall that the images of the above sections in $F^p\overline{\mathcal V}_s|_0$ form a basis of its $\zeta^{-k}$-eigenspace. Since
$$g(s^jV_{k,i})=\zeta^{j-k}s^jV_{k,i},$$
we have
$$
(\rho_*(s^{-m+1}F^p\overline{\mathcal V_s}))^G=\left(\bigoplus_{k=0}^{m-1}\bigoplus_{i=1}^{a_k}s^{-m+1}V_{k,i}\O_{\Delta_s}\right)^G=\bigoplus_{i=1}^{a_0}V_{0,i}\O_{\Delta_t}\oplus\bigoplus_{k=1}^{m-1}\bigoplus_{i=1}^{a_k}s^{-m+k}V_{k,i}\O_{\Delta_t}.
$$
This implies
$$
\dim_\C \left[\rho^*F^p\overline {\mathcal V}_t^{(-1,0]}/F^p\overline {\mathcal V}_s\right]=\sum_{k=1}^{m-1}(m-k)a_k=\sum_{k=1}^{m-1}ka_{m-k},
$$
which is the desired equality.
\end{proof}

In the remainder of this section, we finish the proof of Theorem \ref{thm:BBH indeterminacy}. For this purpose, we prove the following lemma on the singularities of a general pencil through a hypersurface.

\begin{lem}
\label{lem:general pencil is rational}
Let $X\subset \P^n$ be a degree $d$ hypersurface with a defining equation $\{P=0\}$. Then for a general degree $d$ hypersurface $\{Q=0\}\subset \P^n$, the pencil
$$
\X:=\{P+tQ=0\}\subset \P^n\times\P_t^1
$$
has rational singularities.
\end{lem}

\begin{proof}
We stratify the singular locus of $X$:
$$
\Sing(X)=\bigsqcup_{\alpha}X_\alpha
$$
so that $X_\alpha$ is a locally closed smooth subvariety of $X$. Since $Q$ is general, the intersection
$$
X_\alpha\cap\{Q=0\}
$$
is smooth for all $\alpha$ by the Bertini theorem. By the Jacobian criterion, $\X$ is smooth away from $\Sing(X)\cap\{Q=0\}$. Hence, it suffices to prove that $\X$ has rational singularities near every point in
$$
\Sing(X)\cap\{Q=0\}\subset X\subset\X.
$$

Consider a point $x\in X_\alpha\cap\{Q=0\}$. Choose a local coordinate system in a neighborhood of $x$ in $\P^n$ as follows:
$$
u_1,\dots,u_k,v,u_{k+1},\dots,u_{n-1}
$$
where $X_\alpha$ (resp. $\{Q=0\}$) is locally $\{u_1=\dots=u_k=0\}$ (resp. $\{v=0\}$). Then we have a local coordinate system in a neighborhood of $x$ in $\P^n\times\P_t^1$ as follows:
$$
t, u_1,\dots,u_k,v,u_{k+1},\dots,u_{n-1}.
$$
In terms of this local coordinate system, $\X$ is locally defined by
$$
p+tv=0
$$
where $p\in (u_1,\dots,u_{k})^2$, since $X$ is singular along $X_\alpha$. Write
$$
p=p_0+vh,
$$
so that $p_0\in \C\{u_1,\dots,u_{n-1}\}$, the ring of convergent power series in $u_1,\dots, u_{n-1}$. Since $P$ is not a multiple of $Q$, the term $p_0$ is nonzero. Consequently, $\X$ is locally defined by
$$
p_0+v\tilde t=0
$$
where $\tilde t=t+h$. Since $h\in (u_1,\dots,u_{k})^2$, the functions
$$
u_1,\dots,u_{n-1},v, \tilde t
$$
form a local coordinate system in a neighborhood of $x$ in $\P^n\times\P_t^1$. Therefore, the Thom--Sebastiani theorem (see Theorem \ref{thm:Thom-Sebastiani for minimal exponents}) implies that
$$
\wt\alpha_x(\X)=\wt\alpha_0(p_0)+\wt\alpha_0(v\tilde t)>\wt\alpha_0(v\tilde t)=1,
$$
that is, $\X$ has rational singularities at $x$ by \cite{Saito93}.
\end{proof}

\begin{proof}[Proof of Theorem \ref{thm:BBH indeterminacy}]
As in Lemma \ref{lem:general pencil is rational}, consider a general pencil
$$
\X:=\{P+tQ=0\}\subset \P^n\times\P_t^1
$$
with $X:=\{P=0\}\subset\P^n$. Note that every fiber away from $X=\X_0$ has at worst nodal singularities. The locus of nodal hypersurfaces is contained in the domain of definition of $\overline{\mathcal P}^{\rm BBH}$. Indeed, the proof of Theorem \ref{thm:extendability} can be modified to our situation to prove that the locus of hypersurfaces with Du Bois singularities is contained in the domain of definition, or Theorem~\ref{thm:indeterminacy} can be applied since the domain of definition of $\overline {\mathcal P}$ is contained in that of $\overline{\mathcal P}^{\rm BBH}$.

In particular, the image of the moduli morphism $\P^1_t\to \overline\M^{\rm GIT}_{(n,n+1)}$ is contained in the domain of definition of $\overline{\mathcal P}^{\rm BBH}$ if and only if the image of the fiber $[X]$ is contained in the domain of definition. Hence, by Theorem \ref{thm:Mordant question for Calabi Yau}, to prove that the domain of definition is precisely the locus of hypersurfaces with Du Bois singularities, it suffices to prove that
$$
\mathrm{ht}_{CY}(\mathbb H^{n-1}(\X_\eta/\P^1_\eta))=1
$$
if and only if $X$ has Du Bois singularities (or equivalently, semi-log canonical singularities).

As in the proof of Theorem \ref{thm:Mordant question for Calabi Yau}, consider a disk $\Delta\subset \P^1$ centered at a point $c\in \P^1$. Since $\X$ has rational singularities by Lemma \ref{lem:general pencil is rational}, we have the isomorphism of the graded specialization map
$$
\gr_F^0(\sp^{n-1}): \gr_F^0H^{n-1}(\X_c,\C)\xrightarrow{\sim} \gr_F^0H^{n-1}(\X_\infty,\C)_1
$$
(see \cite[Remark 11.6]{Park25}). By the claim below,
$$
\dim_\C\gr_F^0H^{n-1}(\X_\infty,\C)_1=0
$$
if $\X_c$ has singularities worse than Du Bois singularities. In such a case, the invariance of Hodge numbers for limiting mixed Hodge structures gives
$$
\dim_\C\gr_F^0H^{n-1}(\X_\infty,\C)=1.
$$
Hence, there exists $\lambda\neq1$ such that
$$
\dim_\C\gr_F^0H^{n-1}(\X_\infty,\C)_\lambda=1,
$$
and thus, by the symmetry of Hodge numbers, we have
$$
\dim_\C F^{n-1}H^{n-1}(\X_\infty,\C)_{\bar\lambda}=1.
$$

Therefore, when $X$ has singularities worse than Du Bois singularities, Proposition \ref{prop:CY height via LMHS} applied to $\mathbb V_\eta:=\mathbb H^{n-1}(\X_\eta/\P^1_\eta)$ implies that
$$
\mathrm{ht}_{CY,+}(\mathbb H^{n-1}(\X_\eta/\P^1_\eta))-\mathrm{ht}_{CY}(\mathbb H^{n-1}(\X_\eta/\P^1_\eta))>0.
$$
Again, by Lemma \ref{lem:general pencil is rational} and \cite[Theorem 2.6]{Kollar86}, we have
$$
\mathrm{ht}_{CY,+}(\mathbb H^{n-1}(\X_\eta/\P^1_\eta))=\deg_{\P^1}f_*\w_{\X/\P^1}=1
$$
(see the proof of Theorem \ref{thm:Mordant question for Calabi Yau}).

Conversely, when $X$ has Du Bois singularities, we obtain
$$
\mathrm{ht}_{CY}(\mathbb H^{n-1}(\X_\eta/\P^1_\eta))=\mathrm{ht}_{CY,+}(\mathbb H^{n-1}(\X_\eta/\P^1_\eta))=1,
$$
which completes the proof. It remains to prove the following claim for the statement about the domain of definition.

\smallskip
\noindent
\textit{Claim.} Let $X\subset \P^n$ be a degree $n+1$ hypersurface. Then
$$
\dim_\C\gr_F^0H^{n-1}(X,\C)=0
$$
if and only if $X$ has singularities worse than Du Bois singularities.
\smallskip

Recall that
$$
\gr_F^0H^{n-1}(X,\C)\isom\HH^{n-1}(X,\DB_X^0)
$$
(see, for example, \cite{PP25a}), and $\w_X\isom\O_X$. By Grothendieck duality, we have
$$
\HH^{n-1}(X,\DB_X^0)^*\isom H^0(X,R\Hom(\DB_X^0,O_X)).
$$
By the injectivity theorem \cite[Theorem 3.3]{KS16}, we have the vanishing
$$
R^{\neq0}\Hom(\DB_X^0,O_X)=0
$$
and the injective morphism
$$
\Hom(\DB_X^0,O_X)\hookrightarrow \O_X.
$$
Therefore,
$$
\gr_F^0H^{n-1}(X,\C)\neq 0
$$
if and only if $\Hom(\DB_X^0,O_X)\isom\O_X$. This condition is equivalent to the isomorphism
$$
\O_X\isom\DB_X^0,
$$
that is, $X$ has Du Bois singularities.

Finally, we prove the last statement of Theorem \ref{thm:BBH indeterminacy}. Suppose $n\ge 4$. Recall from Proposition~\ref{prop:GIT is Q factorial picard rank 1} that $\overline \M^{\rm GIT}_{(n,n+1)}$ is $\Q$-factorial. Hence, its open subset $\M^{\rm can}_{(n,n+1)}$ is $\Q$-factorial. This implies that the morphism
$$
\pi_{CY}:\left(\overline \M^{\rm BB}_{(n,d)}\right)^\nu\to \overline \M^{\rm BBH}_{(n,n+1)}
$$
is a divisorial contraction over $\M^{\rm can}_{(n,n+1)}$.

When $n$ is even, there exists a hypersurface $X:=\{P=0\}\subset \P^n$ with a unique singularity, analytically equivalent to
$$
x_1^2+\dots+x_{n-1}^2+x_n^3=0.
$$
By Theorem \ref{thm:indeterminacy}, $[X]$ is in the indeterminacy locus of $\overline {\mathcal P}$, which implies that the fiber $\pi_{CY}^{-1}([X])$ is positive dimensional. For a general pencil $f:\X\to \P^1$ as above, the local monodromy at $X$ is finite. Indeed, by the theorem of Milnor \cite[Lemma 9.4]{Milnor68}, the monodromy on the Milnor fiber has finite order when the singularity is weighted homogeneous. By Griffiths' Removable Singularity Theorem (see \cite[Theorem 9.5]{Griffiths70} or \cite[Application 16]{GT84}), the disk $\Delta\subset\P^1$ centered at $0\in \P^1$ admits a morphism to the period space:
$$
\Delta\to \Gamma\backslash \D_{(n,n+1)}.
$$
This implies that there exists an exceptional divisor of $\pi_{CY}$ that generically maps to the interior of the period space.

When $n$ is odd, there exists a hypersurface $X:=\{P=0\}\subset \P^n$ with a unique singularity, analytically equivalent to
$$
x_1^2+\dots+x_{n-2}^2+x_{n-1}^4+x_n^5=0.
$$
Likewise, by Theorem \ref{thm:indeterminacy}, $[X]$ is in the indeterminacy locus of $\overline {\mathcal P}$. For a general pencil $f:\X\to \P^1$ as above, the local monodromy at $X$ is finite. As above, there exists an exceptional divisor of $\pi_{CY}$ that generically maps to the interior of the period space.
\end{proof}

\begin{rmk}
The difference between the Baily--Borel and BBH compactifications is captured by the associated maps from the GIT moduli space. Theorem \ref{thm:indeterminacy} and Theorem \ref{thm:BBH indeterminacy} illustrate that the indeterminacy locus of the rational map to the Baily--Borel compactification is much larger than the indeterminacy locus of the rational map to the BBH compactification from the viewpoint of singularity theory.
\end{rmk}

\bibliographystyle{alpha}
\bibliography{refs}

\end{document}